\documentclass{amsart}

\usepackage{cancel}
\usepackage{soul}
\usepackage[normalem]{ulem}
\usepackage[all]{xy}

\usepackage{amssymb,amsmath,graphicx,mathabx,mathrsfs}

\usepackage{xcolor}
\usepackage{enumerate}

\newtoks\prt

\newtheorem{thm}{Theorem}[section]

\newtheorem{lemma}[thm]{Lemma}
\newtheorem{prop}[thm]{Proposition}
\newtheorem{cor}[thm]{Corollary}

\newtheorem{obs}[thm]{Observation}
\newtheorem{example}[thm]{Example}

\theoremstyle{definition}

\newtheorem{example2}[thm]{Example}
\newtheorem{remark}[thm]{Remark}

\newtheorem{remarks}[thm]{Remarks}

\def\eqn#1$$#2$${\begin{equation}\label#1#2\end{equation}}
 
\def\fra{\mathscr{A}}
\def\1{\boldsymbol{1}}
\def\Db{\boldsymbol{D}}
\def\Tb{\boldsymbol{T}}
\def\A{\mathcal A}

\def\D{\mathcal D}

\def\L{\mathcal L}

\def\es{\mathcal S}

\def\ce{\mathbb C}

\def\lin{Lindel\"of}

\def\co{\operatorname{conv}}
\def\ep{\varepsilon}
\def\K{\mathcal K}

\def\er{\mathbb R}

\def\ef{\mathbb F}

\def\TT{\mathbb T}

\def\Im{\operatorname{Im}}

\def \f{\boldsymbol{f}}
\def \g{\boldsymbol{g}}
\def \h{\boldsymbol{h}}

\def\ov{\overline}
\def \Ch {\operatorname{Ch}}

\def \ext {\operatorname{ext}}

\def\span{\operatorname{span}}
\def\id{\operatorname{id}}

\def \reg {\partial _{\kern1pt\text{reg}}}

\def\iff{\Longleftrightarrow}

\def\aco{\operatorname{aco}}

\def\ip#1#2{\left\langle#1,#2\right\rangle}

\def\di{\,\mbox{\rm d}}

\newcommand{\norm}[1]{\left\|#1\right\|}
\newcommand{\normr}[1]{\left\|#1\right\|_{rep}}

\renewcommand{\Re}{\operatorname{Re}}

\newcommand{\sign}{\operatorname{sign}}

\newcommand{\wscl}[1]{\overline{#1}^{w^*}}

\newcommand{\abs}[1]{\left|#1\right|}
\newcommand{\setsep}{;\,}

\numberwithin{equation}{section}

\definecolor{green}{rgb}{0,0.5,0}
\title {Simpliciality of vector-valued function spaces}
\author{Ondřej F.K. Kalenda and Ji\v r\'\i\ Spurn\'y}

\address{Ondřej F.K. Kalenda\\
Charles University\\
Faculty of Mathematics and Physics\\
Department of Mathematical Analysis \\
Sokolovsk\'{a} 83, 186 \ 75\\Praha 8, Czech Republic}
\email{kalenda@karlin.mff.cuni.cz}

\address{Ji\v r\'\i\ Spurn\'y\\
Charles University\\
Faculty of Mathematics and Physics\\
Department of Mathematical Analysis \\
Sokolovsk\'{a} 83, 186 \ 75\\Praha 8, Czech Republic}
\email{spurny@karlin.mff.cuni.cz}



\thanks{Our research was partially supported by the Research grant GA\v{C}R 23-04776S}

\begin{document}

\begin{abstract}
    We investigate integral representation of vector-valued function spaces, i.e., of subspaces  $H\subset C(K,E)$, where $K$ is a compact space and $E$ is a (real or complex) Banach space. We point out that there are two possible ways of generalizing representation theorems known from the scalar case -- either one may represent (all) functionals from $H^*$ using $E^*$-valued vector measures on $K$ (as it is done in the literature) or one may represent (some) operators from $L(H,E)$ by scalar measures on $K$ using the Bochner integral. These two ways lead to two different notions of simpliciality which we call `vector simpliciality' and `weak simpliciality'. It turns out that these two notions are in general incomparable. Moreover, the weak simpliciality is not affected by renorming the target space $E$, while vector simpliciality may be affected. Further, if $H$ contains constants, vector simpliciality is strictly stronger and admits several characterizations (partially analogous to the characterizations known in the scalar case). We also study orderings of measures inspired by C.J.K.~Batty which may be (in special cases) used to characterize $H$-boundary measures. Finally, we give a finer version of representation theorem using positive measures on $K\times B_{E^*}$ and characterize uniqueness in this case.
\end{abstract}

\maketitle

\tableofcontents

\section{Introduction}

Investigation of Choquet simplices and simplicial function spaces forms an important part of the Choquet theory focused on integral representation theorems. The most classical version works with continuous affine functions on compact convex set. If $X$ is a compact convex subset of a Hausdorff locally convex space, then any Radon probability measure $\mu\in M_1(X)$ admits a (unique) \emph{barycenter} $r(\mu)\in X$, i.e., a point satisfying $f(r(\mu))=\int f\di\mu$ for each $f$ 
from $\fra(X)$, the space of affine continuous functions on $X$. The Choquet-Bishop-de Leeuw theorem says that, given $x\in X$, there is some $\mu\in M_1(X)$ representing $x$ (i.e., such that $r(\mu)=x$) which is maximal in the Choquet ordering defined on the cone $M_+(X)$ of all positive Radon measures on $X$ by
$$\nu_1\prec\nu_2 \equiv^{df} \int f\di\nu_1\le\int f\di\nu_2\mbox{ for each $f:X\to\er$ continuous convex}.$$
If $X$ is metrizable, then maximal measures are exactly the measures carried by $\ext X$, the set of all extreme points of $X$.
A compact convex set $X$ is called a \emph{simplex} (or a \emph{Choquet simplex}) if for each $x\in X$ there is a unique maximal representing measure. The classical reference for integral representation of affine continuous functions on compact convex sets, including the theory of Choquet simplices is \cite{alfsen}.

A slight but very useful generalization is the theory of function spaces, i.e., subspaces of $C(K,\er)$ (for some compact space $K$) containing constants and separating points of $K$. This approach goes back to Choquet, the standard reference book is \cite{lmns}. This theory was extended to the complex case \cite{hustad71,hirsberg72,fuhr-phelps,phelps-complex}. It should be noted that in the study of compact convex sets it does not really matter whether we consider real or complex functions, whereas the theory of complex function spaces is more complicated than the real case. This is mainly due to the fact that a complex function space is not necessarily self-adjoint (i.e., closed under the complex conjugation).

In \cite{fuhr-phelps,phelps-complex} also function spaces not containing constant functions were addressed. In particular, a representation theorem was established there in this case. A thorough
study of this kind of function spaces, including possible notions of simpliciality, was performed  recently in \cite{bezkonstant}. In particular, in this case we have three levels of simpliciality and some pathologies may appear.

In the present paper we investigate vector-valued function spaces, related representation theorems and possible notions of simpliciality. This topic was initiated by P.~Saab \cite{saab-cr,saab-aeq}
for vector-valued affine functions on compact convex sets and continued in \cite{saab-canad,saab-tal} for vector-valued function spaces. In particular, in \cite{saab-tal} a general representation theorem was obtained. Later C.J.K.~Batty \cite{batty-vector} presented a new approach to the representation theorem for vector-valued function spaces containing constants using an ordering of vector measures.
Recently, in \cite{transference-studia}, we analyzed in more detail some methods of \cite{batty-vector} and, using the advanced method of disintegration we substantially strengthened some of the results. However, in \cite{transference-studia} we addressed only properties of the whole space of continuous vector-valued functions $C(K,E)$.

The present paper may be viewed as a continuation of \cite{transference-studia}, we now focus on general vector-valued function spaces.

The paper is organized as follows: In the next section we collect basic properties of vector-valued function spaces $H\subset C(K,E)$ and so introduce the setting of our investigation. In Section~\ref{sec:eval} we introduce and compare various natural evaluation mappings. These evaluation mappings are useful to transfer some notions from the framework of compact convex sets to our context; this played a key role already in \cite{bezkonstant}. In Section~\ref{sec:repmeas} we address two natural ways of extending the notion of representing measure to the vector-valued case. The first one follows the approach of \cite{saab-tal} and uses vector measures to represent functionals, the second one uses scalar measures to represent operators. In Section~\ref{sec:hranice a miry}  we show that there are natural notions of Choquet boundary and boundary measures, which coincide with those used in \cite{saab-tal}
and admit several characterizations. In Section~\ref{sec:reprez} we present two versions of representation theorems (one follows from \cite{saab-tal} and the second one follows from the scalar theory) and introduce four notions of simpliciality, which may be divided to two pairs. One pair is formed by weak simpliciality and functional weak simpliciality, the second one by vector simpliciality and functional vector simpliciality. In Section~\ref{sec:H-aff} we introduce two versions of $H$-affine functions. The first one is associated with the representation of operators and weak simpliciality, the related theory turns out to be completely analogous to the scalar case.
The second version is associated with the representation of functionals and vector simpliciality and appears to have a rather different behavior than the scalar case. In Section~\ref{sec:dilation} we introduce a vector-valued version of the dilation operator, establish its properties and apply it to the abstract Dirichlet problem in our setting. In Section~\ref{sec:skonstantami} we focus on function spaces containing constants and show that in this case the vector simpliciality is stronger than the weak simpliciality. In Section~\ref{sec:ordering} we introduce some orderings of measures and apply some results of \cite{batty-vector} to our context. In this way we characterize boundary measures as maximal ones (for function spaces containing constants). In Section~\ref{sec:nasoucinu} we analyze the orderings in more detail using methods from \cite{transference-studia}. In the last section we provide a brief overview of our results.

\section{Function spaces -- basic facts}\label{sec:fs}

Throughout the paper, $K$ will be a compact Hausdorff space and $E$ will be a (real or complex) Banach space. The field in question, i.e., $\er$ or $\ce$, will be denoted by $\ef$. Further, $C(K,E)$ is the space of $E$-valued continuous functions on $K$ equipped with the supremum norm. 
An important role is played by a distinguished operator $T:C(K,E)\to C(K\times B_{E^*},\ef)$ defined by
$$T\f(t,x^*)=x^*(\f(t)),\quad (t,x^*)\in K\times B_{E^*},\ \f\in C(K,E).$$
By a consequence of the Hahn-Banach theorem, $T$ is an isometric inclusion. We will also repeatedly use the projections $\pi_1:K\times B_{E^*}\to K$ and $\pi_2:K\times B_{E^*}\to B_{E^*}$.

A \emph{function space} will mean a linear subspace  $H\subset C(K,E)$. We will usually assume that $H$ separates points of $K$.
We say that $H$ \emph{contains constants} if it contains all constant functions $K\to E$. 
We say that $H$ \emph{contains some constants} if it contains at least one nonzero constant function $K\to E$. In the opposite case we say that $H$ \emph{contains no constants}.

Given a function space $H$, we consider three associated scalar-valued function spaces:
\begin{itemize}
    \item $H_w=\span\{x^*\circ\f\setsep \f\in H, x^*\in E^*\}\subset C(K,\ef)$;
    \item $H_l=T(H)=\{(t,x^*)\mapsto x^*(\f(t))\setsep \f\in H\}\subset C(K\times B_{E^*},\ef)$;
    \item $H_s=H_l+\{f\circ\pi_1\setsep f\in H_w\}\subset C(K\times B_{E^*},\ef).$
\end{itemize}

\begin{obs} Let $H\subset C(K,E)$ be a function space.
 \begin{enumerate}[$(a)$]
     \item $H$ separates points of $K$ if and only if $H_w$ separates points of $K$.
     \item $H_l$ never separates points of $K\times B_{E^*}$, unless $K$ is a singleton.
 \end{enumerate}   
\end{obs}

\begin{proof}
    $(a)$: The `if part' is obvious. The `only if part' is a consequence of the Hahn-Banach theorem.

    $(b)$: It is enough to observe that $f(t,0)=0$ for each $f\in H_l$ and $t\in K$. 
\end{proof}

The question which points of $K\times B_{E^*}$ are separated by $H_s$ is more complicated. To understand the situation properly we introduce the following notation. Given $s,t\in K$ we set
$$\begin{aligned}
  H(t)&=\{\f(t)\setsep  \f\in H\},\\
  H(s,t)&=\{(\f(s),\f(t))\setsep \f\in H\},\\
  H(s-t)&=\{\f(s)-\f(t)\setsep \f\in H\}.
\end{aligned}$$

Clearly $H(t)$ and $H(s-t)$ are linear subspaces of $E$ and $H(s,t)$ is a linear subspace of $E\times E$. If $H$ contains constants, then obviously $H(t)=E$ for each $t\in K$.

\begin{obs}
    Assume that $H$ separates points. Then
    $$\begin{aligned}
    \forall (s,x^*),(t,y^*)\in K\times B_{E^*}\colon
( \forall f&\in H_s\colon f(s,x^*)=f(t,y^*))\\&\iff 
s=t\mbox{ and }x^*-y^*\in H(t)^\perp).\end{aligned}$$
So, if $H(t)$ is dense in $E$ for each $t\in K$ (in particular, if $H$ contains constants), then $H_s$ separates points of $B_{E^*}\times K$.
\end{obs}

\begin{proof}
    Assume $s\ne t$. Then there is $f\in H_w$ with $f(s)\ne f(t)$. The function $(u,u^*)\mapsto f(u)$ then belongs to $H_s$ and separates $(s,x^*)$ and $(t,y^*)$.

    Next assume that $s=t$ and $x^*-y^*\notin H(t)^\perp$. This means that there is some $\f\in H$ with $(x^*-y^*)(\f(t))\ne 0$. Then $T\f\in H_l\subset H_s$ and this function separates $(s,x^*)$ and $(t,x^*)$.

    Conversely, assume that $s=t$ and $x^*-y^*\in H(t)^\perp$. Let $\f\in H$ be arbitrary. Then
    $$T\f(s,x^*)=x^*(\f(t))=y^*(\f(t))=T\f(t,y^*),$$
    so $f(s,x^*)=f(t,y^*)$ for each $f\in H_l$. Since clearly $g(s)=g(t)$ for each $g\in H_w$, we deduce that $f(s,x^*)=f(t,y^*)$ for $f\in H_s$.

    The additional statement is obvious.
\end{proof}

There are some relations between the sets $H(t)$, $H(s,t)$ and $H(s-t)$, but they are not straightforward. We collect some illustrative examples.

\begin{example2}
(1) Assume $K$ has at least two points and $\dim E\ge 2$. Fix a nonzero vector $x\in E$ and set
$$H=\{y+f\cdot x\setsep y\in E, f\in C(K,\ef)\}.$$
(The function $f\cdot x$ is defined as $f\cdot x(t)=f(t)x$, $t\in K$).
Then $H$ separates points and contains constants, hence $H(t)=E$ for $t\in K$. However,  for $s\ne t$ we have 
$$H(s-t)=\span\{x\} \mbox{ and }H(s,t)=\{(u,v)\in E\times E\setsep u-v\in\span\{x\}\}.$$

(2) Let $K=\TT$ be the unit circle, let $E$ be a nontrivial Banach space and let
$$H=\{\f\in C(\TT,E)\setsep \f(-t)=-\f(t)\mbox{ for }t\in \TT\}.$$
Then $H$ separates points but contains no constants. However for $s,t\in \TT$ we have
$$H(t)=H(s-t)=E\mbox{ and }H(s,t)=\begin{cases}
    \{(x,-x)\setsep x\in E\},& s=-t,\\ E\times E&\mbox{otherwise}.
\end{cases}$$
We note that the behavior is the same even if $\dim E=1$, i.e., in the scalar case.

(3) Let $K=\{0,1\}$, let $E$ be a nontrivial Banach space and let
$$H=\{\f\in C(K,E)\setsep \f(1)=-\f(0)\}.$$
Then $H$ separates points but contains no constants. However, 
$$H(0)=H(1)=H(1-0)=E\mbox{ and }H(0,1)=\{(x,-x)\setsep x\in E\}.
$$
We note that the behavior is again the same even if $\dim E=1$, i.e., in the scalar case.
\end{example2}

Further, if $H$ contains some constants, then both $H_w$ and $H_s$ contain constants. The space $H_l$ never contains constants (as any function from $H_l$ attains value $0$). It is also clear that $H_s$ contains constants if and only if $H_w$ contains constants.

It may happen that $H$ contains no constants, but $H_w$ contains constants. An easy example is provided by $K=\{0,1\}$ and $E=\ef^2$. Let $H$ consists of functions
$$ 0\mapsto (a,b), 1\mapsto (2a,-b),\quad a,b\in\ef.$$
It is clear that $H$ is a function space separating points of $K$ and containing no constants.  On the other hand, consider the function
$$\f: 0\mapsto (1,1), 1\mapsto (2,-1)$$
and $x^*\in E^*$ defined by $x^*(u,v)=2u+v$. Then $x^*\circ \f=3$, so it is a nonzero constant.

An example of a function space which contains some constants but does not contain constants is (taking $K$ and $E$ as above) the space consisting of functions
$$0\mapsto(a,b), 1\mapsto (a,-b), \quad a,b\in\ef.$$

A special case is $H=C(K,E)$. In this case $H_w=C(K,\ef)$, 
$$H_l=\{f\in C(K\times B_{E^*},\ef)\setsep f(t,\cdot)\mbox{ is affine and $\ef$-homogeneous for each }t\in K\}$$
and
$$H_s=\{f\in C(K\times B_{E^*},\ef)\setsep f(t,\cdot)-f(t,0)\mbox{ is affine and $\ef$-homogeneous for each }t\in K\}.$$

\section{Evaluation mappings}\label{sec:eval}

If $H\subset C(K,E)$ is a function space, we define the canonical evaluation mapping $\phi_H\colon K\to L(H,E)$ from $K$ to the space $L(H,E)$ of all bounded linear operators from $H$ to $E$  by
$$\phi_H(t)(\h)=\h(t),\quad \h\in H, t\in K.$$
If $E=\ef$, i.e., if $H$ is a scalar function space, then $\phi_H$ coincides with the classical evaluation mapping used in \cite{bezkonstant} or in \cite[Section 4.3]{lmns}.
It is clear that for each $t\in K$ we have $\phi_H(t)\in L(H,E)$ and $\norm{\phi_H(t)}\le 1$. Moreover, the mapping $\phi_H$ is continuous from $K$ to the strong operator topology. Indeed, if $t_\alpha\to t$ in $K$, and $\h\in H$, then
$$\phi_H(t_\alpha)(\h)=\h(t_\alpha)\to\h(t)=\phi_H(t)(\h).$$
Thus
$$\phi_H(t_\alpha)\overset{\mbox{\footnotesize SOT}}{\longrightarrow}\phi_H(t).$$
If $H$ separates points of $K$, then $\phi_H$ is one-to-one and so it is a homeomorphic injection of $K$ into $(B_{L(H,E)},SOT)$. We continue by a characterization of the closed absolutely convex hull of the range of the evaluation mapping. We write $M(K)=M(K,\ef)$ for the space of all Radon $\ef$-valued measures on $K$ considered as the dual space to $C(K,\ef)$. 

\begin{lemma}\label{L:kompaktnostphiH}
    Let $H\subset C(K,E)$ be a linear subspace. Then
    $$\ov{\aco}^{WOT}\phi_H(K)=\left\{\h\mapsto(B)\int_K\h\di\sigma\setsep \sigma\in B_{M(K)}\right\}$$
    and it is a compact subset of $B_{L(H,E)}$ in the weak operator topology. 
\end{lemma}

\begin{proof} Let $\sigma\in M(K)$. We define 
$$U_\sigma(\h)= (B)\int_K\h \di\sigma,\quad \h\in H.$$
Since $\h$ is continuous and $K$ is compact, the integral exists in the Bochner sense. Moreover, by the basic properties of the Bochner integral we deduce that $U_\sigma\in L(H,E)$ and $\norm{U_\sigma}\le\norm{\sigma}$.

Further, the mapping $\sigma\mapsto U_\sigma$ is clearly a linear operator of norm at most one. In addition, it is weak$^*$-to-WOT continuous. To see this, fix $\h\in H$ and $x^*\in E^*$. Then
$$ \sigma\mapsto x^*(U_\sigma\h)=x^*\left(\int \h\di\sigma\right)=\int x^*\circ\h\di\sigma  
$$
is weak$^*$-continuous, which completes the argument.

Next observe that $U_{\ep_t}=\phi_H(t)$ for $t\in K$ (note that $\ep_t$ denotes the Dirac measure carried by $t$) and that $U(B_{M(K)})$ is WOT-compact by the Banach-Alaoglu theorem and weak$^*$-to WOT continuity of the mapping $\sigma\mapsto U_\sigma$. Thus $U(B_{M(K)})$ is WOT-closed.
Hence we obtain
$$\ov{\aco}^{WOT}\phi_H(K)=\ov{\aco}^{WOT}\{U_{\ep_t}\setsep t\in K\}=\{U_\sigma\setsep \sigma\in B_{M(K)}\},$$
which proves the formula. 
\end{proof}

The above definition and discussion applies also to scalar-valued function spaces, i.e., to $E=\ef$. Then $L(H,E)=H^*$ and both the strong and weak operator topologies coincide with the weak$^*$-topology. In particular, we may consider mappings
$$\phi_{H_w}\colon K\to B_{H_w^*}, \quad \phi_{H_l}\colon K\times B_{E^*}\to B_{H_l^*},\quad \phi_{H_s}\colon K\times B_{E^*}\to B_{H_s^*}.$$

The following lemma summarizes properties of the norms of evaluation mappings.

\begin{lemma}\label{L:normyevaluaci} Let $t\in K$ and $x^*\in B_{E^*}$. Then:
\begin{enumerate}[$(a)$]
    \item $\norm{\phi_{H_l}(t,x^*)}\le\norm{x^*}\cdot\norm{\phi_H(t)}$ and $\norm{\phi_H(t)}=\sup\limits_{y^*\in S_{E^*}}\norm{\phi_{H_l}(t,y^*)}$. 
        If $H$ contains constants, then  $\norm{\phi_{H_l}(t,x^*)}=\norm{x^*}$. In particular, $\norm{\phi_{H}(t)}=1$.
    \item $\norm{\phi_H(t)}\le \norm{\phi_{H_w}(t)}=\norm{\phi_{H_s}(t,x^*)}\le1$. If $H$ contains some constants, the equalities hold. If $H_w$ contains constant, then the last inequality becomes equality.
\end{enumerate}
    
\end{lemma}

\begin{proof}
    $(a)$: Recall that $H_l=T(H)$ and $T$ is an isometry. Therefore, for any $\h\in H$ we have
    $$\abs{\phi_{H_l}(t,x^*)(\h)}=\abs{x^*(\h(t))}\le\norm{x^*}\cdot\norm{\h(t)}=\norm{x^*}\cdot\norm{\phi_H(t)(\h)}\le \norm{x^*}\cdot\norm{\phi_H(t)}\cdot\norm{\h},$$
    which proves the inequality.
    To prove the equality we compute
    $$\begin{aligned}
 \norm{\phi_H(t)}&=\sup_{\h\in B_{H}}\norm{\h(t)}=\sup_{\h\in B_{H}}\sup_{y^*\in S_{E^*}}\abs{y^*(\h(t))}
 =\sup_{y^*\in S_{E^*}}\sup_{\h\in B_{H}}\abs{y^*(\h(t))}\\&=\sup_{y^*\in S_{E^*}}\sup_{\h\in B_{H}}\abs{\phi_{H_l}(t,y^*)(\h)}=\sup_{y^*\in S_{E^*}}\norm{\phi_{H_l}(t,y^*)}.
\end{aligned}$$

Assume that $H$ contains constants. Then 
$$\norm{x^*}\ge\norm{\phi_{H_l}(t,x^*)}=\sup_{\h\in B_{H}}\abs{x^*(\h(t))}\ge \sup_{x\in B_E}\abs{x^*(x)}=\norm{x^*},$$
hence the equalities hold. Equality  $\norm{\phi_H(t)}=1$ now follows (alternatively, it is witnessed by any constant function of norm one).
   
  $(b)$: To prove the first inequality fix $\h\in H$. Then
  $$\begin{aligned}
        \norm{\phi_H(t)(\h)}&=\norm{\h(t)}=\sup_{y^*\in S_{E^*}}\abs{y^*(\h(t))}
  =\sup_{y^*\in S_{E^*}}\abs{\phi_{H_w(t)}(y^*\circ\h)}\\&\le 
  \sup_{y^*\in S_{E^*}}\norm{\phi_{H_w}(t)}\cdot\norm{y^*\circ\h} \le
    \norm{\phi_{H_w}(t)}\cdot\norm{\h},\end{aligned}$$
    which completes the argument.

    Let us pass to the equality. Inequality `$\le$' follows from the fact that $f\mapsto f\circ\pi_1$ is an isometric embedding of $H_w$ into $H_s$. To prove the converse inequality fix $f\in H_s$ with $\norm{f}\le1$. Then $f=g\circ\pi_1+T\h$ for some $g\in H_w$ and $\h\in H$. The assumption $\norm{f}\le 1$ in particular implies that
    $$\abs{g(t)+x^*(\h(t))}\le 1,\quad t\in K.$$
    But this means that $\norm{g+x^*\circ\h}\le1$. Since $g+x^*\circ\h\in H_w$, we deduce
    $$\abs{\phi_{H_s}(t,x^*)(f)}=\abs{g(t)+x^*(\h(t))}=\abs{\phi_{H_w}(t)(g+x^*\circ\h)}\le\norm{\phi_{H_w}(t)},$$
    which completes the argument.

    The last inequality is trivial.

    Assume that $H$ contains some constants. Then there is $x\in E$ with $\norm{x}=1$ such that the constant function equal to $x$ belongs to $H$. This function witnesses that $\norm{\phi_H(t)}=1$, hence equalities hold. If $H_w$ contains constants, then $H_w$ contains constant function equal to one, hence $\norm{\phi_{H_w}(t)}=1$. 
\end{proof}

We continue by collecting some examples showing that the inequalities in the previous lemma may be strict.
 
\begin{example2}\label{ex:normyevaluaci}
$(1)$ Let $K=\{0,1\}$, $E=(\ef^2,\norm{\cdot}_p)$ for some $p\in[1,\infty]$ and let
$$H=\{ 0\mapsto (0,b), 1\mapsto (a,b)\setsep a,b\in\ef\}.$$
Then $H$ contains constant function equal to $(0,1)$ and $\span(H(0)\cup H(1))=E$. Let $e_1^*$ be the first coordinate functional. Then $\norm{e_1^*}=1$ and and $\phi_{H_l}(0,e_1^*)=0$. Hence the inequality in Lemma~\ref{L:normyevaluaci}$(a)$ may be strict even if $H$ contains some constants.

$(2)$ Let $K$ and $E$ be as in $(1)$ and let 
$$H=\{0\mapsto(a,b),1\mapsto(\tfrac b2,\tfrac a2)\setsep a,b\in\ef\}.$$ Then $H$ contains no constants. Moreover,  clearly $\norm{\phi_H(1)} =\frac12$.
Let $\f\in H$ be defined using $a,b$. Then
$$e_1^*\circ \f: 0\mapsto a,1\mapsto \tfrac b2.$$
This covers all elements from $C(K)$. Thus $H_w=C(K)$ and hence $\norm{\phi_{H_w}(1)}=1$. Hence the first inequality in Lemma~\ref{L:normyevaluaci}$(b)$ may be strict (even if $H_w$ contains constants). 

$(3)$ The last inequality in Lemma~\ref{L:normyevaluaci}$(b)$ may be strict even for scalar-valued function spaces (then $H_w=H$). In particular, if
$$H=\{f\in C[0,1]\setsep f(1)=\tfrac12f(0)\},$$
then $\norm{\phi_H(1)}=\frac12$.
\end{example2}

For function spaces not containing constants one more evaluation mapping plays a key role.
In \cite{bezkonstant} it was denoted by $\theta$ (in \cite{fuhr-phelps} it was used without name,
in \cite{phelps-complex} it was denoted by $\Phi$). We are now going to define its vector-valued variant. It is the mapping $\theta_H:S_{\ef}\times K\to L(H,E)$ defined by $\theta_H(\alpha,t)=\alpha\phi_H(t)$. Then $\theta_H$ is continuous to the weak operator topology and, moreover,
$$\ov{\aco}^{WOT}\phi_H(K)=\ov{\co}^{WOT}\theta_H(S_{\ef}\times K).$$
In particular, by the Milman theorem we have
$$\ext \ov{\aco}^{WOT}\phi_H(K)\subset \theta_H(S_{\ef}\times K).$$
We continue by a lemma which enables us to relate $\theta_H$ to its scalar variants, but is formulated in a bit more general setting.

\begin{lemma}\label{L:phit=alphaphis}
    Let $s,t\in K$ and $\alpha\in\ef$. The following assertions are equivalent.
    \begin{enumerate}[$(a)$]
        \item $\phi_H(t)=\alpha\phi_H(s)$;
        \item $\phi_{H_w}(t)=\alpha\phi_{H_w}(s)$;
        \item $\phi_{H_s}(t,x^*)=\alpha\phi_{H_s}(s,x^*)$ for each $x^*\in B_{E^*}$;
        \item $\phi_{H_s}(t,x^*)=\alpha\phi_{H_s}(s,x^*)$ for some $x^*\in B_{E^*}$;
          \item $\phi_{H_l}(t,x^*)=\alpha\phi_{H_l}(s,x^*)$ for each $x^*\in B_{E^*}$.
        \end{enumerate}
\end{lemma}

\begin{proof}
Implications $(a)\implies(b)$ and $(a)\implies(c)\implies(d) \& (e)$ are trivial. 

$(b)\implies(a)$: Assume $(b)$ holds. Then for each $\h\in H$ and $x^*\in E^*$ we have $x^*\circ \h\in H_w$ and hence
$$x^*(\h(t))=\alpha x^*(\h(s))=x^*(\alpha\h(s)).$$
By a consequence of the Hahn-Banach theorem we deduce
$$\h(t)=\alpha\h(s),$$
which completes the argument.

$(d)\implies(b)$: Assume $(d)$ holds and $x^*\in B_{E^*}$ witnesses it. Let $f\in H_w$ be arbitrary. Then $f\circ\pi_1\in H_s$ and so
$$f(t)=(f\circ\pi_1)(t,x^*)=\alpha(f\circ\pi_1)(s,x^*)=\alpha f(s),$$
which completes the proof.  

$(e)\implies (a)$. Assume $(e)$ holds and $\h\in H$. For $x^*\in B_{E^*}$ we then have
$$x^*(\h(t))=T\h(t,x^*)=\alpha T\h(s,x^*)=\alpha x^*(\h(s))=x^*(\alpha\h(s)).$$
By a consequence of the Hahn-Banach theorem we deduce
$$\h(t)=\alpha\h(s),$$
which completes the argument.
\end{proof}

In the next lemma we apply the previous one to the mapping $\theta_H$ and its variants. This lemma enables us to reduce some vector-valued problems to the scalar ones (see Section~\ref{ss:AcwH} below). 

\begin{lemma}\label{L:theta}
   Let $s,t\in K$ and $\alpha,\beta\in S_\ef$. The following assertions are equivalent.
    \begin{enumerate}[$(a)$]
        \item $\theta_H(\alpha,t)=\theta_H(\beta,s)$;
        \item $\theta_{H_w}(\alpha,t)=\theta_{H_w}(\beta,s)$;
        \item $\theta_{H_s}(\alpha,t,x^*)=\theta_{H_s}(\beta,s,x^*)$ for each $x^*\in B_{E^*}$;
        \item $\theta_{H_s}(\alpha,t,x^*)=\theta_{H_s}(\beta,s,x^*)$ for some $x^*\in B_{E^*}$;
        \item $\theta_{H_l}(\alpha,t,x^*)=\theta_{H_l}(\beta,s,x^*)$ for each $x^*\in B_{E^*}$.
        \end{enumerate} 
   Consequently, $\theta_H$ is one-to-one if and only if $\theta_{H_w}$ is one-to-one.       
\end{lemma}

\begin{proof}
    The equivalences follow immediately from Lemma~\ref{L:phit=alphaphis}. The additional statement follows from equivalence $(a)\iff(b)$.
\end{proof}

\section{Representing measures of functionals and operators}\label{sec:repmeas}

In the theory of scalar function spaces an important role is played by representing measures. Recall that, given a scalar function space $H\subset C(K,\ef)$ and $\varphi\in H^*$, the set of representing measures is defined by
$$M_\varphi(H)=\left\{\mu\in M(K,\ef)\setsep \norm{\mu}=\norm{\varphi}\ \&\ \varphi(f)=\int f\di\mu\mbox{ for } f\in H\right\}.$$
This notation follows \cite{bezkonstant}, but the notion itself was used much earlier in \cite{hustad71,hirsberg72,fuhr-phelps,phelps-complex}.

The set $M_\varphi(H)$ is always nonempty. Indeed, given $\varphi\in H^*$, by the Hahn-Banach theorem it may be extended to $\widetilde{\varphi}\in C(K,\ef)^*$ with the same norm and, by the Riesz representation theorem $\widetilde{\varphi}$ is represented by a (unique) measure of the same norm.
An important special case, still in the scalar setting, is the set
$$M_t(H)=M_{\phi_H(t)}(H)=\left\{\mu\in M(K,\ef)\setsep \norm{\mu}=\norm{\phi_H(t)}\ \&\ f(t)=\int f\di\mu\mbox{ for }f\in H\right\}$$
defined for each $t\in K$. If $H$ additionally contains constants, then elements of $M_t(H)$ are probability measures which corresponds to the classical notion from \cite[Definition 3.3]{lmns}.

For vector-valued function spaces $H\subset C(K,E)$ there are two natural ways to generalize representing measures. We may investigate representation of functionals from $H^*$ by vector-valued measures or representations of operators from $L(H,E)$ using the Bochner integral. Both ways coincide in the scalar case. The first one was addressed in \cite{saab-tal}, the second one has not been explicitly studied yet (to our knowledge). We will briefly introduce both ways in two subsections.

\subsection{Representing functionals by vector measures}\label{sec:reprez-functionals}

The idea is to imitate the scalar setting. Let $H\subset C(K,E)$ be a function space. Given $\varphi\in H^*$, the Hahn-Banach theorem provides $\widetilde{\varphi}\in C(K,E)^*$ extending $\varphi$ and having the same norm. By Singer's theorem (see \cite{hensgen} for an easy proof) the dual space $C(K,E)^*$ is canonically isometric to $M(K,E^*)$, the space of all regular $E^*$-valued Borel measures on $K$ with bounded variation. Let us briefly recall the definition of this kind of integral.

If $\mu\in M(K,E^*)$ and $\f=\sum_{j=1}^n \chi_{A_j}\cdot x_j$ is an $E$-valued simple Borel function 
(i.e., $x_1,\dots,x_n\in E$ and $A_1,\dots,A_n$ are pairwise disjoint Borel subsets of $K$), then
$$\int \f\di\mu=\sum_{j=1}^n \mu(A_j)(x_j).$$
It is easy to check that $\f\mapsto \int\f\di\mu$ is a linear functional on the space of $E$-valued
simple Borel functions and $\abs{\int \f\di\mu}\le \norm{f}_\infty\norm{\mu}$, where $\norm{\mu}$ is the total variation of $\mu$. Hence it may be uniquely continuously extended to the space of $E$-valued functions which may be uniformly approximated by $E$-valued simple Borel functions. We denote this space by $I(K,E)$. It clearly contains $C(K,E)$. (Note that in case $E=\ef$ this integral coincides with the classical Lebesgue integral.)

Using this notion of integral we define the set of representing measures of $\varphi\in H^*$ by
$$M_\varphi(H)=\left\{\mu\in M(K,E^*)\setsep \norm{\mu}=\norm{\varphi}\ \&\ \varphi(\f)=\int\f\di\mu\mbox{ for }\f\in H\right\}.$$
As remarked above, it follows by combining the Hahn-Banach and Singer's theorems that the set $M_\varphi(H)$ is always nonempty. In the scalar case an important special case is the set $M_t(H)$ for $t\in K$. An analogous role in the new setting is played by the sets
$$M_{x^*\circ\phi_H(t)}(H),\quad t\in K, x^*\in B_{E^*}.$$
We note that $\norm{x^*\circ\phi_H(t)}=\norm{\phi_{H_l}(t,x^*)}$ for $t\in K$ and $x^*\in B_{E^*}$, so Lemma~\ref{L:normyevaluaci} may be applied.

We will further need some more notation. Given $\mu\in M(K,E^*)$ and $x\in E$, by $\mu_x$ we denote the scalar measure defined by
$$\mu_x(A)=\mu(A)(x),\quad A\subset K\mbox{ Borel}.$$
If $f:K\to\ef$ is a bounded Borel function, then $f\cdot x\in I(K,E)$ and it is easy to check that
$$\int f\cdot x\di\mu=\int f\di\mu_x.$$

\subsection{Representing measures of operators}\label{sec:reprez-oper}

Representing of operators appears to be quite different as, in general, not all operators admit such a representation. It is reflected by the following notion. Let $H\subset C(K,E)$ be a function space and let $U\in L(H,E)$. If there is a measure $\mu\in M(K)$ such that
$$U(\f)=(B)\int\f\di\mu\mbox{ for }\f\in H,$$
where the integral is in the Bochner sense, we call the operator $U$ \emph{representable}. If $U$ is a representable operator, we define
$$\normr{U}=\inf\left\{ \norm{\mu}\setsep \mu\in M(K), U(\f)=(B)\int\f\di\mu\mbox{ for }\f\in H\right\}.$$
Representable operators clearly form a vector space and $\normr{\cdot}$ is a norm on this space satisfying $\normr{\cdot}\ge\norm{\cdot}$.

\begin{obs}\label{obs:scalar}
    If $E=\ef$, then all bounded operators (=functionals) are representable and $\normr{\cdot}=\norm{\cdot}$.
\end{obs}

\begin{proof}
    If $E=\ef$, then $L(H,E)=E^*$. Hence the result follows by combining the Hahn-Banach and Riesz theorems.
\end{proof}

We continue by a lemma on the structure of representable operators.

\begin{lemma}\label{L:reprezentableoper}
Let $H\subset C(K,E)$ be a function space. Let $B=\overline{\aco}^{WOT}\phi_H(K)$. Then:
\begin{enumerate}[$(i)$]
    \item The space of representable operators equals $\span B$.
    \item $B$ is the unit ball of the space of representable operators equipped with the norm $\normr{\cdot}$.
    \item The infimum in the definition of $\normr{\cdot}$ is attained.
\end{enumerate}  
\end{lemma}

\begin{proof} Assertion $(ii)$ is just a reformulation of the formula from Lemma~\ref{L:kompaktnostphiH}, assertion $(i)$ then easily follows. To prove $(iii)$ observe that by $(ii)$ the norm $\normr{\cdot}$ is the Minkowski functional of $B$ and that $B$ is WOT-compact by Lemma~\ref{L:kompaktnostphiH}.
\end{proof}

We continue by observing that not all operators are representable.

\begin{example}\label{ex:nonrepresentable}
    There may be non-representable operators even if $K$ is a singleton and $\dim E=2$.
\end{example}

\begin{proof}
    Assume $K$ is a singleton and $\dim E=2$. Set $H=C(K,E)$. Then $\dim H=2$ and hence $\dim L(H,E)=4$. However, by Lemma~\ref{L:reprezentableoper} we deduce that the space of representable operators has dimension one.
\end{proof}

Given a representable operator $U\in L(H,E)$, we define the set of representing measures by
$$M_U(H)=\left\{\mu\in M(K)\setsep \norm{\mu}=\normr{U}\ \&\ U(\f)=(B)\int\f\di\mu\mbox{ for }\f\in H\right\}.$$
By Lemma~\ref{L:reprezentableoper}$(iii)$ this set is always nonempty. Similarly as in the scalar case we set
$$M_t(H)=M_{\phi_H(t)}(H)\mbox{ for }t\in K.$$

It turns out that the theory of representable operators reduces to the theory of scalar function spaces. The first step to this statement is the following proposition.

\begin{prop}\label{P:MtH=MtHw}
Let $H\subset C(K,E)$ be a function space and let $t\in K$. Then $M_t(H)=M_t(H_w)$ and $\normr{\phi_H(t)}=\norm{\phi_{H_w}(t)}$.
\end{prop}

\begin{proof}
    Let $\mu\in M(K)$. We observe that
$$ \f(t)=(B)\int\f\di\mu\mbox{ for }\f\in H \iff     f(t)=\int f\di\mu\mbox{ for }f\in H_w.$$ 
Indeed, implication `$\implies$' follows easily from the properties of the Bochner integral and definition of $H_w$. To prove the converse assume that the condition on the right-hand side is satisfied and fix $\f\in H$. Since $\f$ is continuous and $K$ is compact, the Bochner integral $(B)\int \f\di\mu$ exists. Moreover, for each $x^*\in E^*$ we have
$$x^*\left( (B)\int\f\di\mu \right)=\int x^*\circ\f\di\mu=x^*(\f(t))$$
as $x^*\circ\f\in H_w$. By a consequence of the Hanh-Banach theorem we complete the argument.

The equalities now follow from the scalar case.
\end{proof}

The following example is an immediate consequence of the previous proposition.

\begin{example2}\label{ex:mensinorma}
Example~\ref{ex:normyevaluaci}$(2)$ witnesses that for a representable operator $U$ we may have $\norm{U}<\normr{U}$.
\end{example2}

The next proposition provides the promised reduction of the theory of representable operators to the scalar case.

\begin{prop}\label{P:Upsilon}
 There is a unique mapping $\Upsilon: H_w^*\to L(H,E)$ satisfying the following properties.
 \begin{enumerate}[$(i)$]
     \item $\Upsilon$ is a linear operator, $\norm{\Upsilon}\le1$;
     \item $\Upsilon$ is weak$^*$-to-WOT continuous;
     \item $\Upsilon(\phi_{H_w}(t))=\phi_H(t)$ for $t\in K$;
     \item the range of $\Upsilon$ coincides with the space of representable operators;
     \item $\normr{\Upsilon(\varphi)}=\norm{\varphi}$ for $\varphi\in H_w^*$;
     \item $M_{\Upsilon(\varphi)}(H)=M_\varphi(H_w)$  for $\varphi\in H_w^*$.
 \end{enumerate}
\end{prop}

\begin{proof}
    The proof will be done in several steps:

\smallskip

 {\tt Step 1.} Construction of $\Upsilon^\prime:H_w^*\to L(H,E^{**})$.

\smallskip

Given $\varphi\in H_w^*$ we define a bilinear form $B_\varphi: H\times E^*\to\ef$ by
$$B_\varphi(\h,x^*)=\varphi(x^*\circ\h), \quad \h\in H, x^*\in E^*.$$
It is clear that $B_\varphi$ is a bilinear form and $\norm{B_\varphi}\le \norm{\varphi}$.
Thus there is a (unique) operator $\Upsilon^\prime(\varphi)\in L(H,E^{**})$ such that
$$\Upsilon^\prime(\varphi)(\h)(x^*)=B_\varphi(\h,x^*), \quad \h\in H, x^*\in E^*.$$
Clearly $\Upsilon^\prime$ is a linear operator $H_w^*\to L(H,E^{**})$ with $\norm{\Upsilon^\prime}\le1$.

\smallskip

{\tt Step 2.} $\Upsilon^\prime$ is weak$^*$-to-W$^*$OT continuous.

\smallskip

Indeed, fix $\h\in H$ and $x^*\in E^*$. Then
$$\varphi\mapsto \Upsilon^\prime(\varphi)(\h)(x^*)=\varphi(x^*\circ \h)$$
is weak$^*$-continuous, which completes the argument.

\smallskip

{\tt Step 3.} $\Upsilon^\prime(\phi_{H_w}(t))=\phi_H(t)$ for $t\in K$.

\smallskip

Fix $t\in K$. For $\h\in H$ and $x^*\in E^*$ we have
$$ \Upsilon^\prime(\phi_{H_w}(t))(\h)(x^*)=\phi_{H_w}(t)(x^*\circ \h)=x^*(\h(t))=x^*(\phi_H(t)(\h)),$$
hence the equality follows.

\smallskip

{\tt Step 4.} $\Upsilon^\prime(B_{H_w^*})=\overline{\aco}^{WOT}\phi_H(K)$.

\smallskip

Recall that $B_{H_w^*}=\wscl{\aco}\phi_{H_w}(K)$. Using Steps 1--3 and the weak$^*$-compactness of $B_{H_w^*}$
we deduce $\Upsilon^\prime(B_{H_w^*})=\overline{\aco}^{W^*OT}\phi_H(K)$. We conclude by Lemma~\ref{L:kompaktnostphiH}.

\smallskip

{\tt Step 5.} Conclusion.

\smallskip

By Step 4 we get that $\Upsilon^\prime(\varphi)\in L(H,E)$ for each $\varphi\in H_w^*$.  Thus we have $\Upsilon:H_w^*\to L(H,E)$. Steps 1--3 imply properties $(i)$--$(iii)$. Properties $(iv)$ and $(v)$ follow from Step 4 using Lemma~\ref{L:reprezentableoper}.

It remains to prove $(vi)$. Let $\mu\in M(K)$. For each $\h\in H$ the Bochner integral $(B)\int \h\di\mu$ exists, so we have
$$\begin{gathered}
 \forall\h\in H\colon (B)\int \h\di\mu =\Upsilon(\varphi)(\h) \\ \iff   \forall\h\in H\,\forall x^*\in E^*\colon x^*((B)\int \h\di\mu) =x^*(\Upsilon(\varphi)(\h))
 \\ \iff   \forall\h\in H\,\forall x^*\in E^*\colon \int x^*\circ\h\di\mu =\varphi(x^*\circ\h)
 \\ \iff   \forall f\in H_w\colon \int f\di\mu =\varphi(f),
 \end{gathered}$$
 where the second equivalence follows from the construction of $\Upsilon$. In view of $(v)$ the proof is complete.
\end{proof}

\begin{cor}\label{cor:identifikacekouli}
    $(B_{H_w^*},w^*)$ is affinely homeomorphic to $(\overline{\aco}^{WOT}\phi_H(K),WOT)$ via a mapping sending $\phi_{H_w}(t)$ to $\phi_H(t)$ for each $t\in K$.
\end{cor}

\section{Choquet boundaries and $H$-boundary measures}\label{sec:hranice a miry}

In the classical theory of function space Choquet boundary and boundary measures belong to the key notions. If $H\subset C(K,\ef)$ is a scalar function space, one of the possible definitions of Choquet boundary is
$$\Ch_HK=\{t\in K\setsep \phi_H(t)\in \ext B_{H^*}\}.$$
This definition is used in \cite[Section 7]{fuhr-phelps}, or recently in \cite{bezkonstant}. For function spaces containing constant and separating points it is usual to define
$$\Ch_HK=\{t\in K\setsep M_t(H)=\{\ep_t\}\},$$
see \cite[Definition 3.4]{lmns}. In this case the two definitions are equivalent by \cite[Proposition 4.26(d)]{lmns}. 

Further, still in the scalar case, a measure $\mu\in M(K,\ef)$ is called \emph{$H$-boundary} if $\phi_H(\abs{\mu})$ is a maximal measure on $B_{H^*}$. This corresponds to the setting of \cite[Section 7]{fuhr-phelps} and \cite{bezkonstant}. If $\ef=\er$ and $H$ contains constants and separates points, this is equivalent to $\abs{\mu}$ be $H$-maximal (see \cite[Definition 3.57 and Proposition 4.28]{lmns}). This description will be used in Section~\ref{sec:ordering} below.

We recall that, provided $K$ is metrizable, $H$-boundary measures are exactly measures carried by the Choquet boundary (see \cite[Corollary 3.62]{lmns} for the classical case and \cite[Observation 3.9]{bezkonstant} for the case of function spaces not containing constants). If $K$ is non-metrizable, the situation is more complicated. In the classical case $H$-boundary measures are carried by any Baire set containing the Choquet boundary (see \cite[Theorem 3.79]{lmns}), for function spaces not containing constants this may fail (see \cite[Examples 3.11, 3.12 and 6.17]{bezkonstant}). An effective characterization of $H$-boundary measures is provided by the Mokobodzki test:

Let $X$ be a compact convex set. For a bounded function $f:X\to\er$ we define its upper and lower envelopes by
$$\begin{aligned}
  f^*(x)&=\inf\{ h(x)\setsep h\in\fra(X,\er), h\ge f\}, \\
  f_*(x)&=\sup\{ h(x)\setsep h\in\fra(X,\er), h\le f\},
\end{aligned}\qquad x\in X.$$
Mokobodzki's criterion \cite[Proposition I.4.5]{alfsen} then asserts that $\mu\in M_+(X)$ is maximal if and only if $\int f\di\mu=\int f^*\di\mu$ for each $f\in C(X,\er)$. Moreover, it is enough to test it for convex continuous functions. If $X$ is the dual unit ball of a Banach space (equipped with the weak$^*$ topology), it is enough to test it for $\ef$-invariant continuous convex functions (by \cite[Lemma 4.1]{effros}).

In the present section we address Choquet boundaries and boundary measures for vector-valued function spaces. In the literature they are usually defined via $H_w$. We will show this is a natural choice and present some equivalent conditions. This will be done within two subsections -- the first one devoted to Choquet boundaries and the second one to boundary measures. 

\subsection{Choquet boundaries}\label{sec:choquet boundaries}

We start by formulating the promised equivalences.

\begin{thm}\label{T:hranice}
    Let $H\subset C(K,E)$ be a function space and let $t\in K$. Then the following assertions are equivalent.
    \begin{enumerate}[$(1)$]
        \item $\phi_H(t)\in \ext \ov{\aco}^{WOT} (\phi_H(K))$. 
        \item $\phi_{H_w}(t)\in \ext B_{H_w^*}$.
        \item $\phi_{H_s}(t,x^*)\in \ext B_{H_s^*}$ for some $x^*\in B_{E^*}$.
        \item $\phi_{H_s}(t,e^*)\in \ext B_{H_s^*}$ for some $e^*\in\ext B_{E^*}$.
    \end{enumerate}
    If, moreover, $H(t)$ is dense in $E$, these assertions are also equivalent to the following one:
    \begin{enumerate}[$(5)$]
        \item  $\phi_{H_s}(t,e^*)\in \ext B_{H_s^*}$ for each $e^*\in\ext B_{E^*}$.
    \end{enumerate}
\end{thm}

To prove the theorem we need some lemmata on the scalar case. The first one deals with measures representing extreme points.

\begin{lemma}\label{L:reprez-ext-bod}
    Let $X$ be a normed space and let $\nu\in M(B_{X^*})$ with $\norm{\nu}=1$. Define $x_0^*\in X^*$ by
$$x_0^*(x)=\int x^*(x)\di\nu(x^*), \quad x\in X.$$
 Assume that $x_0^*\in\ext B_{X^*}$. Then
$\nu$ is carried by $S_{\ef}x_0^*$ and the function
$$h(\alpha x_0^*)=\ov{\alpha},\quad \alpha\in S_{\ef}$$
is the Radon-Nikod\'ym density of $\nu$ with respect to $\abs{\nu}$.
\end{lemma}

\begin{proof}

First observe that  $\norm{x_0^*}=1$ and hence
$$\begin{aligned}
1&=\sup_{x\in B_X}\abs{x_0^*(x)}=\sup_{x\in B_X} \abs  {\int x^*(x)\di\nu(x^*)}
\le\sup_{x\in B_X} \int \abs{x^*(x)}\di\abs{\nu}(x^*)
\\&\le \int \norm{x^*}\di\abs{\nu}(x^*)\le \int 1\di\abs{\nu}(x^*)=1.
\end{aligned}$$
We deduce that equalities take place and hence $\norm{x^*}=1$ $\abs{\nu}$-a.e. I.e., $\nu$ is carried by the sphere.

Hence, to prove the first part it is enough to show that $\nu$ is carried by $\span\{x_0^*\}$. Assume not. It means that there is $x_1^*\in B_{X^*}\setminus \span\{x_0^*\}$ belonging to the support of $\nu$. By the Hahn-Banach theorem there is $x\in X$ such that $x_0^*(x)=0$ and $x_1^*(x)=1$. Set
$$A=\{x^*\in B_{X^*}\setsep \Re x^*(x)\ge \tfrac12\}.$$ 
Then $A$ is a weak$^*$-neighborhood of $x_1^*$ in $B_{X^*}$, so $\nu$ restricted to $A$ is nonzero.
Let $\nu_1=\nu|_A$ and $\nu_2=\nu|_{B_{X^*}\setminus A}$. Let 
$$y_j^*(y)=\int x^*(y)\di\nu_j(x^*), \quad y\in X, j=1,2.$$
Then $\norm{y_j^*}\le\norm{\nu_j}$ and $x_0^*=y_1^*+y_2^*$. It follows that 
$$1=\norm{x_0^*}\le\norm{y_1^*}+\norm{y_2^*}\le \norm{\nu_1}+\norm{\nu_2}=1,$$
hence equalities hold. In particular $\norm{y_j^*}=\norm{\nu_j}$, so $y_1^*\ne0$. Further, since $x_0^*$ is an extreme point, necessarily $y_1^*$ and $y_2^*$ lie on the segment $[0,x_0^*]$.
In particular, $x_0^*=\frac{y_1^*}{\norm{y_1^*}}=\frac{y_1^*}{\norm{\nu_1}}$. I.e.,
$$x_0^*(y)=\frac1{\abs{\nu}(A)}\int_A x^*(y)\di\nu(x^*),\quad y\in X.$$
It follows by the bipolar theorem that 
$$x_0^*\in\wscl{\aco} A=\wscl{\co} S_{\ef}\cdot A.$$
Since $x_0^*$ is an extreme point, it follows from Milman's theorem 
that $x_0^*\in\wscl {S_{\ef}\cdot A}=S_{\ef}\cdot A$. But this implies that $S_\ef\cdot x_0^*\cap A\ne\emptyset$, a contradiction. This completes the proof that $\nu$ is carried by $S_{\ef}x_0^*$.

By the Radon-Nikod\'ym theorem there is a Borel function $g:S_{\ef}\to S_{\ef}$ such that for each $B\subset S_{\ef}$ Borel we have
$$\nu(B\cdot x_0^*)=\int g(\alpha)\di\abs{\nu} (\alpha x_0^*).$$
Thus for each $y\in X$ we have
$$\begin{aligned}   
x_0^*(y)&=\int x^*(y)\di\nu(x^*)=\int_{S_{\ef}x_0^*} \alpha x_0^*(y) g(\alpha)\di\abs{\nu} (\alpha x_0^*) \\&= x_0^*(y)\int_{S_{\ef}x_0^*} \alpha g(\alpha) \di\abs{\nu} (\alpha x_0^*),\end{aligned} $$
so
$$\int_{S_{\ef}x_0^*} \alpha g(\alpha) \di\abs{\nu} (\alpha x_0^*)=1.$$
It follows that $\alpha g(\alpha)=1$ $\abs{\nu}$-a.e., so $g(\alpha)=\overline{\alpha}$ $\abs{\nu}$-a.e., which completes the proof.
\end{proof}

We continue by an easy measure-theoretic lemma.

\begin{lemma}\label{L:podmsh}
    Let $\Omega$ be a set, $\Sigma^\prime\subset\Sigma$ two $\sigma$-algebras on $\Omega$ and $\mu$ a complete complex measure on $(\Omega,\Sigma)$. Let $h$ be the Radon-Nikod\'ym density of $\mu$ with respect to $\abs{\mu}$. If $\norm{\mu_{|\Sigma^\prime}}=\norm{\mu}$, then $h$ is measurable with respect to (the completion of) $\mu_{|\Sigma^\prime}$.
\end{lemma}

\begin{proof}
    Let $g$ be the Radon-Nikod\'ym density of $\mu|_{\Sigma^\prime}$ with respect to $\abs{\mu_{|\Sigma^\prime}}$. Since $g$ is measurable with respect to $\mu_{|\Sigma^\prime}$, it is enough to show that $h=g$ $\abs{\mu}$-almost everywhere.
    
    Observe that equality $\norm{\mu_{|\Sigma^\prime}}=\norm{\mu}$ implies $\abs{\mu_{|\Sigma^\prime}}=\abs{\mu}_{|\Sigma^\prime}$. Using properties of Radon-Nikod\'ym densities we get, for any $A\in \Sigma^\prime$,
    $$\int_A h\di\abs{\mu}=\mu(A)=\int_A g\di  \abs{\mu_{|\Sigma^\prime}} =\int_A g\di\abs{\mu}_{|\Sigma^\prime}=\int_A g\di\abs{\mu}.$$
   By standard measure-theoretic arguments we deduce that
   $$\int fh\di\abs{\mu}=\int fg\di\abs{\mu} \mbox{ for each }f\in L^1(\abs{\mu}_{|\Sigma^\prime}).$$
   Hence
   $$\begin{aligned}
       \int \abs{g-h}^2\di\abs{\mu}&=\int (g-h)(\overline{g}-\overline{h})\di\abs{\mu}
       \\&=\int \abs{g}^2\di\abs{\mu} - \int g\overline{h}\di\abs{\mu}-\int \overline{g}h\di\abs{\mu} + \int \abs{h}^2\di\abs{\mu}
        \\&=\int \abs{g}^2\di\abs{\mu} - \int g\overline{g}\di\abs{\mu}-\int \overline{g}g\di\abs{\mu} + \int \abs{h}^2\di\abs{\mu}
        \\&= -\int \abs{g}^2\di\abs{\mu}+\int \abs{h}^2\di\abs{\mu} = -\norm{\mu_{|\Sigma^\prime}}+\norm{\mu}=0,
   \end{aligned}$$
   where  the third equality follows from the fact that $g$ and $\overline{g}$ belong to $L^1(\abs{\mu}_{|\Sigma^\prime})$ and in the fifth equality we used that $\abs{g}=\abs{h}=1$ $\abs{\mu}$-almost everywhere.

   The above computation completes the proof.
\end{proof}

The previous lemma has the following consequence on factorization of densities.

\begin{lemma}
    \label{l:factor-density}
    Let $\pi\colon K\to L$ be a continuous surjection of a Hausdorff compact space $K$ onto a Hausdorff compact space $L$. Let $\mu\in M(K)$ be such that $\norm{\pi(\mu)}=\norm{\mu}$ and let $h\colon L\to S_\ef$ be a Borel density of $\pi(\mu)$ with respect to $\abs{\pi(\mu)}$. Then $h\circ\pi$ is a density of $\mu$ with respect to $\abs{\mu}$.
\end{lemma}

\begin{proof}
 We use Lemma~\ref{L:podmsh}. In order to do this, we set $\Omega=K$, $\Sigma$ be the
 family of all $\mu$-measurable sets in $K$ and $\mu$ is our given measure. Let \[
 \Sigma^\prime=\{\pi^{-1}(B)\setsep B\subset L\ \pi(\mu)\mbox{-measurable}\}.
 \]
 Since $\norm{\pi(\mu)}=\norm{\mu}$, we obtain $\norm{\mu_{|\Sigma'}}=\norm{\pi(\mu)}=\norm{\mu}$.
If $g\colon K\to S_\ef$ is a Borel density of $\mu$ with respect $\abs{\mu}$, from Lemma~\ref{L:podmsh}  we infer that $g$ is measurable with respect to the completion of $\mu_{|\Sigma'}$. The function $h\circ \pi$ is a density of $\mu|_{\Sigma'}$ with respect to $\abs{\mu_{|\Sigma'}}$, and thus for each $A\in \Sigma'$ we have
\[
\int_A h\circ\pi\di\mu=\abs{\mu}(A)=\int_A g\di\mu.
\]
It follows that $h\circ \pi=g$ almost everywhere with respect to the completion of $\mu|_{\Sigma'}$. Thus there exists a Borel set $B\subset L$ with 
\[
0=\abs{\pi(\mu)}(B)=\pi(\abs{\mu})(B)=\abs{\mu}(\pi^{-1}(B))
\]
such that $g=h\circ \pi$ on $K\setminus \pi^{-1}(B)$. By the previous equation, $\abs{\mu}(\pi^{-1}(B))=0$, and thus $h\circ \pi=g$ $\abs{\mu}$-almost everywhere. Hence $h\circ \pi$ is a density of $\mu$ with respect to $\abs{\mu}$.
\end{proof}

As mentioned in the introduction of this section, for function spaces containing constants and separating points the Choquet boundary consists exactly of points $t$ satisfying $M_t(H)=\{\ep_t\}$. For function spaces not containing constants this is not true by \cite[Example 3.4(2)]{bezkonstant}. However, even in this case the set of representing measures has a special form. It is the content of the following lemma.

\begin{lemma}\label{L:reprez-t-chb}
    Let $H\subset C(K)$ be a (scalar) function space and Let $t\in \Ch_HK$. Let
    $$K_t=\{s\in K\setsep \phi_H(s)=\alpha\phi_H(t) \mbox{ for some }\alpha\in S_{\ef}\}.$$
    Then the function $g\colon K_t\to S_{\ef}$ given by
    $$g(s)=\alpha\mbox{ if }\phi_H(s)=\alpha\phi_H(t)$$
    is well defined and continuous and, moreover,
    $$M_t(H)=\{\overline{g}\sigma\setsep \sigma\in M_1(K_t)\}.$$
\end{lemma}

\begin{proof}
    The mapping $g_1:S_{\ef}\to B_{H^*}$ defined by $g_1(\alpha)=\ov{\alpha}\phi_H(t)$ for $\alpha\in S_{\ef}$ is clearly a homeomorphic injection. Let $K_t^0=\phi_H(K)\cap g_1(S_\ef)$. It is clearly a weak$^*$ closed subset of $B_{H^*}$.  It is then easy to check that
    $$K_t=\phi_H^{-1}(K_t^0) \mbox{ and } g= \ov{g_1^{-1} \circ \phi_H|_{K_t}}.$$
    This shows that $g$ is well-defined and continuous.

    Let us proceed by proving the equality. Inclusion `$\supset$' is obvious. Let us prove the converse one. Fix $\nu\in M_t(H)$. Then $\norm{\nu}=\norm{\phi_H(t)}=1$ (as $\phi_H(t)\in\ext B_{H^*}$) and
    for each $h\in H$ we have
    $$\phi_H(t)(h)=h(t)=\int_K h\di\nu=\int_K \phi_H(s)(h)\di\nu(s)=\int_{B_{H^*}} \varphi(h)\di\phi_H(\nu)(\varphi),$$
    so Lemma~\ref{L:reprez-ext-bod} may be applied to $\phi_H(\nu)$ and $\phi_H(t)$. Thus $\phi_H(\nu)$ is carried by $S_\ef\phi_H(t)\cap\phi_H(K)=K_t^0$. Since $\norm{\phi_H(\nu)}=\norm{\nu}=1$, necessarily $\abs{\phi_H(\nu)}=\phi_H(\abs{\nu})$, hence $\nu$ is carried by $K_t$.

    Now we may apply Lemma~\ref{l:factor-density} along with Lemma~\ref{L:reprez-ext-bod} to the mapping $\phi_{H}\colon K_t\to K_t^0$. The second lemma asserts that the function $h\colon g_1(S_\ef)\to S_{\ef}$ given  by $\overline{\alpha}\phi_H(t)\mapsto \alpha$ is a density of $\phi_H(\nu)$ with respect to $\abs{\phi_{H}(\nu)}$. By Lemma~\ref{l:factor-density}, $\overline{g}=h\circ\phi_H$ is a density of $\nu$ with respect to $\abs{\nu}$.
    \end{proof}

Now we are ready to prove the above theorem.

\begin{proof}[Proof of Theorem~\ref{T:hranice}.]
$(1)\iff(2)$: This follows immediately from Corollary~\ref{cor:identifikacekouli}.

$(2)\implies(3)$: Assume $(2)$ holds. The operator $V:f\mapsto f\circ \pi_1$ is a linear isometric inclusion of $H_w$ into $H_s$. Therefore the dual operator $V^*$ is a quotient map and its restriction to $B_{H_s^*}$ is a continuous affine surjection onto $B_{H_w^*}$. In particular,  $F=(V^*)^{-1}(\phi_{H_w}(t))\cap B_{H_s^*}$ is a closed face of $B_{H_s^*}$.

For $\sigma\in M_1(B_{E^*})$ define
$$\Psi(\sigma)(f)=\int f(t,x^*)\di\sigma(x^*),\quad f\in H_s.$$
Clearly $\Psi(\sigma)\in B_{H_s^*}$. Further, for $g\in H_w$ we have
$$V^*\Psi(\sigma)(g)=\Psi(\sigma)(g\circ\pi_1)=\int g(t)\di\sigma(x^*)=g(t),$$
so $\Psi(\sigma)\in F$. Thus $\Psi:M_1(B_{E^*})\to F$ is an affine continuous mapping.

Further, $\Psi$ is a surjection. To see this fix $\psi\in F$ and let $\nu\in M_\psi(H_s)$. Then $\pi_1(\nu)\in M_t(H_w)$. Indeed, assume $f\in H_w$. Then
$$\begin{aligned}
\int f\di\pi_1(\nu)&=\int f\circ\pi_1\di\nu=\psi(f\circ\pi_1)=\psi(V f)=V^*\psi (f)=\phi_{H_w}(t)(f)=f(t)
\end{aligned}$$
Moreover,
$$1=\norm{\phi_{H_w}(t)}\le\norm{\pi_1(\nu)}\le\norm{\nu}=\norm{\psi}=1.$$
Therefore Lemma~\ref{L:reprez-t-chb} may be applied to $\pi_1(\nu)$ (then $\sigma=\abs{\pi_1(\nu)}=\pi_1(\abs{\nu})$, let $g$ be the function provided by the quoted lemma).

It follows that $\nu$ is carried by $K_t\times B_{E^*}$ and that the density of $\nu$ with respect to $\abs{\nu}$ is the function $(s,x^*)\mapsto \ov{g(s)}$.
(Indeed, $s\mapsto\overline{g(s)}$ is a density of $\pi_1(\nu)$ with respect to $\abs{\pi_1(\nu)}$, and hence Lemma~\ref{l:factor-density} yields that the function $(s,x^*)\mapsto \overline{g(s)}$ is a density of $\nu$ with respect to $\abs{\nu}$.)

For $\h\in H$ we have
$$\begin{aligned}
 \psi(T\h)&=\int T\h\di\nu=\int_{K_t\times B_{E^*}}    x^*(\h(s))\overline{g(s)}\di\abs{\nu}(s,x^*)
 \\&=\int_{K_t\times B_{E^*}} x^*(g(s)\h(t))\overline{g(s)}\di\abs{\nu}(s,x^*)
 =\int_{K_t\times B_{E^*}} x^*(\h(t))\di\abs{\nu}(s,x^*)
 \\&=\int_{B_{E^*}}  x^*(\h(t))\di\pi_2(\abs{\nu})(x^*).
\end{aligned}$$
Further, if $f\in H_w$, then
$$\psi(f\circ \pi_1)=\int_{K\times B_{E^*}} f\circ\pi_1\di\nu=\int_K f\di\pi_1(\nu)=f(t)=\int_{B_{E^*}}f(t)\di\pi_2(\abs{\nu})(x^*).$$
It follows that $\psi=\Psi(\pi_2(\abs{\nu}))$, which completes the proof that $\Psi$ is a surjection.

Let $\psi\in\ext F$. Then $\Psi^{-1}(\psi)$ is a closed face of $M_1(B_{E^*})$, so it contains an extreme point, i.e. a Dirac measure $\ep_{x^*}$ for some $x^*\in B_{E^*}$. Then
$$\psi(f)=\Psi(\ep_{x^*})(f)=f(t,x^*),\quad f\in H_s.$$
I.e., $\psi=\phi_{H_s}{(t,x^*)}$. Thus $\phi_{H_s}(t,x^*)\in\ext F\subset \ext B_{H_s^*}$. This completes the proof of $(3)$.

$(3)\implies (4)$:  Let $x^*\in B_{E^*}$ be provided by $(3)$. Set
$$F=\{z^*\in B_{E^*}\setsep \phi_{H_s}(t,z^*)=\phi_{H_s}(t,x^*)\}.$$
This is a closed face of $B_{E^*}$. 
Indeed, assume that $z^*\in F$ and $z^*=\frac12(y_1^*+y_2^*)$ for some $y_1^*,y_2^*\in B_{E^*}$. Since each element of $H_s$ is affine in the second variable, we deduce that 
$$\phi_{H_s}(t,x^*)=\phi_{H_s}(t,z^*)=\tfrac12(\phi_{H_s}(t,y_1^*)+\phi_{H_s}(t,y_2^*)).$$
By the assumption we deduce that $\phi_{H_s}(t,y_1^*)=\phi_{H_s}(t,y_2^*)=\phi_{H_s}(t,x^*)$, so $y_1^*,y_2^*\in F$. 

The set $F$, being a closed face, contains some $e^*\in\ext B_{E^*}$. Then $$\phi_{H_s}(t,e^*)=\phi_{H_s}(t,x^*)=\psi\in\ext F\subset\ext B_{H_s^*}.$$
This completes the proof of $(4)$.

$(4)\implies (2)$: Assume $(4)$ holds and let $e^*\in\ext B_{E^*}$ witness it. Assume $\phi_{H_w}(t)=\frac12(\psi_1+\psi_2)$ where 
$\psi_1,\psi_2\in B_{H_w^*}$. We define $\widetilde{\psi}_1,\widetilde{\psi}_2\in H_s^*$ as follows:

If $f\in H_w$, $\h\in H$ and $j\in\{1,2\}$, we set
$$\widetilde{\psi}_j(f\circ\pi_1+T\h)=\psi_j(f+e^*\circ\h).$$
Then $\widetilde{\psi}_j$ is a well-defined linear functional. Indeed, assume $f\circ\pi_1+T\h=0$. Then for each $s\in K$ we have
$$f(s)=f\circ\pi_1(s,0)=f\circ\pi_1(s,0)+T\h(s,0)=0,$$
so $f=0$. Thus $T\h=0$ and therefore $\h=0$. The linearity is clear.

Moreover, assume $\norm{f\circ\pi_1+T\h}\le1$. This means that
$$\forall s\in K\;\forall y^*\in B_{E^*}\colon \abs{f(s)+y^*(\h(s))}\le 1.$$
In particular, applying to $y^*=e^*$ we deduce that $\norm{f+e^*\circ\h}\le 1$. Therefore
$$\abs{\widetilde{\psi}_j(f\circ\pi_1+T\h)}=\abs{\psi_j(f+e^*\circ\h)}\le \norm{\psi_j}\le1.$$
Moreover,
$$\tfrac12(\widetilde{\psi}_1+\widetilde{\psi}_2)(f\circ\pi_1+T\h)=
\tfrac12(\psi_1+\psi_2)(f+e^*\circ \h)=f(t)+e^*(\h(t)),$$
so 
$$\tfrac12(\widetilde{\psi}_1+\widetilde{\psi}_2)=\phi_{H_s}(t,e^*).$$
By the assumption we get
$$\widetilde{\psi}_1=\widetilde{\psi}_2=\phi_{H_s}(t,e^*).$$
In particular, for each $f\in H_w$ we have
$$\psi_j(f)=\widetilde{\psi}_j(f\circ\pi_1)=(f\circ\pi_1)(t,e^*)=f(t),$$
thus $\psi_1=\psi_2=\phi_{H_w}(t)$. This shows that $\phi_{H_w}(t)\in \ext B_{H_w^*}$ and completes the proof.

Implication $(5)\implies (4)$ is trivial.

It remains to prove $(2)\implies(5)$ assuming $H(t)$ is dense in $E$:  Let $V$, $F$ and $\Psi$ be as in the proof of $(2)\implies(3)$. Fix an arbitrary $e^*\in \ext B_{E^*}$ and let $\psi_1,\psi_2\in B_{H_s^*}$ such that $\phi_{H_s}(t,e^*)=\frac12(\psi_1+\psi_2)$. Since $\phi_{H_s}(t,e^*)\in F$ and $F$ is a face, we deduce that $\psi_1,\psi_2\in F$. Fix $j\in\{1,2\}$. Let $\sigma_j\in M_1(B_{E^*})$ be such that $\Psi(\sigma_j)=\psi_j$. This implies that for each $\h\in H$ we have
$$\begin{aligned}   
e^*(\h(t))&=\phi_{H_s}(t,e^*)(T\h)=\tfrac12(\psi_1(T\h)+\psi_2(T\h))=\int T\h(t,x^*)\di\tfrac{\sigma_1+\sigma_2}2(x^*)\\
&=\int x^*(\h(t))\di\tfrac{\sigma_1+\sigma_2}2(x^*).\end{aligned}$$
In other words,
$$e^*(x)=\int x^*(x)\di\tfrac{\sigma_1+\sigma_2}2(x^*) \mbox{ for each }x\in H(t).$$
Since $H(t)$ is dense in $E$, we deduce that
$$e^*(x)=\int x^*(x)\di\tfrac{\sigma_1+\sigma_2}2(x^*) \mbox{ for each }x\in E.$$

Since $\frac{\sigma_1+\sigma_2}2$ is a probability measure, the last equality means that $e^*$ is its barycenter. But $e^*$ is an extreme point, so $\frac{\sigma_1+\sigma_2}2=\ep_{e^*}$, hence $\sigma_1=\sigma_2=\ep_{e^*}$.
It now easily follows that $\psi_1=\psi_2=\phi_{H_s}(t,e^*)$, which completes the argument.
\end{proof}

We define the \emph{Choquet boundary of $H$} (denoted by $\Ch_HK$) to be the set of all $t\in K$ satisfying the equivalent conditions of Theorem~\ref{T:hranice}.

We continue by some examples showing optimality of Theorem~\ref{T:hranice}

\begin{example}\label{ex:chH-protipr}
 Let $K=\{0,1\}$, $p\in[1,\infty]$, $E=(\ef^2,\norm{\cdot}_p)$ and 
$$H=\{\f=(f_1,f_2)\in C(K,E)\setsep f_1(0)=0\}.$$
Then $H$ is a function space separating points and containing some constants. Moreover, the following
assertions are valid:
\begin{enumerate}[$(i)$]
    \item $H_w=C(K)$.
    \item $\Ch_HK=K$.
    \item If $p=\infty$, then $\phi_H(0)\notin \ext B_{L(H,E)}$.
    \item If $p<\infty$, then both $\phi_{H}(0)$ and $\phi_{H}(1)$ belong to $\ext B_{L(H,E)}$.
    \item If $p\in(1,\infty]$, then $x^*=(1,0)\in\ext B_{E^*}$ but $\phi_{H_s}(0,x^*)=\phi_{H_s}(0,0)\notin \ext B_{H_s^*}$.
\end{enumerate}
\end{example}

\begin{proof}
    Note that $E^*=(\ef^2,\norm{\cdot}_q)$, where $q$ is the dual exponent, via the standard duality. Denote by $e_1^*$ and $e_2^*$ the coordinate functionals. It is clear that $H$ separates points and that it contains constant function $(0,1)$.

    $(i)$: Let $f\in C(K)$ be arbitrary. Define $\f\in C(K,E)$ by setting $f_1=0$, $f_2=f$. Then $\f\in H$ and $e_2^*\circ\f=f$. So, $f\in H_w$. 

    $(ii)$: This follows easily from $(i)$.

    $(iii)$: Assume $p=\infty$. Then operators defined by
      $$U_1(\f)=(f_1(1),f_2(0)) \mbox{ and }U_2(\f)=(-f_1(1),f_2(0)), \qquad \f\in H,$$
    have norm one, satisfy $U_1\ne U_2$ and $\phi_{H}(0)=\frac12(U_1+U_2)$.

    $(iv)$: Assume $p<\infty$.  Assume $U_1,U_2\in B_{L(H,E)}$ with $\phi_H(0)=\frac12(U_1+U_2)$.
     If we plug there the function $\f:0\mapsto (0,1), 1\mapsto (0,0)$,
    we deduce that
    $$(0,1)=\tfrac12(U_1(\f)+U_2(\f)).$$
     Since $(0,1)\in \ext B_E$, we get $U_1(\f)=U_2(\f)=(0,1)$.
     
    Further, let $\g: 0\mapsto (0,0), 1\mapsto (1,0)$. Set $(\alpha,\beta)=U_1(\g)$. Since $\norm{\f\pm\g}=1$, we deduce
    that $\norm{(\pm\alpha,1\pm\beta)}_p\le 1$. But this implies $\alpha=\beta=0$, i.e. $U_1(\g)=0$. Similarly $U_2(\g)=0$. Analogously we show that $U_1(\h)=U_2(\h)=0$ where $\h:0\mapsto (0,0), 1\mapsto (0,1)$. It follows $U_1=U_2=\phi_H(0)$. Hence $\phi_H(0)\in\ext B_{L(H,E)}$. 

    The proof that $\phi_H(1)\in B_{L(H,E)}$ is similar.

  $(v)$: If $p\in(1,\infty)$, then $E^*$ is strictly convex and hence any norm-one element is an extreme point. If $p=\infty$, then $E^*=(\ef^2,\norm{\cdot}_1)$ and hence $(1,0)$ is an extreme point. Moreover, if $f\in H_w$, then
  $$(f\circ\pi_1)(0,e_1^*)=f(0)=(f\circ\pi_1)(0,0)$$
  and for $\h\in H$ we have
  $$T\h (0,e_1^*)=e_1^*(\h(0))=h_1(0)=0=T\h(0,0).$$
 Hence indeed $\phi_{H_s}(0,x^*)=\phi_{H_s}(0,0)$.

 Finally, $\phi_{H_s}(0,0)\notin \ext B_{H_s^*}$ as
 $$\phi_{H_s}(0,0)=\tfrac12(\phi_{H_s}(0,e_2^*)+\phi_{H_s}(0,-e_2^*))$$
 and, for $\f\in H$ we have  $T\f(0,e_2^*)=f_2(0)$  which may be nonzero.
\end{proof}

The previous example shows that:
\begin{itemize}
    \item In condition $(1)$ of Theorem~\ref{T:hranice} we cannot require $\phi_H(t)\in\ext B_{L(H,E)}$.
    \item Implication $(1)\implies (5)$ of Theorem~\ref{T:hranice} fails without the additional assumptions; even if $(1)$ is strengthened to assume $\phi_H(t)\in\ext B_{L(H,E)}$.
\end{itemize} 

The following examples witnesses that for $t\in\Ch_HK$ we may have $\norm{\phi_H(t)}<1$.

\begin{example}
    Let $K$, $p$ and $E$ be as in Example~\ref{ex:chH-protipr}. Set
    $$H=\{\f=(f_1,f_2)\in C(K)\setsep \f(1)=\tfrac12(f_2(0),f_1(0))\}.$$
Then $H$ is a function space separating points and containing no constants. Moreover, the following
assertions are valid:
\begin{enumerate}[$(i)$]
 \item $H_w=C(K)$.
    \item $\Ch_HK=K$.
    \item $\norm{\phi_H(1)}=\frac12$.
\end{enumerate}
\end{example}

\begin{proof}
    We use the notation from the proof of Example~\ref{ex:chH-protipr}.
    
    $(i)$: Let $f\in C(K)$ be arbitrary. Define $\f\in C(K,E)$ by
    $$f_1=f, f_2(0)=2 f(1), f_2(1)=\tfrac12 f(0).$$
    Then $\f\in H$ and $e_1^*\circ \f=f$. So, $f\in H_w$.

    Assertion $(ii)$ follows easily from $(i)$.
    
   $(iii)$: Let $\f\in H$. Then
   $$\norm{\phi_H(1)(\f)}=\norm{\f(1)}=\tfrac12\norm{(f_2(0),f_1(0))}=\tfrac12\norm{\f(0)}\le\tfrac12\norm{\f}.$$
    Thus $\norm{\phi_H(1)}\le\frac12$. Any nonzero $\f\in H$ witnesses that the equality holds.
\end{proof}

\subsection{$H$-boundary measures}\label{sec:boundary measures}

We now focus on boundary measures. To define them we will use the characterizations below. To formulate them in optimal way we will need some auxiliary results proved using method of disintegration. We start by briefly recalling this method.

\begin{lemma}\label{L:dezintegrace-obec}  Let $K$ and $B$ be compact Hausdorff spaces and let $\nu\in M_+(K\times B)$.  Then there is a family $(\nu_t)_{t\in K}$ of Radon probability measures on $B$ (called \emph{a disintegration kernel of} $\nu$) such that
\begin{enumerate}[$(i)$]
    \item If $f\in C(K\times B)$, then
    $$\int_{K\times B} f\di\nu=\int_K \left(\int_{B} f(t,b)\di\nu_t(b)\right)\di\pi_1(\nu)(t).$$
    \item If $C\subset K$ and $D\subset B$ are Borel sets, then
    $$\nu(C\times D)=\int_C \nu_t(D)\di\pi_1(\nu)(t).$$
\end{enumerate}
\end{lemma}

\begin{proof}
    This is a classical fact, see \cite[Theorem 452M]{fremlin4} or \cite[Lemma 2.7]{transference-studia}. 
\end{proof}

We continue by a lemma on factorization of measures.

\begin{lemma}
    \label{L:factorizace-quotient}
    Let $K,L,B$ be Hausdorff compact spaces and $q\colon K\to L$ be a continuous surjection. Let $\mu\in M_+(K)$ and $\lambda\in M_+(L\times B)$ be such that $\pi_1(\lambda)=q(\mu)$. Then there exists a measure $\nu\in M_+(K\times B)$ such that $\pi_1(\nu)=\mu$ and $(q\times \id)(\nu)=\lambda$.
\end{lemma}

\begin{proof}
Let $(\lambda_l)_{l\in L}\subset M_1(B)$ be a disintegration kernel of $\lambda$ provided by Lemma~\ref{L:dezintegrace-obec}. For $f\in C(K\times B,\er)$ we set
\[
p(f)=\overline{\int_K}\left(\int_B f(k,b)\di\lambda_{q(k)}(b)\right)\di\mu(k),
\]
where $\overline{\int}$ denotes the upper integral (see \cite[133I]{fremlin1}). By \cite[Proposition 133J]{fremlin1} we deduce that $p$ is a sublinear functional on $C(K\times B,\er)$. It is clear that $p(f)\le \norm{f}\cdot\norm{\mu}$. We further observe that for $g\in C(L\times B,\er)$ we have
$$\int g\di\lambda= p(g\circ (q\times\id)).$$
Indeed,
\[\begin{aligned}
\int g\di\lambda&=\int_L\left(\int_B g(l,b)\di\lambda_l(b)\right)\di q(\mu)(l)=\int_K\left(\int_B g(q(k),b)\di\lambda_{q(k)}(b)\right)\di \mu(k)\\&=p(g\circ(q\times\id)).
\end{aligned}\]
Here the first equality follows from the disintegration formula (Lemma~\ref{L:dezintegrace-obec}$(i)$), the second one follows from the rules for integration with respect to the image of a measure and the last one is an immediate consequence of the definition of $p$.

Hence, the Hahn-Banach extension theorem yields a linear functional $\psi$ on $C(K\times B,\er)$ such that $\psi\le p$ and $$\psi(g\circ(q\times\id))=\int g\di\lambda\mbox{ for }g\in C(L\times B,\er).$$
Since $\psi(f)\le p(f)\le\norm{f}\cdot\norm{\mu}$, we deduce that $\psi$ is continuous and $\norm{\psi}\le\norm{\mu}$. Further, 
$$\psi(1)=\int 1\di\lambda=\norm{\lambda}=\norm{\mu},$$
so $\psi$ is positive. Let $\nu\in M_+(K\times B)$ represent $\psi$. For each $g\in C(L\times B)$ we have
$$\int g \di (q\times \id)(\nu)=\int g\circ(q\times\id)\di\nu=\psi(g\circ(q\times\id))=\int g\di\lambda,$$
so $(q\times \id)(\nu)=\lambda$.

On the other hand, for $f\in C(K,\er)$ we have
\[\begin{aligned}
\int_K f\di \pi_1(\nu)&=\int_{K\times B} f\circ\pi_1\di\nu=\psi(f\circ\pi_1)\le p(f\circ\pi_1)\\&=\overline{\int_K}\left(\int_B f(k)\di\lambda_{q(k)}(b)\right)\di\mu(k)=\overline{\int_K}f(k)\di\mu(k)=\int_K f\di\mu.
\end{aligned}\]
Replacing $f$ by $-f$ we infer $\pi_1(\nu)=\mu$ and the proof is complete.
\end{proof}

Now we are ready to prove the promised characterizations. 

\begin{thm}\label{T:H-maximal} Let $H\subset C(K,E)$ be a function space.
Let $\sigma\in M_+(K)$. The following assertions are equivalent.
\begin{enumerate}[$(a)$]
     \item $\phi_{H}(\sigma)$ is a maximal measure on $\ov{\aco}^{WOT}\phi_H(K)$.
     \item $\phi_{H_w}(\sigma)$ is a maximal measure on $B_{H_w^*}$.
     \item $\exists\nu\in M_+(K\times B_{E^*})$ such that $\pi_1(\nu)=\sigma$ and $\phi_{H_s}(\nu)$ is a maximal measure on $B_{H_s^*}$.
\end{enumerate}
\end{thm}

\begin{proof}
$(a)\iff(b)$: This follows immediately from Corollary~\ref{cor:identifikacekouli}.

$(c)\implies (b)$:  This is the content of \cite[Lemma 4]{saab-tal}. But the argument provided there is very brief and the key step is not explained, so we decided to give a complete proof here.

Assume $(c)$ holds. To prove that $\phi_{H_w}(\sigma)$ is maximal, we use the Mokobodzki criterion. Let $f:B_{H_w^*}\to\er$ be convex, weak$^*$-continuous and $\ef$-invariant. The proof will be complete if we show that
$$\int f\di\phi_{H_w}(\sigma)=\int f^*\di\phi_{H_w}(\sigma).$$

To this end let $V:H_w\to H_s$ be the canonical isometric injection defined by $V(g)=g\circ\pi_1$ for $g\in H_w$. Then $\phi_{H_w}\circ\pi_1=V^*\circ \phi_{H_s}$. Indeed, for $(t,x^*)\in K\times B_{E^*}$ and $g\in H_w$ we have
$$(V^*\circ\phi_{H_s})(t,x^*)(g)=\phi_{H_s}(t,x^*)(Vg)=(g\circ\pi_1)(t,x^*)=\phi_{H_w}(\pi_1(t,x^*))(g).$$
Further, if $g:B_{H_w^*}\to\er$ is a weak$^*$-Borel function, we may compute
$$\begin{aligned}
 \int_{B_{H_w^*}} g\di\phi_{H_w}(\sigma) &=  \int_{K} g\circ\phi_{H_w} \di\sigma
 = \int_{K\times B_{E^*}} g\circ\phi_{H_w}\circ \pi_1 \di\nu \\ 
 & = \int_{K\times B_{E^*}} g\circ V^*\circ \phi_{H_s} \di\nu
 = \int_{B_{H_s^*}} g\circ V^*\di\phi_{H_s}(\nu).
 \end{aligned}$$

This may be applied, in particular, to $g=f$ and to $g=f^*$. Taking into account that $\phi_{H_s}(\nu)$ is maximal, it is enough to show $f^*\circ V^*=(f\circ V^*)^*$ on $\phi_{H_s}(K\times B_{E^*})$. So, let us compute:

$$\begin{aligned}
 (f\circ V^*)^*(\phi_{H_s}(t,x^*))&= \inf\{ c + \Re (x^*(\h(t))+v(t))\setsep    
 c\in\er,\h\in H,v\in H_w, \\&\qquad c + \Re (\alpha y^*(\h(s))+\alpha v(s))\ge (f\circ V^*)(\alpha \phi_{H_s}(s,y^*))  \\&\qquad\qquad\qquad\mbox{ for }\alpha\in S_{\ef}, s\in K, y^*\in B_{E^*}\}
\\& = \inf\{ c + \Re (x^*(\h(t))+v(t))\setsep    
 c\in\er,\h\in H,v\in H_w, \\&\qquad c + \Re (\alpha y^*(\h(s))+\alpha v(s))\ge (f\circ V^*)(\phi_{H_s}(s,y^*))  \\&\qquad\qquad\qquad\mbox{ for }\alpha\in S_{\ef}, s\in K, y^*\in B_{E^*}\}
 \\& = \inf\{ c + \Re (x^*(\h(t))+v(t))\setsep    
 c\in\er,\h\in H,v\in H_w, \\&\qquad c + \Re ( \alpha y^*(\h(s))+\alpha v(s))\ge f(\phi_{H_w}(s))  \\&\qquad\qquad\qquad\mbox{ for }\alpha\in S_{\ef}, s\in K, y^*\in B_{E^*}\}
  \\&  = \inf\{ c + \Re (x^*(\h(t))+v(t))\setsep    
 c\in\er,\h\in H,v\in H_w, \\&\qquad c + \Re ( y^*(\h(s))+\alpha v(s))\ge f(\phi_{H_w}(s))  \\&\qquad\qquad\qquad\mbox{ for } s\in K, y^*\in B_{E^*},\alpha\in S_{\ef}\}
  \\&  = \inf\{ c + \Re (x^*(\h(t))+v(t))\setsep    
 c\in\er,\h\in H,v\in H_w, \\&\qquad c -\norm{\h(s)}+ \Re \alpha v(s)\ge f(\phi_{H_w}(s))  \mbox{ for } s\in K,\alpha\in S_{\ef}\}
   \\&  = \inf\{ c + \Re (x^*(\h(t))+v(t))\setsep    
 c\in\er,\h\in H,v\in H_w, \\&\qquad c + \Re \alpha v(s)\ge f(\phi_{H_w}(s))+\norm{\h(s)}  \mbox{ for } s\in K,\alpha\in S_{\ef}\}
  \\&=\inf\{\Re x^*(\h(t)) + (f+\norm{\Upsilon(\cdot)(\h)})^*(\phi_{H_w}(t))\setsep \h\in H\}
 \end{aligned}$$

Let us explain it a bit. The first equality is an application of the definition of the upper envelope.
Note that affine continuous functions on $B_{H_s^*}$ are of the form $c+g$, where $c\in \er$ and $g$ is a real-linear weak$^*$-continuous function on $H_s^*$. Such real-linear functions may be uniformly approximated by evaluation at an element of $H_s=H_l+V(H_w)$. Moreover, it is enough to assume that the inequality holds on extreme points which are contained in $\TT\phi_{H_s}(K\times B_{E^*})$.

The second equality follows from the fact that $V^*$ is linear and $f$ is $\ef$-invariant. To show the third equality recall that $V^*\circ\phi_{H_s}=\phi_{H_w}\circ\pi_1$. The fourth equality follows from the observation that $\alpha y^*\in B_{E^*}$ whenever $y^*\in B_{E^*}$ and $\alpha\in S_{\ef}$, using the fact that the right-hand side of the inequality does not depend on $y^*$. The fifth equality follows by passing to the infimum (over $y^*\in B_{E^*}$) on the left-hand side of the inequality. The sixth equality is obvious. 
In the last equality we use the operator $\Upsilon$ from Proposition~\ref{P:Upsilon}. Note that
$f+\norm{\Upsilon(\cdot)(\h)}$ is an $\ef$-invariant convex weak$^*$-lower semicontinuous convex function and hence the last equality follows from the definition of the upper envelope similarly as the first two equalities. 

Now let us apply the above computation. Since $H$ contains zero function, we deduce $ (f\circ V^*)^*(\phi_{H_s}(t,x^*))\le f^*(\phi_{H_w}(t))$.
Further, given $\h\in H$ we have
$$\begin{aligned}
 \Re x^*(\h(t)) & + (f+\norm{\Upsilon(\cdot)(\h)})^*(\phi_{H_w}(t))\\&=
 \Re x^*(\Upsilon(\phi_{H_w}(t))(\h)) +(f+\norm{\Upsilon(\cdot)(\h)})^*(\phi_{H_w}(t))
 \\&=( \Re x^*(\Upsilon(\cdot)(\h)) +f+\norm{\Upsilon(\cdot)(\h)})^*(\phi_{H_w}(t))
 \\&\ge f^*(\phi(H_w(t)),
\end{aligned}$$
because we have, for each $\psi\in H_w^*$,
$$ \Re x^*(\Upsilon(\psi)(\h))\ge - \abs{x^*(\Upsilon(\psi)(\h))}\ge -\norm{\Upsilon(\psi)(\h)}.$$

So,
 $$ (f\circ V^*)^*(\phi_{H_s}(t,x^*))\ge f^*(\phi_{H_w}(t))=f^*(V^*(\phi_{H_s}(t,x^*))),$$
 which completes the proof.

$(b)\implies(c)$: Assume that $\phi_{H_w}(\sigma)$ is maximal. Let $V$ be as in the proof of the converse implication. We set
$$F=\{\nu\in M_+(B_{H_s^*})\setsep V^*(\nu)=\phi_{H_w}(\sigma)\}.$$
Then
$$\nu\in F,\nu_1\in M_+(B_{H_s^*}),\nu_1\succ \nu\implies \nu_1\in F.$$
Indeed, if $\nu_1\succ\nu$, then clearly $V^*(\nu_1)\succ V^*(\nu)$. Since $V^*(\nu)=\phi_{H_w}(\sigma)$ is maximal, we deduce $V^*(\nu_1)=\phi_{H_w}(\sigma)$, so $\nu_1\in F$.

Hence, $F$ contains some maximal measures. Let $\nu$ be a maximal measure in $F$. Then $\nu$ is carried by the set
$$B=S_{\ef}\phi_{H_s}(K\times B_{E^*})\cap (V^*)^{-1}(\phi_{H_w}(K)).$$
We claim that $B\subset \phi_{H_s}(K\times B_{E^*})$. To see it, fix any $\psi\in B$. 
Then $\psi= \alpha \phi_{H_s}(t,x^*)$ for some $\alpha\in S_{\ef}$, $t\in K$, $x^*\in B_{E^*}$. Moreover, 
$$V^*(\psi)=V^*(\alpha \phi_{H_s}(t,x^*))=\alpha\phi_{H_w}(t)$$
(recall that $V^*$ is linear and $V^*\circ\phi_{H_s}=\phi_{H_w}\circ \pi_1$).
Since $V^*(\psi)\in \phi_{H_w}(K)$, there is $t^\prime\in K$ with $\alpha\phi_{H_w}(t)=\phi_{H_w}(t^\prime)$. But then 
$\alpha\phi_{H_s}(t,x^*)=\phi_{H_s}(t^\prime,x^*)$ (by Lemma~\ref{L:phit=alphaphis}), thus indeed $\psi\in \phi_{H_s}(K\times B_{E^*})$. This completes the proof that $\nu$ is carried by $\phi_{H_s}(K\times B_{E^*})$.

Next observe that there is a continuous mapping $\omega$ making the following diagramm commutative
$$\xymatrix{K\times B_{E^*} \ar[r]^{\phi_{H_s}} \ar[d]_{\phi_{H_w}\times\id} & \phi_{H_s}(K\times B_{E^*}) \\ \phi_{H_w}(K)\times B_{E^*} \ar[ur]_{\omega} & }$$
Indeed, the horizontal and vertical arrows represent continuous surjections. Moreover, if $(\phi_{H_w}\times \id)(t,x^*)=(\phi_{H_w}\times\id)(s,y^*)$, then $\phi_{H_s}(t,x^*)=\phi_{H_s}(s,y^*)$ by Lemma~\ref{L:phit=alphaphis} (applied to $\alpha=1$). Hence there is some $\lambda\in M_+(\phi_{H_w}(K)\times B_{E^*})$ such that $\omega(\lambda)=\nu$.

Moreover, for $t\in K$ and $x^*\in B_{E^*}$ we have
$$ (V^*\circ\omega)(\phi_{H_w}(t),x^*) =V^*(\omega(\phi_{H_w}(t,x^*)))=V^*(\phi_{H_s}(t,x^*))=\phi_{H_w}(t),$$
i.e., $V^*\circ \omega=\pi_1$ on $\phi_{H_w}(K)\times B_{E^*}$. In particular, 
$$\phi_{H_w}(\sigma)=(V^*\circ\omega)(\lambda)=\pi_1(\lambda).$$
Therefore we may use Lemma~\ref{L:factorizace-quotient} to get some $\nu_0\in M_+(K\times B_{E^*})$ such that
$\pi_1(\nu_0)=\sigma$ and $(\phi_{H_w}\times \id)(\nu_0)=\lambda$. 
Then  $\phi_{H_s}(\nu_0)=\omega(\lambda)=\nu$ and the proof is finished (as $\nu$ is maximal).
\end{proof}

We finish this section by a definition: Assume that $H\subset C(K,E)$ is a function space.  A complex or vector measure $\mu$ on $K$ will be called \emph{$H$-boundary} if its variation $\abs{\mu}$ satisfies the equivalent conditions of Theorem~\ref{T:H-maximal}.

\section{Representation theorems and notions of simpliciality}\label{sec:reprez}

In this section we formulate two basic forms of representation theorems for vector valued function spaces and introduce possible variants of simpliciality. We start by the first version which easily follows from the known results using the identification from Section~\ref{sec:reprez-oper}.

\begin{thm}\label{T:reprez-oper}
    Let $H\subset C(K,E)$ be a function space. Then for each representable operator $U\in L(H,E)$ there is some $H$-boundary measure $\mu\in M_U(H)$.
\end{thm}

\begin{proof}
    Let $\Upsilon$ be the operator from Proposition~\ref{P:Upsilon} and let $\varphi\in H_w^*$ be the preimage of $U$ under $\Upsilon$. By \cite[Theorem 1.3]{saab-canad} (applied to the one-dimensional Banach space $E$) there is an $H_w$-boundary measure $\mu\in M_\varphi(H_w)$. We conclude by Proposition~\ref{P:Upsilon}$(vi)$.
\end{proof}

\begin{thm}\label{T:reprez-funct}
    Let $H\subset C(K,E)$ be a function space. Then for each $\varphi\in H^*$ there is some $H$-boundary vector measure $\mu\in M_\varphi(H)$.
\end{thm}

\begin{proof}
    This is the result of \cite{saab-tal}. Let us briefly recall the argument. Let $\varphi\in H^*$.
    Let $\widetilde{\varphi}\in H_l^*$ we such that $\varphi=\widetilde{\varphi}\circ T$. Let $\psi\in H_s^*$ be an extension of $\widetilde{\varphi}$ with the same norm. By \cite[Theorem 1.3]{saab-canad} (applied to the one-dimensional Banach space $E$) there is an $H_s$-boundary measure $\nu\in M_\psi(H_s)$. Then $T^*\nu\in M_\varphi(H)$. Moreover, $\abs{T^*\nu}=\pi_1(\abs{\nu})$ (by \cite[Proposition 3.5(a)]{transference-studia}) and hence $T^*\nu$ is $H$-boundary by Theorem~\ref{T:H-maximal}.
\end{proof}

Next we are going to define four versions of simpliciality of vector-valued function spaces. They generalize the scalar versions in two ways. To compare them let us recall the scalar notions. Let $H\subset C(K,\ef)$ be a (scalar) function space separating points. $H$ is said to be
\begin{itemize}
      \item \emph{simplicial} if for any $t\in K$ there is a unique $H$-boundary measure in $M_t(H)$;
    \item \emph{functionally simplicial} if for any $\varphi\in H^*$ there is a unique $H$-boundary vector measure in $M_\varphi(H)$.
\end{itemize}
Simpliciality is a direct generalization of the classical notion (see \cite[Definition 6.1]{lmns}) used in \cite{bezkonstant}. Functional simpliciality was introduced in \cite{bezkonstant}, in \cite{phelps-complex} it was used under the name `uniqueness holds for $H$'. If $H$ is a real function space containing constants and separating points, it is functionally simplicial if and only if its state space is a simplex (see, e.g., \cite[Proposition 5.3]{bezkonstant}). 

Note that we consider simpliciality and its variants only for function spaces separating points, even though the representation theorems do not require this assumption. We find this assumption natural, although the relevant definitions have sense also in general.

Now we are ready to give the promised definitions.
Let $H\subset C(K,E)$ be a function space separating points of $K$. We say that $H$ is
\begin{itemize}
    \item \emph{functionally vector simplicial} if for any $\varphi\in H^*$ there is a unique $H$-boundary vector measure in $M_\varphi(H)$;
    \item \emph{vector simplicial} if for any $t\in K$ and $x^*\in E^*$ there is a unique $H$-boundary vector measure in $M_{x^*\circ\phi_H(t)}(H)$;
    \item \emph{weakly simplicial} if for any $t\in K$ there is a unique $H$-boundary measure in $M_t(H)$;
    \item \emph{functionally weakly simplicial} if for any representable operator $U\in L(H,E)$ there is a unique $H$-boundary measure in $M_U(H)$.
\end{itemize}

In the following observation we formulate basic relationship of these notions to the scalar theory.

\begin{obs}\label{obs:weak simpl} Let $H\subset C(K,E)$ be a function space separating points.
    \begin{enumerate}[$(a)$]
        \item $H$ is weakly simplicial if and only if $H_w$ is simplicial.
        \item $H$ is functionally weakly simplicial if and only if $H_w$ is functionally simplicial.
        \item If $E=\ef$, then
        \begin{enumerate}[$(i)$]
            \item $H$ is vector simplicial $\iff$ $H$ is weakly simplicial $\iff$ $H$ is simplicial;
            \item $H$ is functionally vector simplicial $\iff$ $H$ is functionally weakly simplicial $\iff$ $H$ is functionally simplicial.
        \end{enumerate}
    \end{enumerate}
\end{obs}

\begin{proof}
    Assertions $(a)$ and $(b)$ follow from Proposition~\ref{P:Upsilon} and Theorem~\ref{T:H-maximal}.

    Assertion $(c)$ follows by combining $(a)$ and $(b)$ with Observation~\ref{obs:scalar}.
\end{proof}

\begin{remark}
    Obviously functional vector simpliciality implies vector simpliciality and, similarly, functional weak simpliciality implies weak simpliciality. Converse implications fail even for scalar function spaces containing constants (see \cite[Exercise 6.79]{lmns} or \cite[p. 15]{bezkonstant}).

    The exact relationship between vector simpliciality and weak simpliciality seems not to be clear.
    In the next section we show that these two notions are in general incomparable. However, if $H$ contains constants, then vector simpliciality is stronger than weak simpliciality.  We will analyze this situation in more detail in Section~\ref{sec:skonstantami} below.
\end{remark}

\section{$H$-affine functions}\label{sec:H-aff}

In the study of simpliciality of scalar-valued function spaces, an important role is played by the space of continuous $H$-affine functions denoted by $A_c(H)$ and defined by
$$A_c(H)=\left\{f\in C(K)\setsep f(t)=\int f\di\mu\mbox{ for each }t\in K, \mu\in M_t(H)\right\},$$
see \cite[Definition 3.8]{lmns} for the classical case and \cite[Section 4]{bezkonstant}. In the vector-valued case we have two directions of representation theorems and of simpliciality, therefore we also have two different notions of $H$-affine functions. Their definitions read as follows.

If $H\subset C(K,E)$ is a function space, we set
$$\begin{aligned}
    A_c^w(H)&=\left\{\f\in C(K,E)\setsep \f(t)=(B)\int \f\di\mu\mbox{ for each }t\in K, \mu\in M_t(H)\right\}; 
\\
A_c^v(H)&=\Bigg\{\f\in C(K,E)\setsep x^*(\f(t))=\int \f\di\mu\mbox{ for each }t\in K, x^*\in E^*,\\ &\qquad\qquad\qquad\qquad\qquad\qquad\mu\in M_{x^*\circ\phi_H(t)}(H)\Bigg\}.
\end{aligned}
$$

\begin{obs}
    If $E=\ef$, then $A_c^w(H)=A_c^v(H)=A_c(H)$.
\end{obs}

\begin{proof}
    Assume $E=\ef$. Then $H=H_w$, so equality $A_c^w(H)=A_c(H)$ easily follows. Moreover, $E^*=\ef$, so equality $A_c^v(H)=A_c(H)$ follows as well.
\end{proof}

We shall now investigate properties of $A_c^w(H)$ and $A_c^v(H)$ separately. 
 
\subsection{Properties of $A_c^w(H)$}\label{ss:AcwH}

We start by collecting basic easy properties of $A_c^w(H)$.

\begin{lemma}\label{L:AcH} Let $H\subset C(K,E)$ be a function space.
    \begin{enumerate}[$(a)$]
        \item $A_c^w(H)$ is a closed subspace of $C(K,E)$ containing $H$.
        \item $H_w\subset A_c^w(H)_w\subset A_c(H_w)$.
        \item $\normr{\phi_H(t)}=\normr{\phi_{A_c^w(H)}(t)}$ for $t\in K$.
        \item $M_t(A_c^w(H))=M_t(H)$ for $t\in K$.
        \item $A_c^w(A_c^w(H))=A_c^w(H)$.
    \end{enumerate}
\end{lemma}

\begin{proof}
    Assertion $(a)$ is trivial. 

    $(b)$: Since $H\subset A_c^w(H)$, the first inclusion follows. To prove the second one fix $\f\in A_c^w(H)$ and $x^*\in E^*$. We shall show that $x^*\circ\f\in A_c(H_w)$. To this end take any $t\in K$ and $\mu\in M_t(H_w)$. By Proposition~\ref{P:MtH=MtHw} we get $\mu\in M_t(H)$ and so
    $$\int x^*\circ\f\di\mu=x^*((B)\int \f\di\mu)=x^*(\f(t)),$$
    which completes the argument.

    $(c) \& (d)$: Fix $t\in K$. Assume $\mu\in M_t(A_c^w(H))$. Since $H\subset A_c^w(H)$, we deduce
    $$\f(t)=(B)\int \f\di\mu\mbox{ for each }\f\in H.$$
    Moreover,
    $$\normr{\phi_{A_c^w(H)}(t)}=\norm{\mu}\ge\normr{\phi_H(t)}.$$

    Further, assume $\nu\in M_t(H)$. By the definition of $A_c^w(H)$ we have
    $$\f(t)=(B)\int \f\di\nu\mbox{ for each }\f\in A_c^w(H).$$
    Moreover,
    $$\normr{\phi_{H}(t)}=\norm{\nu}\ge\normr{\phi_{A_c^w(H)}(t)}.$$
    By combining the two inequalities we deduce $\normr{\phi_H(t)}=\normr{\phi_{A_c^w(H)}(t)}$, so $(c)$ is proved. We conclude that $\mu\in M_t(H)$ and $\nu\in M_t(A_c(H))$, so $(d)$ follows.

    Finally, assertion $(e)$ follows from $(d)$.
\end{proof}

We continue by some further (more precise) properties.

\begin{prop}\label{P:AcH vlastnosti}  Let $H\subset C(K,E)$ be a function space.
    \begin{enumerate}[$(a)$]
        \item If $f\in A_c(H_w)$ and $x\in E$, then $f\cdot x\in A_c^w(H)$.
        \item $A_c(H_w)=A_c^w(H)_w$.
        \item Assume that $H_w$ contains constants (this takes place, in particular, if $H$ contains some constants). Then $A_c^w(H)$ contains constants and $M_t(H)\subset M_1(K)$ for each $t\in K$.
    \end{enumerate}
\end{prop}

\begin{proof}
    $(a)$: Assume $f\in A_c(H_w)$ and $x\in E$. Fix $t\in K$ and $\mu\in M_t(H)$. By Proposition~\ref{P:MtH=MtHw} we deduce $\mu\in M_t(H_w)$, so $\mu\in M_t(A_c(H_w))$ by Lemma~\ref{L:AcH}$(d)$. It follows that
    $$f(t)=\int f\di \mu.$$
    We deduce that
    $$(f\cdot x)(t)=f(t)\cdot x=\left(\int f\di\mu\right)\cdot x=(B)\int f\cdot x\di\mu.$$
    We obtain that $f\cdot x\in A_c^w(H)$.

    $(b)$: Inclusion `$\supset$' follows from Lemma~\ref{L:AcH}$(b)$. To prove the converse fix any
    $f\in A_c(H_w)$. Find $x\in E\setminus\{0\}$ and $x^*\in E^*$ with $x^*(x)=1$. By $(a)$ we know that $f\cdot x\in A_c^w(H)$. Thus $f=x^*\circ(f\cdot x)\in A_c^w(H)_w$.

    $(c)$: Assume that $H_w$ contains constants. Given $t\in K$,  Proposition~\ref{P:MtH=MtHw} and \cite[Remark 3.2]{bezkonstant} yield $M_t(H)=M_t(H_w)\subset M_1(K)$. Further, by $(a)$ we deduce that $A_c^w(H)$ contains constants.
\end{proof}

We continue by a result on preservation of the Choquet boundary and boundary measures directly generalizing its scalar analogues \cite[Proposition 3.67]{lmns} and \cite[Proposition 4.5]{bezkonstant}.

\begin{prop}\label{P:AcwHboundary} Let $H\subset C(K,E)$ be a function space separating points of $K$. Then the following assertions hold.
    \begin{enumerate}[$(a)$]
        \item $\Ch_HK=\Ch_{A_c^w(H)}K$;
        \item $H$-boundary measures  and $A_c^w(H)$-boundary measures on $K$ coincide.
    \end{enumerate}
\end{prop}

\begin{proof}
    Assertion $(a)$ follows from $(b)$, so it is enough to prove $(b)$. It suffices to prove it for
    positive measures. So, let $\mu\in M_+(K)$. We have the following sequence of equivalences:
    $$\begin{gathered}
        \mu\mbox{ is $H$-boundary} \iff \mu\mbox{ is $H_w$-boundary} \iff \mu\mbox{ is $A_c(H_w)$-boundary} \\ \iff \mu\mbox{ is $A_c^w(H)_w$-boundary}  \iff \mu\mbox{ is $A_c^w(H)$-boundary} 
    \end{gathered}$$
    The first equivalence follows from Theorem~\ref{T:H-maximal}. The second one follows from \cite[Proposition 4.5$(ii)$]{bezkonstant}. The third one follows from Proposition~\ref{P:AcH vlastnosti}$(b)$. The last equivalence follows again from Theorem~\ref{T:H-maximal}.
\end{proof}

In the following proposition we show that weak simpliciality may be canonically reduced to the scalar case.

\begin{prop}\label{P:weaksimp} Let $H\subset C(K,E)$ be a function space separating points of $K$.
    The following assertions are equivalent.
    \begin{enumerate}[$(a)$]
        \item $H$ is weakly simplicial.
        \item $A_c^w(H)$ is weakly simplicial.
        \item $H_w$ is simplicial.
        \item $A_c(H_w)$ is simplicial.
    \end{enumerate}
\end{prop}

\begin{proof}
    Equivalence $(a)\iff (c)$ follows from Observation~\ref{obs:weak simpl}$(a)$. Equivalence $(b)\iff (d)$ follows again from Observation~\ref{obs:weak simpl}$(a)$ using Proposition~\ref{P:AcH vlastnosti}$(b)$. Finally, equivalence $(c)\iff(d)$ follows from \cite[Proposition 4.6]{bezkonstant}.
\end{proof}

\begin{remark}
    The previous proposition implies that studying weak simpliciality of $H$ reduces to the study of simpliciality of $H_w$ or $A_c(H_w)=A_c^w(H)_w$ and hence the results of \cite{bezkonstant} may be applied. In particular, if $H_w$ contains constants, \cite[Proposition 5.4]{bezkonstant} applies.  In general \cite[Theorem 6.1]{bezkonstant} may be applied. We will comment these connections at the appropriate places below.
\end{remark}

We finish this subsection by a result identifying $A_c^w(H)$ with $\fra(\es(A_c(H_w)),E)$ in case $H_w$ contains constants.
We recall that $\fra(X,E)$ denotes the space of all affine continuous functions from a compact convex set $X$ into a Banach space $E$. It is a vector-valued version of the space $\fra(X)$ mentioned in the introduction. For a (scalar) function space $H\subset C(K)$ containing constants and separating points of $K$, the state space is defined as
\[
\es(H)=\{s\in H^*\setsep \norm{s}=1=s(1)\},
\]
see \cite[Definition 4.25]{lmns}.
Then $\es(H)$ is a weak$^*$-compact convex subset of $H^*$ and $K$ is homeomorphically embedded into $\es(H)$ by a mapping $\phi_H\colon K\to \es(H)$ defined as $\phi_H(t)(h)=h(t)$, $t\in K$, $h\in H$. Further, the mapping $\Phi_H\colon H\to \fra(\es(H))$ defined by $\Phi_H(h)(s)=s(h)$, $s\in \es(H)$, $h\in H$, is an isometric embedding of $H$ into $\fra(\es(H))$. If $H$ is closed and selfadjoint, $\Phi_H$ is surjective with the inverse given as $\Phi_H^{-1}(F)=F\circ \phi_H$, $F\in \fra(\es(H))$. (For proofs of these observations see \cite[Section 4.3]{lmns} and \cite[Lemma 5.1]{bezkonstant}.)

\begin{prop}
    \label{p:acwh-aese}
Let $H\subset C(K,E)$ be a function space separating points such that $H_w$ contains constants. Let $X=\es(A_c(H_w)))$ be the state space of $A_c(H_w)$. 
Then the mapping $I$ given by 
\[
I \f=\f\circ\phi_{A_c(H_w)},\quad  \f\in \fra(X,E),
\]
is a surjective isometry of $\fra(X,E)$ onto $A_c^w(H)$.
\end{prop}

\begin{proof}
We observe that $A_c(H_w)$ is closed and selfadjoint as $M_t(H_w)\subset M_1(K)$ for each $t\in K$ (by Proposition~\ref{P:AcH vlastnosti}$(c)$). We write $\phi$ for the mapping $\phi_{A_c(H_w)}$. 

 If $\f\in\fra(X,E)$, then clearly $I\f\in C(K,E)$. Moreover,
 $$\norm{I\f}=\norm{\f|_{\phi(K)}}=\sup_{x^*\in B_{E^*}}\norm{x^*\circ \f|_{\phi(K)}}=\sup_{x^*\in B_{E^*}}\norm{x^*\circ \f}=\norm{\f}.$$
 Indeed, the first equality follows from the definition of $I$, the second and fourth ones follow from the dual formula for the norm. The third one is due to the fact that $X=\ov{\co \phi(K)}$. Thus $I$ is a linear isometric inclusion of $\fra(X,E)$ into $C(K,E)$.
 It remains to show that the range is $A_c^w(H)$.

 To prove inclusion `$\subset$' fix $\f\in \fra(X,E)$,  $t\in K$ and $\mu\in M_t(H)=M_t(H_w)=M_t(A_c(H_w))$. By \cite[Proposition 4.26(c)]{lmns} we deduce that $\phi(\mu)$ is a probability measure on $X$ with the barycenter $r(\phi(\mu))=\phi(t)$. Hence, given $x^*\in E^*$, we have
\[\begin{aligned}
x^*\left((B)\int_K I\f\di\mu\right)&=\int_K x^*\circ\f\circ \phi\di\mu=
\int_X x^*\circ \f\di\phi(\mu)=(x^*\circ\f)(r(\phi(\mu)))\\&=x^*(\f(\phi(t)))=x^*(I\f(t)).
 \end{aligned}\]
Hence $(B)\int I\f\di\mu=I\f(t)$. We deduce that $I\f\in A_c^w(H)$. 

To check the converse inclusion, let $\g\in A_c^w(H)$ be given. We consider the bilinear mapping $B\colon (A_c(H_w))^*\times E^*\to \ef$ defined as
\[
B(s,x^*)=s(x^*\circ \g),\quad s\in (A_c(H_w))^*, x^*\in E^*.
\]
For a fixed $s\in (A_c(H_w))^*$, the mapping $x^*\mapsto B(s,x^*)$ is a linear weak$^*$-continuous form on $E^*$. 

Indeed, linearity is obvious and, if a bounded net $(x_i^*)$ weak$^*$-converges to $x^*\in E^*$, then 
\[
\sup_{t\in K}\abs{x_i^*(\g(t))-x^*(\g(t))}=\sup_{x\in \g(K)}\abs{x_i^*(x)-x^*(x)}\to 0,
\]
as $g(K)$ is a norm compact set in $E$ and $(x_i^*)$ is a bounded net.
Hence $$\norm{x_i^*\circ \g-x^*\circ \g}\to 0.$$ It follows that 
\[
B(s,x_i^*)=s(x_i^*\circ \g)\to s(x^*\circ \g)=B(s,x^*).
\] 
So, the mapping $x^*\mapsto B(s,x^*)$ is a linear functional which is weak$^*$-continuous on bounded sets. By the Banach-Dieudonn\'e theorem it is weak$^*$-continuous on $E^*$, so there exists $\f(s)\in E$ such that $B(s,x^*)=x^*(\f(s))$, $x^*\in E^*$.

In this way we define a function $\f\colon X\to E$ which is clearly affine. Further, for any $t\in K$ and $x^*\in E^*$ we have
$$x^*(\f(\phi(t)))=\phi(t)(x^*\circ\g)=x^*(\g(t)).$$
It follows that $\f\circ\phi=\g$. Hence, to complete the proof it is enough to show that $\f$ is continuous on $X$

To this end, let a net $(s_i)_{i\in I}\subset X$ converge (weak$^*$) to $s\in X$. We want to prove that $\norm{\f(s_i)-\f(s)}\to 0$. So let us assume the contrary. Then  there are $\ep>0$ and $x_i^*\in B_{E^*}$, $i\in I$, such that 
\[
\abs{x_i^*(\f(s_i))-x_i^*(\f(s))}\ge \ep \mbox{ for }i\in I.
\]
Up to passing to a subnet we may assume that  $x_i^*\to x^*\in B_{E^*}$ (in the weak$^*$ topology).
Since $x_i^*(\f(s))\to x^*(\f(s))$, we may assume that
\[
\abs{x_i^*(\f(s_i))-x^*(\f(s))}\ge \tfrac{\ep}2 \mbox{ for }i\in I.
\]
Then we obtain
\[\begin{aligned}
\tfrac{\ep}{2}&\le \abs{x_i^*(\f(s_i))-x^*(\f(s))}=\abs{B(s_i, x_i^*)-B(s,x^*)}=\abs{s_i(x_i^*\circ\g)-s(x^*\circ \g)}\\&\le\abs{s_i(x_i^*\circ\g)-s_i(x^*\circ \g)}+\abs{s_i(x^*\circ\g)-s(x^*\circ \g)}
\\&\le \norm{x_i^*\circ\g-x^*\circ \g}+\abs{s_i(x^*\circ\g)-s(x^*\circ \g)}\to0,\end{aligned}
\]
as both summands have limit zero -- the first one by compactness of $\g(K)$ (as above) and the second one by the assumption $s_i\to s$ in $X$. In this way we get a contradiction completing the proof.
\end{proof}

\subsection{Properties of $A_c^v(H)$}

We continue with properties of the second version, i.e., of $A_c^v(H)$. It turns out that the situation is quite different. However, we start by basic properties which are rather similar to the previous case.

\begin{lemma}\label{L:AcvH} Let $H\subset C(K,E)$ be a function space.
    \begin{enumerate}[$(a)$]
        \item $A_c^v(H)$ is a closed subspace of $C(K,E)$ containing $H$.
        \item $\norm{\phi_H(t)}=\norm{\phi_{A_c^v(H)}(t)}$ for $t\in K$.
        \item $\norm{x^*\circ\phi_H(t)}=\norm{x^*\circ\phi_{A_c^v(H)}(t)}$ for $t\in K$ and $x^*\in E^*$.
        \item $M_{x^*\circ \phi_{A_c^v(H}(t)}(A_c^v(H))=M_{x^*\circ \phi_H(t)}(H)$ for $t\in K$ and $x^*\in E^*$.
        \item $A_c^v(A_c^v(H))=A_c^v(H)$.
    \end{enumerate}
\end{lemma}

\begin{proof}
    Assertion $(a)$ is trivial.

    $(c)\&(d)$: Fix $t\in K$ and $x^*\in E^*$. Let $\mu\in M_{x^*\circ\phi_{A_c^v(H)}(t)}(A_c^v(H))$.
    Since $H\subset A_c^v(H)$, necessarily
    $$x^*(\h(t))=\int \h\di\mu,\quad \h\in H.$$
    It follows that 
    $$\norm{x^*\circ\phi_H(t)}\le\norm{\mu}=\norm{x^*\circ\phi_{A_c^v(H)}(t)}.$$

    Further, if $\nu\in M_{x^*\circ\phi_H(t)}(H)$, the definition of $A_c^v(H)$ implies that
    $$x^*(\f(t))=\int \f\di\mu,\quad \f\in A_c^v(H).$$
    It follows that
    $$\norm{x^*\circ\phi_{A_c^v(H)}(t)}\le\norm{\nu}=\norm{x^*\circ\phi_{H}(t)}.$$

    Combining these two inequalities we see that $\norm{x^*\circ\phi_{A_c^v(H)}(t)}=\norm{x^*\circ\phi_{H}(t)}$, i.e. $(c)$ holds. Moreover, $\mu\in M_{x^*\circ\phi_H(t)}(H)$ and 
    $\nu\in  M_{x^*\circ\phi_{A_c^v(H)}(t)}(A_c^v(H))$, which completes the proof of $(d)$.

    Assertion $(b)$ now follows from $(c)$ using Lemma~\ref{L:normyevaluaci}$(a)$. 

    Finally, assertion $(e)$ follows from $(d)$.
\end{proof}

We continue by looking at the relationship of $H$-boundary and $A_c^v(H)$-boundary measures. One implication is rather easy, as witnessed by the following lemma.

\begin{lemma} Let $H\subset C(K,E)$ be a function space.
    \begin{enumerate}[$(a)$]
        \item Any $H$-boundary measure is also $A_c^v(H)$-boundary. In particular, $\Ch_H K\subset \Ch_{A_c^v(H)}K$.
        \item If $A_c^v(H)$ is vector simplicial, then $H$ is also vector simplicial.
    \end{enumerate}
\end{lemma}

\begin{proof}
    $(a)$: By Theorem~\ref{T:H-maximal} $H$-boundary means $H_w$-boundary and $A_c^v(H)$-boundary means $A_c^v(H)_w$-boundary. Since $H_w\subset A_c^v(H)_w$, the statement follows from the proof of \cite[Proposition 4.5]{bezkonstant}. More precisely, the `only if' part of the quoted proof works in our case (just replace $H$ by $H_w$ and $A_c(H)$ by $A_c^v(H)_w$).
    
   $(b)$: This follows from $(a)$ and Lemma~\ref{L:AcvH}$(d)$.
\end{proof}

The converse implication fails in general. This is witnessed by the following example.

\begin{example}\label{ex:nezachovani}
    There are a compact space $K$, Banach space $E$ and a closed function space $H\subset C(K,E)$ separating points such that the following conditions are satisfied.
    \begin{enumerate}[$(i)$]
        \item $H(t)=E$ for each $t\in K$.
        \item $H$ is both vector simplicial and weakly simplicial.
        \item The spaces $A_c^v(H)$ and $A_c^w(H)$ are mutually incomparable.
        \item $A_c^v(H)_w=C(K)$.
        \item $\Ch_HK\subsetneqq \Ch_{A_c^v(H)}K$.
        \item $A_c^v(H)$ is not vector simplicial.
    \end{enumerate}
\end{example}

\begin{proof}
    Let $K=\{-2,-1,0,1,2\}$, $E=(\ef^2,\norm{\cdot}_\infty)$ and let
    $$H=\{\f\in C(K,E)\setsep \f(0)=\tfrac12(\f(-1)+\f(1))\ \&\ f_1(0)=\tfrac14(f_1(-2)+f_1(2))\}.$$
    We will show that $H$ has the required properties. It is clear that $H$ is a closed subspace of $C(K,E)$ separating points of $K$. Further properties will be proved in several steps.

\smallskip
        
    {\tt Step 1:} $H(t)=E$ for each $t\in K$, hence $(i)$ holds.

\smallskip

  Let $x,y\in E$ be arbitrary. Then the functions
   $$\begin{gathered}
     \h_1(1)=x,\h_1(-1)=y,\h_1(0)=\tfrac12(x+y),\h_1(2)=\h_1(-2)=x+y,\\
      \h_2(1)=\h_2(-1)=\h_2(0)=\tfrac14(x+y),\h_2(2)=x,\h_2(-2)=y, \\
       \h_3(1)=\h_3(-1)=\h_3(0)=x,\h_3(2)=\h_3(-2)=2x
   \end{gathered}$$
   belong to $H$ and witness that $H(t)=E$ for each $t\in K$.

   \smallskip

{\tt Step 2:} $H_w=\{f\in C(K)\setsep f(0)=\tfrac12(f(-1)+f(1))\}$.
    
    \smallskip
    
    Indeed, inclusion `$\subset$' is obvious. To see the converse, fix any $f$ in the set on the right-hand side. Define $\f\in C(K,E)$ by setting $f_1=0$, $f_2=f$. Then $\f\in H$ and $e_2^*\circ\f=f$, so $f\in H_w$.

\smallskip

{\tt Step 3:} We easily deduce from Step 2 that $\Ch_HK=\{-2,-1,1,2\}$, $H$ is weakly simplicial and 
    $A_c^w(H)=\{\f\in C(K,E)\setsep \f(0)=\tfrac12(\f(-1)+\f(1))\}$.

    \smallskip

{\tt Step 4:} If $t\in \{-1,1\}$ and $x^*\in E^*$, then $M_{x^*\circ\phi_H(t)}(H)=\{\ep_t\otimes x^*\}$.

\smallskip

Since the role of $1$ and $-1$ is symmetric, it is enough to prove the statement for $t=1$. Let $x^*\in E^*$. Since $E^*$ is canonically identified with $(\ef^2,\norm{\cdot}_1)$, we identify $x^*$ with some $(\alpha,\beta)\in \ef^2$. Then
$$x^*\circ\phi_H(1)(\f)=\alpha f_1(1)+\beta f_2(1),\quad \f\in H.$$
Let $(\gamma,\delta)\in E$ be of norm one such that $\alpha\gamma+\beta\delta=\norm{(\alpha,\beta)}_1$. Consider the function
 $$\f_1(1)=(\gamma,\delta),\ \f_1(-1)=(-\gamma,-\delta),\ \f_1(0)=\f_1(2)=\f_1(-2)=(0,0).$$
It belongs to $H$, $\norm{\f_1}=1$ and $x^*\circ\phi_H(1)(\f_1)=\norm{x^*}$. We deduce that
$\norm{x^*\circ\phi_H(1)}=\norm{x^*}$. (More precisely, $\f_1$ witnesses that `$\ge$' holds, but the converse inequality is obvious.)

Let $\mu\in M_{x^*\circ\phi_H(1)}(H)$. Then $\norm{\mu}=\norm{x^*}$ and
$$\alpha f_1(1)+\beta f_2(1)=\int \f\di\mu,\quad \f\in H.$$
If we plug there $\f_1$, we get
$$\norm{x^*}=\ip{\mu(1)-\mu(-1)}{(\gamma,\delta)}.$$
It follows that $\mu$ is carried by $\{1,-1\}$. If we plug there function $\f_2$ defined by
$$\f_2(1)=\f_2(2)=\f_2(-2)=(\gamma,\delta),\ \f_2(0)=(\tfrac\gamma2,\tfrac\delta2),\ \f_2(-1)=(0,0),$$
we deduce that 
 $$\norm{x^*}=\ip{\mu(1)}{(\gamma,\delta)},$$
 hence $\mu$ is carried by $1$. It follows that for each $\f\in H$ we have
 $$\ip{\mu(1)}{\f(1)}=\int\f\di\mu =x^*(\f(1)).$$
 Since $H(1)=E$ (by Step 1), we deduce $\mu(1)=x^*$. Hence $\mu=\ep_1\otimes x^*$ and the argument is complete.

 \smallskip

 {\tt Step 5:} If $t\in \{-2,2\}$ and $x^*\in E^*$, then $M_{x^*\circ\phi_H(t)}(H)=\{\ep_t\otimes x^*\}$.

\smallskip 

Since the role of $2$ and $-2$ is symmetric, it is enough to prove the statement for $t=2$. The proof is completely analogous to that of Step 4, just instead of functions $\f_1$ and $\f_2$ we use functions
$$\begin{gathered}
   \g_1(2)=(\gamma,\delta),\ \g_1(-2)=(-\gamma,-\delta),\ \g_1(0)=\g_1(1)=\g_1(-1)=(0,0),\\
   \g_2(1)=\g_2(-1)=\g_2(0)=(\tfrac\gamma2,\tfrac\delta4),\ \g_2(2)=(\gamma,\delta),\ \g_2(-2)=(0,0).
\end{gathered}$$

\smallskip

{\tt Step 6:} Assume $x^*=(\alpha,\beta)$. Then $\norm{x^*\circ\phi_H(0)}=\frac12\abs{\alpha}+\abs{\beta}$.

\smallskip

By the very definition we have
$$\norm{x^*\circ\phi_H(0)}=\max\{\abs{\ip{x^*}{\f(0)}}\setsep \f\in H,\norm{\f}\le 1\}.$$
(Note that $\max$ is attained due to compactness.) If $\f\in H$ and $\norm{\f}\le1$, necessarily $\abs{f_2(0)}\le1$ and $\abs{f_1(0)}\le \frac12$.  Conversely, assume $\abs{\zeta}\le \frac12$ and $\abs{\xi}\le1$. Then the function $\f$ defined by
$$\f(0)=\f(1)=\f(-1)=(\zeta,\xi), \f(2)=\f(-2)=(2\zeta,0)$$
belongs to $H$ and has norm at most one. This shows
$$\norm{x^*\circ\phi_H(0)}=\max\{\abs{\alpha\zeta+\beta\xi}\setsep \abs{\zeta}\le\tfrac12,\abs{\xi}\le 1\}=\tfrac12\abs{\alpha}+\abs{\beta}.$$

\smallskip

{\tt Step 7:} Assume that $x^*=(\alpha,\beta)$. Then 
$$M_{x^*\circ\phi_H(0)}(H)=
\{\tfrac\alpha4(\ep_2+\ep_{-2})\otimes (1,0) + \beta(s(\ep_1+\ep_{-1})+(1-2s)\ep_0)\otimes (0,1)\setsep s\in[0,\tfrac12]\}.$$

\smallskip

Assume that $\mu\in M_{x^*\circ\phi_H(0)}$. Since $\mu$ is $E^*$-valued, we may represent it as $\mu=(\mu_1,\mu_2)$ where $\mu_1,\mu_2$ are scalar measures. Moreover, since $E^*=(\ef^2,\norm{\cdot}_1)$, we easily see that $\norm{\mu}=\norm{\mu_1}+\norm{\mu_2}$.  By Step 6 we further deduce that $\norm{\mu}=\frac12\abs{\alpha}+\abs{\beta}$. Moreover,
$$\alpha f_1(0)+\beta f_2(0)=\int \f\di\mu, \quad \f\in H.$$
Fix $\zeta,\xi\in S_{\ef}$ such that $\alpha\zeta=\abs{\alpha}$ and $\beta\xi=\abs{\beta}$. Define the function $\h$ by setting
$$\h(1)=\h(-1)=\h(0)=(\tfrac\zeta2,\xi), \h(2)=\h(-2)=(\zeta,0).$$
Then $\h\in H$ and if we plug $\h$ in the above formula, we get
$$\begin{aligned}
    \norm{\mu}&=\tfrac12\abs{\alpha}+\abs{\beta}=\ip{\mu(\{-1,0,1\})}{(\tfrac\zeta2,\xi)}+\ip{\mu(\{-2,2\})}{(\zeta,0)}
    \\&=\tfrac\zeta2\mu_1(\{-1,0,1\})+\xi\mu_2(\{-1,0,1\})+\zeta\mu_1(\{-2,2\})
    \\&\le \tfrac12\abs{\mu_1}(\{-1,0,1\})+\abs{\mu_2}(\{-1,0,1\})+\abs{\mu_1}(\{-2,2\}) \le\norm{\mu},\end{aligned}$$
 so the equalities hold. It follows that $\mu_1$ is carried by $\{-2,2\}$ and $\mu_2$ is carried by $\{-1,0,1\}$.

 Let us plug there two more functions:
 $$\begin{gathered}
     \h_1(0)=\h_1(2)=\h_1(-2)=0, \h_1(1)=(0,1),\h_1(-1)=(0,-1),\\
      \h_2(0)=\h_2(1)=\h_2(-1)=0, \h_2(2)=(1,0),\h_2(-2)=(-1,0).
 \end{gathered}$$
Since both functions belong to $H$, we deduce $\mu_2(1)=\mu_2(-1)$ and $\mu_1(2)=\mu_1(-2)$. I.e.,
$$\mu_1=u(\ep_2+\ep_{-2}), \mu_2= v(\ep_1+\ep_{-1})+w\ep_0$$
for some $u,v,w\in\ef$. Then for each $\f\in H$ we have
$$\begin{aligned}
    \alpha f_1(0)+\beta f_2(0)&=u(f_1(2)+f_1(-2))+v(f_2(1)+f_2(-1))+w f_2(0)\\&= 4uf_1(0)+(2v+w)f_2(0).\end{aligned}$$
Since $H(0)=E$ (by Step 1), we deduce $u=\frac\alpha4$ and $2v+w=\beta$. Since
$$\tfrac12\abs{\alpha}+\abs{\beta}=\norm{\mu}=2\abs{u}+2\abs{v}+\abs{w},$$
we deduce $\abs{\beta}=2\abs{v}+\abs{w}$. Thus $v=s\beta$ for some $s\in[0,\frac12]$.
Thus 
$$\mu=\tfrac\alpha4(\ep_2+\ep_{-2})\otimes (1,0) + \beta(s(\ep_1+\ep_{-1})+(1-2s)\ep_0)\otimes (0,1).$$
This proves inclusion `$\subset$'. Since the converse inclusion is obvious, proof of Step 7 is complete.

\smallskip

{\tt Step 8:} $H$ is vector simplicial. In particular, condition $(ii)$ is fulfilled.

\smallskip

Let $t\in K$ and $x^*=(\alpha,\beta)\in E^*$. If $t\ne0$, $M_{x^*\circ\phi_H(t)}$ contains only one measure (by Steps 4 and 5). For $t=0$ the description of $M_{x^*\circ\phi_H(0)}$ is provided by Step 7. Hence, if $\beta=0$, it contains only one measure. If $\beta\ne0$, it contains infinitely many measures, but only one of them is $H$-boundary. Indeed, since $K$ is finite, all measures are discrete and discrete $H$-boundary measures are carried by the Choquet boundary. This boundary is described in Step 3, so the only $H$-boundary measure in 
$M_{x^*\circ\phi_H(0)}$ is 
$$\tfrac\alpha4(\ep_2+\ep_{-2})\otimes (1,0) + \tfrac\beta2(\ep_1+\ep_{-1})\otimes (0,1).$$
This completes the proof of vector simpliciality of $H$. By combining with Step 3 we conclude that $(ii)$ is fulfilled.

\smallskip

{\tt Step 9:} We have
$$A_c^v(H)=\{\f\in C(K,E)\setsep f_2(0)=\tfrac12(f_2(-1)+f_2(1))\ \&\ f_1(0)=\tfrac14(f_1(-2)+f_1(2))\},$$ in particular, condition $(iii)$ is fulfilled.

\smallskip

Due to Steps 4, 5 and 7 we get
$$\begin{aligned}
    A_c^v(H)&=\{\f\in C(K,E)\setsep \forall \alpha,\beta\in\ef\,\forall s\in[0,\tfrac12]\colon\alpha f_1(0)+\beta f_2(0)\\&\qquad\qquad=\tfrac\alpha4(f_1(2)+f_1(-2))+\beta(s(f_2(1)+f_2(-1))+(1-2s) f_2(0))\}
    \\&=\{\f\in C(K,E)\setsep f_1(0)=\tfrac14(f_1(2)+f_1(-2)) 
      \\&\qquad \qquad \&\ \forall s\in[0,\tfrac12]\colon f_2(0)=s(f_2(1)+f_2(-1))+(1-2s) f_2(0))\}
      \\&=\{\f\in C(K,E)\setsep f_2(0)=\tfrac12(f_2(-1)+f_2(1))\ \&\ f_1(0)=\tfrac14(f_1(-2)+f_1(2))\}.\end{aligned}$$
This proves the formula. By combining with Step 3 we see that $A_c^w(H)$ and $A_c^v(H)$ are incomparable, hence $(iii)$ is fulfilled.

\smallskip

{\tt Step 10:} $A_c^v(H)_w=C(K)$. In particular, conditions $(iv)-(vi)$ are fulfilled.

\smallskip

It follows from the formula in Step 9 that the space $A_c^v(H)_w$ contains functions $f\in C(K)$ satisfying either $f(0)=\frac12(f(1)+f(-1))$ or $f(0)=\frac14(f(2)+f(-2))$. We will check that any $f\in C(K)$ is the sum of two such functions. Indeed, let $f\in C(K)$ be arbitrary. Define two functions $g,h$ by setting
$$\begin{gathered}
 g(0)=g(1)=g(-1)=f(0), g(2)=f(2), g(-2)=f(-2),  \\
 h(0)=h(2)=h(-2)=0, h(1)=f(1)-f(0), h(-1)=f(-1)-f(0).
\end{gathered}$$
Then $g,h\in A_c^v(H)_w$ by the above and clearly $f=g+h$. This completes the proof of the equality, which is simultaneously the content of condition $(iv)$. We easily deduce that $\Ch_H A_c^v(H)=K$, so condition $(v)$ follows (using Step 3). Hence all measures on $K$ are $A_c^v(H)$-boundary. Since, for example, $M_{(0,1)\otimes\phi_{A_c^v(H)}(0)}=M_{(0,1)\otimes\phi_H(0)}$ contains infinitely many measures (by Step 7), we conclude that $A_c^v(H)$ is not vector simplicial, i.e., $(vi)$ holds.
\end{proof}

We continue by two examples witnessing that $A_c^v(H)$ and vector simpliciality of $H$ depend on the isometric structure of $E$ and may change after a renorming of $E$.

\begin{example}\label{ex:renorm1}
   Let $E=(\ef^2,\norm{\cdot}_p)$ for some $p\in (1,\infty)$. Let $K$ and $H$ be as in the proof of Example~\ref{ex:nezachovani}. Then $A_c^v(H)=H$.
\end{example}

\begin{proof}
Note that $E^*=(\ef^2,\norm{\cdot}_q)$ where $q$ is the conjugate exponent. Steps 1--5 of the proof work in the new setting exactly in the same way and the same properties may be established. Then the difference starts. Similarly as in Step 6 we get
$$\norm{x^*\circ\phi_H(0)}=\max\{ \abs{\alpha\zeta+\beta\xi}\setsep \norm{(\zeta,\xi)}_p\le 1,\abs{\zeta}\le\tfrac12\},$$
but there is not such a simple formula as in the previous example. 

On the other hand, in this case we have that
$$\exists \varepsilon>0\, \forall \alpha\in\ef\colon \abs{\alpha}<\ep\implies \norm{(\alpha,1)\circ \phi_H(0)}=\norm{(\alpha,1)}_q.$$
Indeed, given $\alpha\in\ef$, there is a unique point $(\zeta,\xi)\in E$, $\norm{(\zeta,\xi)}_p=1$ such that
$\alpha\zeta+\xi=\norm{(\alpha,1)}_q$. This point may be explicitly computed --
$$(\zeta,\xi)=\tfrac1{(\abs{\alpha}^q+1)^{1/p}}(\overline{\alpha}\cdot\abs{\alpha}^{q-2},1).$$
Hence $\abs{\zeta}<\frac12$ for $\abs{\alpha}$ small enough. In view of the above formula for such values of $\alpha$ we have $\norm{(\alpha,1)\circ \phi_H(0)}=\norm{(\alpha,1)}_q$.
It follows that for such $\alpha$
$$\tfrac12(\ep_1+\ep_{-1})\otimes(\alpha,1)\in M_{(\alpha,1)\circ\phi_H(0)}(H).$$
Hence, for each $\f\in A_c^v(H)$ we have
$$\alpha f_1(0)+f_2(0)=\tfrac\alpha2(f_1(1)+f_1(-1))+\tfrac12(f_2(1)+f_2(-1))$$
whenever $\abs{\alpha}$ is small enough. If we apply this to $\alpha=0$ and to some small $\alpha>0$, we deduce that
$$ f_1(0)=\tfrac12(f_1(1)+f_1(-1))\mbox{ and } f_2(0)=\tfrac12(f_2(1)+f_2(-1)),$$
i.e., $\f(0)=\frac12(\f(1)+\f(-1))$.

Further, the above formula yields that $\norm{(1,0)\circ\phi_H(0)}=\frac12$. Hence, $\frac14(\ep_2+\ep_{-2})\otimes(1,0)\in M_{(1,0)\circ\phi_H(0)}(H)$. Thus each $\f\in A_c^v(H)$ satisfies $f_1(0)=\frac14(f_1(2)+f_1(-2))$. 

By combining the above results we conclude that any $\f\in A_c^v(H)$ belongs to $H$, which completes the proof.
\end{proof}

\begin{example}\label{ex:renorm2}
    Let $K$, $E$ and $H$ be as in Example~\ref{ex:nezachovani}. Then there is an equivalent norm $\norm{\cdot}$ on $E$ such that $H$ is not vector simplicial as a subspace of $C(K,(E,\norm{\cdot}))$.
\end{example}

\begin{proof}
    Let us equip $E=\ef^2$ with the norm defined by
    $$\norm{x}=\max\{\norm{x}_\infty,\tfrac23\norm{x}_1\}, \quad x\in E.$$
    It is clearly an equivalent norm on $E$. Since $H_w$ is not affected by renorming of $E$, the first three steps of the proof of Example~\ref{ex:nezachovani} remain valid, in particular $\Ch_HK=\{-2,-1,1,2\}$.
    
    Consider $x^*\in E^*$ represented by $(1,1)$, i.e.,
    $x^*(x)=x_1+x_2$ where $x_1,x_2$ are the coordinates of $x$. Clearly $\norm{x^*}=\frac32$.
    Thus $\norm{x^*\circ\phi_H(0,0)}\le \frac32$. Moreover, the equality holds as witnessed by the function 
    $$\f(0)=\f(1)=\f(-1)=(\tfrac12,1), \f(2)=\f(-2)=(1,0).$$
    Indeed, $\f\in H$, $\norm{\f}=1$ and $x^*(\f(0))=\frac32$.

    Now consider two measures
    $$\mu=\tfrac12(\ep_1+\ep_{-1})\otimes(1,1),\quad \nu=\tfrac12(\ep_1+\ep_{-1})\otimes(0,1)+\tfrac14(\ep_2+\ep_{-2})\otimes(1,0).$$
    Both are $H$-boundary and 
    $$\int \h\di\mu=\int\h\di\nu=h_1(0)+h_2(0)=x^*(\h(0))\mbox{ for }\h\in H.$$
    Finally,
    $$\begin{aligned}
            \norm{\mu}&=\norm{(1,1)}_{(E,\norm{\cdot})^*}=\tfrac32,\\
            \norm{\nu}&=\norm{(0,1)}_{(E,\norm{\cdot})^*}+\tfrac12\norm{(1,0)}_{(E,\norm{\cdot})^*}=\tfrac32,\end{aligned}$$
    so both measures belong to $M_{x^*\circ\phi_H(0)}$. We conclude that $H$ is not vector simplicial.
\end{proof}

\begin{remark}
    The three previous examples witness the difference between vector simpliciality and weak simpliciality. While $H_w$, weak simpliciality and $A_c^w(H)$ do not change after renorming on $E$ and hence depend just on the isomorphic structure of $E$, vector simpliciality and
    $A_c^v(H)$ strongly depend on the isometric structure of $E$ and may change if we renorm $E$ by an equivalent norm. 
\end{remark}

The last example of this section shows that vector simpliciality does not imply weak simpliciality. Therefore, if we combine the next example with Example~\ref{ex:nezachovani} (more precisely with the properties of $A_c^v(H)$ from that example), we deduce that weak simpliciality and vector simpliciality are mutually incomparable.

\begin{example}\label{ex:vsnews}
    There is a compact metrizable space $K$, a Banach space $E$ and a closed function space $H\subset C(K,E)$ separating points of $K$ and satisfying the following conditions.
    \begin{enumerate}[$(i)$]
    \item $H(t)=E$ for each $t\in K$;
        \item $H$ is vector simplicial;
        \item $H$ is not weakly simplicial;
        \item $H=A_c^v(H)\subsetneqq A_c^w(H)$.
    \end{enumerate}
\end{example}

\begin{proof}
We set
    $$K=([-1,1]\times \{0\}) \cup ([0,1]\times\{-1,1\}) \cup ([-1,0]\times\{-2,2\})\cup\{a,b\},$$
where $K\setminus \{a,b\}$ is equipped with the topology inherited from $\er^2$ and $a,b$ are isolated points. Then $K$ is a compact metrizable space.

Let $E=(\ef^2,\norm{\cdot}_p)$ for some $p\in [1,\infty]$. Set
$$\begin{aligned}
  H=\{\f\in C(K,E)\setsep & \f(t,0)=\tfrac12(\f(t,1)+\f(t,-1))\mbox{ for }t\in [0,1],  
  \\  &\f(t,1)=\tfrac12(\f(t,2)+\f(t,-2))\mbox{ for }t\in [-1,0],
  \\ & \f(0,0)=\tfrac14((f_2(a),f_1(a))+(f_2(b),f_1(b))) \}.
\end{aligned}$$
Then $H$ is clearly a closed subspace of $C(K,E)$. Further properties will be proved in several steps.

\smallskip

{\tt Step 1:} $H$ separates points of $K$.

\smallskip

Let $u,v\in K$ be such that $u\ne v$. We distinguish several cases:
\begin{itemize}
    \item $u\in \{a,b\}$ or $v\in \{a,b\}$. Then $u,v$ are separated by $\f_1\in H$ defined by
    $$\f_1(a)=(1,1), \f_1(b)=(-1,-1), \f_1=(0,0)\mbox{ on }K\setminus\{a,b\}.$$
    \item $u=(t_1,j_1)$, $v=(t_2,j_2)$ where $t_1\ne t_2$. Find $g\in C([-1,1])$ with $g(t_1)\ne g(t_2)$ and define $\f_2$ by
    $$\f_2(s,i)=(g(s),g(s)), \f_2(a)=\f_2(b)=(2g(0),2g(0)).$$
    Then $\f_2\in H$ and $\f_2(u)\ne \f_2(v)$
    \item $u=(t,j_1)$, $v=(t,j_2)$ where $j_1\ne j_2$. Then $u,v$ are separated by the following function:
    $$ \f_3(s,i)=(i,i), \f_3(a)=\f_3(b)=(0,0).
    $$
\end{itemize}

{\tt Step 2:} $H(u)=E$ for each $u\in K$, hence $(ii)$ is valid.

\smallskip

Let $x\in E$ be arbitrary. The function
$$\g(t,i)=x, \g(a)=\g(b)=2(x_2,x_1)$$
witnesses the validity of the property.

\smallskip

{\tt Step 3:}
$\begin{aligned}[t]
  H_w=\{f\in C(K)\setsep & f(t,0)=\tfrac12(f(t,1)+f(t,-1))\mbox{ for }t\in [0,1],  
  \\  &f(t,1)=\tfrac12(f(t,2)+f(t,-2))\mbox{ for }t\in [-1,0] \}.
\end{aligned}$

\smallskip

Indeed, inclusion `$\subset$' is obvious. To prove the converse fix any function $f$ in the set on the right-hand side. We define $\f=(f_1,f_2)\in C(K,E)$ by setting $f_1=f$ and
$$  f_2= \tfrac14(f(a)+f(b))\mbox{ on }K\setminus\{a,b\},\ f_2(a)=f_2(b)=2f(0,0).  $$
Then $\f\in H$ and $e_1^*\circ\f=f$.

\smallskip

{\tt Step 4:} It is now clear that $\Ch_{H_w}K=K\setminus [-1,1]\times\{0\}$, $A_c(H_w)=H_w$ and that $H_w$ is not simplicial. Thus condition $(iii)$ is fulfilled. Further, clearly
$$\begin{aligned}
  A_c^w(H)=\{\f\in C(K,E)\setsep & \f(t,0)=\tfrac12(\f(t,1)+\f(t,-1))\mbox{ for }t\in [0,1],  
  \\  &\f(t,1)=\tfrac12(\f(t,2)+\f(t,-2))\mbox{ for }t\in [-1,0] \}.
\end{aligned}$$

\smallskip

{\tt Step 5:} Fix $x^*\in E^*$. Since $E^*$ is canonically identified with $(\ef^2,\norm{\cdot}_q)$ where $q$ is the conjugate exponent, we may represent $x^*$ by $(\alpha,\beta)\in \ef^2$. 
We fix $(\gamma,\delta)\in E$ of norm one such that $\alpha\gamma+\beta\delta=\norm{x^*}$.

\smallskip

{\tt Step 6:} $M_{x^*\circ\phi_H(a)}=\{\ep_a\otimes x^*\}$ and $M_{x^*\circ\phi_H(b)}=\{\ep_b\otimes x^*\}$.

\smallskip

Since the roles of $a$ and $b$ are symmetric, it is enough to prove the statement for $a$.
 Define 
 $$\h(a)=(\gamma,\delta), \h(b)=(-\gamma,-\delta), \h=(0,0) \mbox{ elsewhere.}$$ 
 Then $\h\in H$, $\norm{\h}=1$ and $x^*\circ\phi_H(a)(\h)=\norm{x^*}$. Hence 
 $\norm{x^*\circ\phi_H(a)}=\norm{x^*}$.

Let $\mu\in M_{x^*\circ\phi_H(a)}$. Then $\norm{\mu}=\norm{x^*}$ and for each $\f\in H$ we have
$$\alpha f_1(a)+\beta f_2(a)=\int\f\di\mu.$$
If we plug there $\h$, we get
$$\norm{x^*}=\ip{\mu(a)}{(\gamma,\delta)} -\ip{\mu(b)}{(\gamma,\delta)}.$$
It folows that $\mu$ is caried by $\{a,b\}$.
Let us further plug there function $\h_1$ defined by
$$\h_1(a)=(\gamma,\delta),\h_1(b)=(0,0), \h_1=(\tfrac\delta4,\tfrac\gamma4) \mbox{ elsewhere}.$$
Then $\h_1\in H$ and we deduce that
$$\norm{x^*}= (x^*\circ \phi_H(a))(\f_1)=\ip{\mu(a)}{(\gamma,\delta)}.$$
We deduce that $\mu$ is carried by $a$. Thus for each $\f\in H$ we have
$$\ip{x^*}{\f(a)}=\int \f\di\mu=\ip{\mu(a)}{\f(a)}.$$
Since $H(a)=E$ (by Step 2), we deduce  $\mu(a)=x^*$, hence $\mu=\ep_a\otimes x^*$. I.e., $M_{x^*\circ\phi_H(a)}=\{\ep_a\otimes x^*\}$ and the argument is complete.

\smallskip

{\tt Step 7:} $M_{x^*\circ \phi_H(t,1)}=\{\ep_{(t,1)}\otimes x^*\}$ and  $M_{x^*\circ \phi_H(t,-1)}=\{\ep_{(t,-1)}\otimes x^*\}$ whenever $t\in[0,1]$.

\smallskip
Fix $t\in[0,1]$. The roles of $(t,1)$ and $(t,-1)$ are symmetric, so it is enough to prove the statement for $(t,1)$. Let $u:[0,1]\to[0,1]$ be continuous such that $u(t)=1$ and $u<1$ elsewhere. Define $\widetilde{u}\in C(K)$ by 
$$\widetilde{u}(s,1)= u(s), \widetilde{u}(s,-1)=-u(s) \mbox{ for }s\in [0,1],\widetilde{u}=0\mbox{ elsewhere}.$$
Further, let  $\g_t=\widetilde{u}\cdot(\gamma,\delta)$ where $(\gamma,\delta)$ is as in Step 5.
Then $\g_t\in H$, $\norm{\g_t}=1$ and it witnesses that $\norm{x^*\circ\phi_H(t,1)}=\norm{x^*}$.
So, any $\mu\in M_{x^*\circ\phi_H(t,1)}(H)$ satisfies $\norm{\mu}=\norm{x^*}$ and 
 for each $\f\in H$ we have 
 $$\alpha f_1(t,1)+\beta f_2(t,1)=\int\f\di\mu.$$
If we plug there $\g_t$, we get
$$\norm{x^*}=\int \g_t\di\mu =\int \widetilde{u}\di\mu_{(\gamma,\delta)},$$
so $\mu_{(\gamma,\delta)}$ is carried by $\{(t,1),(t,-1)\}$ and
$$\norm{x^*}=\mu_{(\gamma,\delta)}(t,1)-\mu_{(\gamma,\delta)}(t,-1).$$
We deduce that also $\mu$ is carried by $\{(t,1),(t,-1)\}$.

If we plug there function $\f_2$ defined by
$$\begin{gathered}
    \f_2=(\gamma,\delta) \mbox{ on }[0,1]\times\{1\}\cup [-1,0]\times\{2\},\\
    \f_2=(0,0) \mbox{ on }[0,1]\times\{-1\}\cup [-1,0]\times\{-2\}, \\
    \f_2=(\tfrac\gamma2,\tfrac\delta2) \mbox{ on }[-1,1]\times\{0\},\\
    \f_2(a)=\f_2(b)=(\delta,\gamma),
\end{gathered}$$
we conclude that $\mu$ is carried by $\{(t,1)\}$. Using that $H(t,1)=E$ we, similarly as in Step 5, deduce that $\mu=\ep_{(t,1)}\otimes x^*$. Thus $M_{x^*\circ\phi_H(t,1)}=\{\ep_{(t,1)\otimes x^*}\}$
which completes the argument.

\smallskip 

{\tt Step 8:} $M_{x^*\circ \phi_H(t,2)}=\{\ep_{(t,2)}\otimes x^*\}$ and  $M_{x^*\circ \phi_H(t,-2)}=\{\ep_{(t,-2)}\otimes x^*\}$ whenever $t\in[-1,0]$.

\smallskip

The proof is completely analogous to that of Step 7 -- instead of functions $u$ and $\widetilde{u}$ we use a continuous function $v:[-1,0]\to[0,1]$ with $v(t)=1$ and $v<1$ elsewhere and function $\widetilde{v}\in C(K)$ by 
$$\widetilde{v}(s,1)= v(s), \widetilde{v}(s,-1)=-v(s) \mbox{ for }s\in [-1,0],\widetilde{v}=0\mbox{ elsewhere}.$$
Further, instead of $\g_t$ we use the function $\h_t=\widetilde{v}\cdot(\gamma,\delta)$.

\smallskip

{\tt Step 9.} If $t\in (0,1]$, then 
$$\begin{aligned}
    M_{x^*\circ\phi_H(t,0)}=\{\mu\in M(K,E^*)\setsep & \mu\mbox{ is carried by }\{t\}\times\{-1,0,1\}, \mu\{(t,1)\}=\mu\{(t,-1)\}, \\& \mu\{(t,1)\}+\mu\{(t,0)\}+\mu\{(t,-1)\}=x^*, \\& \norm{\mu\{(t,1)\}}+\norm{\mu\{(t,0)\}}+\norm{\mu\{(t,-1)\}}=\norm{x^*}\}\end{aligned}$$

\smallskip

Assume $t\in (0,1]$. Let $v:[-1,1]\to[0,1]$ be continuous such that $v(t)=1$, $v(0)=0$ and $v<1$ elsewhere. Let $\widetilde{v}\in C(K)$ be defined by $\widetilde{v}(t,j)=v(t)$, $\widetilde{v}(a)=\widetilde{v}(b)=0$. Let $\f_t=\widetilde{v}\cdot(\gamma,\delta)$. Then $\f_t\in H$ and it witnesses that $\norm{x^*\circ\phi_H(t,0)}=\norm{x^*}$. 

Let $\mu\in M_{x^*\circ\phi_H(t,0)}$. Then $\norm{\mu}=\norm{x^*}$ and
for each $\f\in H$ we have 
 $$\alpha f_1(t,0)+\beta f_2(t,0)=\int\f\di\mu.$$
 If we plug there $\f_t$, as above we deduce that $\mu$ is carried by $\{t\}\times\{-1,0,1\}$.
 
 Let $x\in E$ ve arbitrary. Then the function defined by
 $$\f(s,j)=\sign j \cdot x, \f(a)=\f(b)=0$$
 belongs to $H$ and plugging it to the above formula yields
 $$0=\ip{\mu\{(t,1)\}-\mu\{(t,-1)\}}{x}.$$
 Since $x\in E$ is arbitrary, we deduce $\mu\{(t,1)\}=\mu\{(t,-1)\}$. Hence, for each $\h\in H$ we have
 $$\begin{aligned}
      \alpha h_1(t,0)+\beta h_2(t,0)&=\int \h\di\mu=\ip{\mu\{(t,1)\}}{\h(t,1)+\h(t,-1)}+\ip{\mu\{(t,0)\}}{\h(t,0)}
      \\&= \ip{2\mu\{(t,1)\}+\mu\{(t,0)\}}{\h(t,0)}
 \end{aligned}$$
Since $H(t,0)=E$ (by Step 2), we deduce that $2\mu\{(t,1)\}+\mu\{(t,0)\}=(\alpha,\beta)=x^*$.
Inclusion `$\subset$' now easily follows. The converse inclusion is clear (taking into account that
$\norm{x^*\circ\phi_H(t,0)}=\norm{x^*}$).

\smallskip

{\tt Step 10.} If $t\in [-1,0)$, then 
$$\begin{aligned}
    M_{x^*\circ\phi_H(t,0)}=\{\mu\in M(K,E^*)\setsep & \mu\mbox{ is carried by }\{t\}\times\{-2,0,2\}, \mu\{(t,2)\}=\mu\{(t,-2)\}, \\& \mu\{(t,2)\}+\mu\{(t,0)\}+\mu\{(t,-2)\}=x^*, \\& \norm{\mu\{(t,2)\}}+\norm{\mu\{(t,0)\}}+\norm{\mu\{(t,-2)\}}=\norm{x^*}\}\end{aligned}$$

\smallskip

This is completely analogous to Step 9.

\smallskip

{\tt Step 11.} $M_{x^*\circ\phi_H(0,0)}=\{\frac14(\ep_a+\ep_b)\otimes(\beta,\alpha)\}$.

\smallskip

If $\h\in H$, then 
$$x^*(\h(0,0))=\tfrac14 x^*(h_2(a)+h_2(b),h_1(a)+h_1(b))\le \tfrac12\norm{x^*}\cdot\norm{\h},$$
hence $\norm{x^*\circ \phi_H(0,0)}\le\frac12 \norm{x^*}$. The function $\h$ defined by
$$\h(a)=\h(b)=(\delta,\gamma), \h=\tfrac12(\gamma,\delta)\mbox{ elsewhere}$$
(where $(\gamma,\delta)$ is from Step 5) witnesses that $\norm{x^*\circ \phi_H(0,0)}=\frac12 \norm{x^*}$.

Let $\mu\in M_{x^*\circ\phi_H(0,0)}$. Then $\norm{\mu}=\frac12 \norm{x^*}$ and for each $\f\in H$ we have
$$\alpha f_1(0,0)+\beta f_2(0,0)=\int \f\di\mu.$$
If we plug there the above function $\h$, we get
$$\norm{\mu}=\tfrac12 \norm{x^*} = \int \h\di\mu= \ip{\mu(\{a,b\})}{(\delta,\gamma)} + \tfrac12\ip{\mu(K\setminus\{a,b\})}{(\gamma,\delta)}.$$
We deduce that $\mu$ is carried by $\{a,b\}$. 

Given $x\in E$, the function
$$\f(a)=x,\f(b)=-x, \f=(0,0) \mbox{ elsewhere}$$
belongs to $H$ and by plugging it in the above formula we get
$$0=\ip{\mu(a)-\mu(b)}{x}.$$
We deduce that $\mu(a)=\mu(b)$. Hence for each $\h\in H$ we have
$$\alpha h_1(0,0)+\beta h_2(0,0)=\ip{\mu(a)}{\h(a)+\h(b)}=4\ip{\mu(a)}{(h_2(0,0),h_1(0,0))}.$$
Since $H(0,0)=E$ (by Step 2), we deduce $\mu(a)=\frac14(\beta,\alpha)$. Now we easily conclude.

\smallskip

{\tt Step 12.} $H$ is vector simplicial and $A_c^v(H)=H$. In particular, properties $(ii)$ and $(iv)$ are fulfilled.

\smallskip

In Steps 6--11 sets of representing measures are described. In most cases there is only one representing measure, except for the cases addressed in Steps 9 and 10. But even in these cases there is only one $H$-boundary measure. Indeed, if we combine Step 9 with Step 4, we see that an $H$-boundary measure in $M_{x^*\circ\phi_H(t,0)}$ necessarily satisfies $\mu(t,0)=(0,0)$ and $\mu(t,1)=\mu(t,-1)=\frac12x^*$, so the uniqueness follows. The case from Step 10 is completely analogous. This proves the vector simpliciality, hence condition $(ii)$ is satisfied.

Next assume that $\f\in A_c^v(H)$. By Step 11 we deduce that 
$$\f(0,0)=\tfrac14((f_2(a),f_1(a))+(f_2(b),f_1(b))).$$ By Step 9 we deduce that 
$$\f(t,0)=\tfrac12(\f(t,1)+\f(t,-1))\mbox{  for }t\in (0,1].$$ Since $\f$ is continuous, the equality holds also for $t=0$. Similarly,
by Step 10 we get 
$$\f(t,0)=\tfrac12(\f(t,2)+\f(t,-2))\mbox{ for }t\in [-1,0)$$ and by continuity of $\f$ this extends to $t=0$. Hence $f\in H$. Thus $A_c^v(H)=H$. Using moreover Step 4 we conclude that condition $(iv)$ is satisfied.
\end{proof}

We point out that a large part of properties of the above examples is connected with absence of
constants in $H$ (although in these examples $H_w$ always contains constants). Properties of function spaces containing constants are analyzed in Section~\ref{sec:skonstantami} below.

\section{Dilation operators for weakly simplicial spaces}\label{sec:dilation}

Let $H\subset C(K,E)$ be a function space separating points of $K$.
Assume that $H$ is weakly simplicial, i.e., $H_w$ is simplicial.
For $t\in K$ let $\delta_t$ denote the unique $H$-boundary measure in $M_t(H)=M_t(H_w)$. We denote by $D$ the induced dilation mapping,   i.e.,
    $$Df(t)=\int f\di\delta_t,\quad t\in K, f\in C(K).$$
This operator is a standard tool in the classical Choquet theory. In  \cite[Section 6.1]{lmns} (where it is denoted by $T$) basic properties are established for scalar real function spaces containing constants. In \cite{bezkonstant} this operator was studied for function spaces not containing constants, where some strange behavior may appear.  

In this section we will assume that $H_w$ contains constants. Then each $\delta_t$ is a probability measure. Moreover, the real and complex theories essentially coincide. The reason is that the representing measures or boundary measures considered with respect to $H_w$ or $A_c(H_w)$ coincide. Further, if $\ef=\ce$, then $A_c(H_w)$ is a self-adjoint subspace, hence the representing measures or boundary measures considered with respect to $A_c(H_w)$ or $A_c(H_w)\cap C(K,\er)$ coincide. Therefore the results established in \cite[Section 6.1]{lmns} in the real case hold in the complex case equally.

    In particular, for each $f\in C(K)$ the function $Df$ is Borel (by \cite[Theorem 6.8(b)]{lmns}). Further, given $\mu\in M(K)$, there is a unique measure $D\mu\in M(K)$ satisfying
    $$\int f\di D\mu=\int Df\di\mu,\quad f\in C(K).$$
    This measure is moreover $H$-boundary (see \cite[Theorem 6.11]{lmns}).

Let us define a vector-valued variant. For $\f\in C(K,E)$ define
\[
\Db \f(t)=(B)\int \f\di\delta_t,\quad t\in K.
\]
It is clear that $\Db\f$ is a well-defined function $K\to E$. Let us further recall that  $I(K,E)$ is the space of bounded functions from $K$ to $E$ integrable with respect of all measures in $M(K,E^*)$, i.e., the uniform closure of simple Borel functions.

\begin{lemma}
\label{l:vlast-depe}
Let a function space $H\subset C(K,E)$ separate points of $K$. Assume that $H$ is weakly simplicial and $H_w$ contains constants.
Then the following assertions are valid.
\begin{enumerate}[$(1)$]
    \item $\Db$ is a norm-one linear operator from $C(K,E)$ to $I(K,E)$ (equipped with the sup-norm).
       \item $\Db\f=\f$ on $\Ch_HK$ whenever $\f\in C(K,E)$.
    \item For each measure $\mu\in M(K,E^*)$ there exists a unique $\Db\mu\in M(K,E^*)$ such that $\int \f\di \Db\mu=\int \Db\f \di\mu$ for $\f\in C(K,E)$. Moreover, $\norm{\Db\mu}\le\norm{\mu}$.
    \item The measure $\Db\mu$ is $H$-boundary for each $\mu\in M(K,E^*)$.
    \item If $\mu\in M(K,E^*)$, then $\Db\mu=\mu$ if and only if $\mu$ is $H$-boundary. In particular, $\Db\circ \Db=\Db$ on $M(K,E^*)$.
    \item $A_c^w(H)=\{\f\in C(K,E)\setsep \Db \f=\f\}$.
\end{enumerate}
\end{lemma}

\begin{proof}
$(1)$ Given $\f\in C(K,E)$ and $t\in K$, we have
\[
\norm{\Db \f (t)}=\norm{(B)\int \f \di\delta_t}\le \int\norm{\f(s)}\di\delta_t(s)\le \norm{\f}.
\]
Thus $\Db\f$ is a bounded function. Since the operator $\Db$ is clearly linear, $\Db$ is a norm-one operator $C(K,E)\to\ell^\infty(K,E)$.

Further, if $\f=f\cdot x$ for some $f\in C(K)$ and $x\in E$, then 
\[
\Db \f(t)=(B)\int f(s)\cdot x\di\delta_t(s)=Df(t)\cdot x.
\]
Since $Df$ is Borel, we deduce $\Db\f\in I(K,E)$. Since functions of the form $f\cdot x$, $f\in C(K)$, $x\in E$, are linearly dense in $C(K,E)$, the assertion follows.

$(2)$ This follows from the obvious fact that $\delta_t=\ep_t$ for $t\in\Ch_HK$.

$(3)$ The mapping $\f \mapsto \int\Db f\di\mu$ is clearly a linear  functional on $C(K,E)$. Moreover,
$$\abs{\int\Db \f\di\mu}\le\norm{\Db \f}\norm{\mu}\le\norm{\f}\norm{\mu},$$
so the functional has norm at most $\norm{\mu}$.
Hence there exists a unique measure $\Db\mu\in M(K,E^*)$ representing this functional. It satisfies $\norm{\Db\mu}\le\norm{\mu}$.

$(4)$ For every $x\in E$ and $f\in C(K)$ we have 
\[\begin{aligned}
 \int f \di (\Db\mu)_x&=\int f\cdot x\di \Db\mu=\int \Db (f\cdot x)\di\mu=\int (Df \cdot x)\di\mu=\int Df\di\mu_x\\&=\int f\di (D\mu_x).
\end{aligned}\]
Thus $(\Db\mu)_x=D\mu_x$ is a $H$-boundary measure for each $x\in E$. 

It follows that $\Db\mu$ is $H$-boundary. Indeed, this follows easily from the Mokobodzki test \cite[Theorem 3.58]{lmns}: Given $f\in C(K,\er)$, denote by $\widehat{f}$ its upper envelope, i.e.,
$$\widehat{f}=\inf\{ g\in A_c(H_w)\cap C(K,\er)\setsep g\ge f\}.$$
Assume that $\Db\mu$ is not $H$-boundary, i.e., $\abs{\Db\mu}$ is not $H$-boundary. By the quoted Mokobodzki test there is $f\in C(K,\er)$ such that $\abs{\Db\mu}[\widehat{f}>f]>0$, By the definition of variation there is a Borel set $A\subset [\widehat{f}>f]$ with $\norm{\Db\mu(A)}>0$. Then there is some $x\in E$ with $\Db\mu(A)(x)\ne0$. It follows that $\abs{(\Db\mu)_x}[\widehat{f}>f]>0$, so $(\Db\mu)_x$ is not $H$-boundary (using again the Mokobodzki test).
This contradicts the previous paragraph.

$(5)$ If $\Db\mu=\mu$, then $\mu$ is $H$-boundary by $(4)$. Conversely, assume that $\mu$ is $H$-boundary and $x\in E$. Then clearly $\mu_x$ is also $H$-boundary. It follows from \cite[Theorem 6.11(a4)]{lmns} that $D(\mu_x)=\mu_x$. Within the proof of $(4)$ we showed that $(\Db\mu)_x=D(\mu_x)$.
We deduce $(\Db\mu)_x=\mu_x$. Since $x\in E$ is arbitrary, we conclude that $\Db\mu=\mu$. 
The `in particular part' then follows easily.

$(6)$ Inclusion `$\subset$' is obvious.
Conversely, let $\f\in C(K,E)$ with $\Db \f=\f$ and $t\in K$ and $\mu\in M_t(H)$ be given. Then $D\mu\in M_t(H)$ is $H_w$-boundary and thus $D\mu=\delta_t$ (by \cite[Theorem 6.11(a)]{lmns}).
Then, given $\g\in C(K,E)$ and $x^*\in E^*$ we get 
\[
\begin{aligned}
x^*\left((B)\int \g\di D\mu\right)&=\int x^*\circ \g\di D\mu=\int D(x^*\circ \g)\di\mu=\int x^*\circ \Db \g\di\mu\\
&=x^*\left((B)\int \Db\g\di\mu\right),
\end{aligned}
\]
and thus $(B)\int \g\di D\mu=(B)\int \Db \g \di\mu$ for each $\g \in C(K,E)$.
Hence 
\[
\f(t)=(B)\int \f\di\delta_t=(B)\int\f \di D\mu=(B)\int \Db \f\di\mu=(B)\int \f\di\mu.
\]
Therefore $\f\in A_c^w(H)$.
\end{proof}

We further define another variant of the operator $\Db$.
 Given $\f\in C(K,E)$, we define 
  $$\Tb\f(t,x^*)=\int \f\di(\delta_t\otimes x^*),\quad(t,x^*)\in K\times B_{E^*}.$$ 
Its properties are summarized in the following lemma.

\begin{lemma}\label{L:vlast-Tb}
    Let a function space $H\subset C(K,E)$ separate points of $K$. Assume that $H$ is weakly simplicial and $H_w$ contains constants.
Then the following assertions are valid.
\begin{enumerate}[$(a)$]
    \item $\Tb\f(t,x^*)=x^*(\Db\f(t))=D(x^*\circ\f)(t)$ for  $\f\in C(K,E)$ and $(t,x^*)\in K\times B_{E^*}$.
    \item $\Tb$ is a linear operator from $C(K,E)$ into the space of bounded Borel functions on $K\times B_{E^*}$ satisfying $\norm{\Tb}\le1$.
    \item $A_c^w(H)=\{\f\in C(K,E)\setsep \Tb\f=T\f\}$.
    \item If $\nu\in M(K\times B_{E^*})$, then
    $$\int \Tb\f\di\nu=\int \f\di \Db(T^*\nu),\quad \f\in C(K,E).$$
 \end{enumerate}
\end{lemma}
 
\begin{proof}
 $(a)$: We start by observing that $\delta_t\otimes x^*=T^*(\delta_t\times \ep_{x^*})$. Hence
 $$\begin{aligned}     
\Tb\f(t,x^*)&=\int_K \f\di(\delta_t\otimes x^*)=\int_K \f\di T^*(\delta_t\times x^*)= \int_{K\times B_{E^*}} T\f\di(\delta_t\times \ep_{x^*})\\& =\int_K T\f(s,x^*)\di\delta_t(s)
 =\int_K x^*\circ\f\di\delta_t=x^*\left((B)\int_K \f\di\delta_t\right). \end{aligned}$$
 Since the last expression is $x^*(\Db\f(t))$ and the previous one is $D(x^*\circ\f)(t)$, the argument is complete.

$(b)$: Using $(a)$ and Lemma~\ref{l:vlast-depe}$(1)$ we get
$$\abs{\Tb\f(t,x^*)}=\abs{x^*(\Db\f(t))}\le \norm{\Db\f(t)}\le \norm{\Db\f}\le\norm{\f},$$
so $\norm{\Tb\f}\le\norm{\f}$. Since the linearity of $\Tb$ is obvious, we get that $\Tb$ is a norm-one linear operator from $C(K,E)$ into $\ell^\infty(K\times B_{E^*})$. 

Further, if $\f=f\cdot x$ for some $x\in E$ and $f\in C(K)$, by $(a)$ we deduce
$$\Tb\f(t,x^*)=D(x^*\circ\f)(t)=D(x^*(x)\cdot f)(t)=x^*(x)\cdot Df(t).$$
Since $Df$ is Borel and $x^*\mapsto x^*(x)$ is continuous, we deduce that $\Tb\f$ is Borel.  Since functions of this kind are linearly dense in $C(K,E)$,  the argument is complete.

$(c)$: For $\f\in C(K,E)$ we have
\begin{multline*}
    \Tb\f=T\f \iff \forall x^*\in B_{E^*}\; \forall t\in K\colon x^*(\Db\f(t))=x^*(\f(t))\\
    \iff \forall t\in K\colon \Db\f(t)=\f(t)\iff\Db\f=\f,
\end{multline*}
so the assertion follows from Lemma~\ref{l:vlast-depe}$(6)$.

$(d)$: Due to $(b)$ it is enough to prove the equality for $\f=f\cdot x$ where $f\in C(K)$ and $x\in E$ (as functions of this form are linearly dense in $C(K,E)$). For such functions we have
$$\begin{aligned}
  \int_K \f\di \Db T^*\nu&=\int_K f\di (\Db T^*\nu)_x  =\int_K f\di D((T^*\nu)_x)=\int_K Df\di (T^*\nu)_x\\&=\int Df(t)\cdot x^*(x)\di \nu(t,x^*)=\int \Tb\f\di\nu.
\end{aligned}$$
The first equality follows from the definition of $(\Db T^*\nu)_x$, the second one from the equality $(\Db T^*\nu)_x=D((T^*\nu)_x)$ proved in Lemma~\ref{l:vlast-depe}$(4)$. The third equality follows from the properties of scalar dilation operators. The fourth equality follows from the fact that $(T^*\nu)_x$ is the image of the measure on $K\times B_{E^*}$ with density $(t,x^*)\mapsto x^*(x)$ with respect to $\nu$ under the projection to $K$ (see, e.g., \cite[formula (2.2)]{transference-studia}). The last equality follows from the special form of $\f$.
\end{proof}

\subsection{The Dirichlet problem for simplicial function spaces}
\label{ss:dirichlet}

One of the main applications of Choquet theory is the investigation of abstract Dirichlet problem,
i.e., of the possibility to extend functions defined on the Choquet boundary with preserving some properties. In this subsection we present some results in this direction for vector-valued functions in $A_c^w(H)$. To this end we use some abstract results from \cite{spurny-archiv}. These results deal with continuous functions and functions of the first affine class.
To formulate them we have to introduce a piece of notation. Given a function space $H\subset C(K,E)$, we write $H_1$ for the space of all pointwise limits of sequences contained in $H$.

The first result is a function-space variant of \cite[Theorem 1.1]{spurny-archiv} which is a vector-valued version of the main result of \cite{Jel}. 

\begin{thm}
    \label{t:dirich-lindel}
    Let $H\subset C(K,E)$ be a function space separating points of $K$ such that $H$ is weakly simplicial and $H_w$ contains constants. Assume that $\Ch_HK$ is \lin. If $\f\colon\Ch_HK\to E$ is a bounded continuous function, then there exists a function $\h\in (A_c^w(H))_1$ with $\norm{\h}_\infty=\norm{\f}_\infty$ such that $\h=\f$ on $\Ch_H K$.
    Moreover, the approximating sequence for $\h$ can be taken to be bounded.
\end{thm}

\begin{proof}
Let $X=\es(A_c(H_w))$ be the state space of $A_c(H_w)$. Then $\Ch_HK=\Ch_{H_w}K$ and $\phi=\phi_{A_c^w(H)}$ maps $\Ch_H K=\Ch_{H_w} K$ onto $\ext X$ (see Theorem~\ref{T:hranice}). Further,  $X$ is a simplex (see \cite[Theorem 6.54]{lmns}) and hence we can apply \cite[Theorem 1.1]{spurny-archiv} for the function $\g=\f\circ \phi^{-1}\colon \ext X\to E$ to obtain a function $\widetilde{\g}\in (\fra(X,E))_1$ extending $\g$. Moreover, by the same theorem the approximating sequence $\widetilde{\g_n}\in \fra(X,E)$ can be taken to be bounded. Moreover, $\norm{\widetilde{\g}}_\infty=\norm{\f}_\infty$. Indeed, for each $x^*\in B_{E^*}$ the function $x^*\circ\widetilde{\g}$ is an affine Baire-one function on $X$ and hence, by the minimum principle \cite[Corollary 4.23]{lmns} we have
$$\norm{x^*\circ\widetilde{\g}}_\infty=\norm{x^*\circ\g}_\infty\le\norm{\g}_\infty=\norm{\f}_\infty.$$
Hence, a consequence to the Hahn-Banach theorem yields $\norm{\widetilde{\g}}_\infty\le\norm{\f}_\infty$. The converse inequality is obvious. To conclude we set $\h_n=\widetilde{\g_n}\circ\phi$ and $\h=\widetilde{\g}\circ\phi$. Then $\h_n\in A_c^w(H)$ by Proposition~\ref{p:acwh-aese} and $\h_n\to \h$. Hence $\h$ is the desired extension.
\end{proof}

We continue by a more precise vector-valued version of \cite[Theorem 6.8(a)]{lmns}.

\begin{thm}
    \label{t:dirich-reseni}
    Let $H\subset C(K,E)$ be a function space separating points of $K$ such that $H$ is weakly simplicial and $H_w$ contains constants. If $K$ is metrizable and $\f \in C(K,E)$, then $\Db \f\in (A_c^w(H))_1$, where the approximating sequence can be taken to be bounded.
\end{thm}

\begin{proof}
Using Theorem~\ref{t:dirich-lindel} we find a function $\h\in (A_c^w(H))_1$ extending $\f|_{\Ch_H K}$. Let $(\h_n)\subset A_c^w(H)$ be a bounded approximating sequence for $\h$. Then for every $t\in K$, $\delta_t$ is carried by $\Ch_H K=\Ch_{H_w} K$ and hence
\[
\begin{aligned}
\Db \f(t)&=(B)\int_{\Ch_H K} \f\di\delta_t=(B)\int_{\Ch_H K}(\lim_{n\to \infty} \h_n)\di\delta_t\\
&=\lim_{n\to\infty} (B)\int_{\Ch_H K} \h_n\di\delta_t=\lim_{n\to\infty} \h_n(t)=\h(t),
\end{aligned}
\]
where the third equality follows from the Lebesgue dominated convergence theorem. Hence $\Db \f=\h\in (A_c^w(H))_1$.
\end{proof}

We finish this subsection by a vector-valued variant of characterization of Bauer simplicial spaces 
(cf. \cite[Section 6.6.A]{lmns}).

\begin{thm}
    \label{t:dirich-closed}
    Let $H\subset C(K,E)$ be a function space separating points of $K$ such that $H$ is weakly simplicial and $H_w$ contains constants. If $\Ch_HK$ is closed, then the restriction map
    $$R\colon\f\mapsto \f|_{\Ch_HK}$$
    is a surjective linear isometry of $A_c^w(H)$ onto $C(\Ch_HK,E)$.
\end{thm}

\begin{proof}
    It is clear that $R$ is linear and that $\norm{R}\le 1$. Further, if $\f\in A_c^w(H)$ and $t\in K$, then $\delta_t$ is carried by $\Ch_HK$ and
    $$\f(t)=(B)\int \f\di\delta_t,$$
    hence we easily deduce that $\norm{\f}\le\norm{R\f}$, so $R$ is an isometry.

   It remains to prove it is surjective. To prove we use the pattern of the proof of Theorem~\ref{t:dirich-lindel}. Let $X$ be as in the quoted proof and let $\phi=\phi_{A_c^w(H)}$. Fix $\f\in C(\Ch_HK,E)$. Then
   $\g=\f\circ\phi^{-1}$ belongs to $C(\ext X,E)$. Since $X$ is a simplex and $\ext X$ is closed,
   \cite[Theorem 2.5]{rondos-spurny} yields $\widetilde{\g}\in \fra(X,E)$ extending $\g$. Then $\widetilde{\f}=\widetilde{\g}\circ\phi$ belongs to $A_c^w(H)$ (by Proposition~\ref{p:acwh-aese}).
   It remains to observe that $\f=R\widetilde{\f}$. This completes the proof.
\end{proof}

\section{Simpliciality of function spaces containing constant functions}\label{sec:skonstantami}

In Section~\ref{sec:reprez} we introduced four variants of simpliciality of vector-valued function spaces. These variants go in two basic directions which were compared in Section~\ref{sec:H-aff}. In particular, some examples showed that these two directions are incomparable. In the current section we will show that for spaces containing constants the situation is simpler and quite interesting.
We start by a lemma which points out some specific properties of function spaces containing constants.

\begin{lemma}\label{L:skonstantami}
     Let $H\subset C(K,E)$ be a  function space separating points of $K$ and containing constants.
     Then the following assertions are valid.
     \begin{enumerate}[$(i)$]
         \item $\norm{x^*\circ\phi_H(t)}=\norm{x^*}$ whenever $t\in K$ and $x^*\in E^*$.
         \item Let $t\in K$ and $x^*\in E^*$. If $\mu\in M_t(H)$, then $\mu\otimes x^*\in M_{x^*\circ\phi_H(t)}(H)$.
         \item $A_c^v(H)\subset A_c^w(H)$.
         \item $H$-boundary measures coincide with $A_c^v(H)$-boundary measures.
     \end{enumerate}
\end{lemma}

\begin{proof}
    $(i)$: This follows from Lemma~\ref{L:normyevaluaci}$(a)$.

    $(ii)$: Assume $\mu\in M_t(H)$ and $x^*\in E^*$. Since $H$ contains constants, we deduce $\norm{\mu}=1$ (in fact, $\mu\in M_1(K)$). So, $\norm{\mu\otimes x^*}=\norm{x^*}=\norm{x^*\circ\phi_H(t)}$. Moreover, for each $\h\in H$ we have
$$x^*(\h(t))=x^*\left((B)\int \h\di\mu\right)=\int x^*\circ\h\di\mu=\int T\h \di \mu\times \ep_{x^*}=\int \h\di T^*(\mu\times \ep_{x^*})$$
and 
$$T^*(\mu\times \ep_{x^*})(A)(x)=\int_{A\times B_{E^*}} y^*(x)\di(\mu\times \ep_{x^*})(t,y^*)=\int_A x^*(x)\di\mu=\mu(A)\cdot x^*(x),$$
i.e., $T^*(\mu\times \ep_{x^*})=\mu\otimes x^*$. Hence
$$x^*(\h(t))=\int \h\di (\mu\otimes x^*).$$
This completes the argument.

$(iii)$: Let $\f\in A_c^v(H)$. Let $t\in K$ and $\mu\in M_t(H)$. By $(ii)$ we know that
$\mu\otimes x^*$ belongs to $M_{x^*\circ\phi_H(t)}(H)$ for each $x^*\in E^*$. Hence, for each $x^*\in E^*$ we have
$$x^*(\f(t))=\int \f \di \mu\otimes x^* =\int x^*\circ\f\di\mu=x^*\left((B)\int \f\di\mu\right).$$
Since $x^*\in E^*$ is arbitrary, we deduce
$\f(t)=(B)\int \f\di\mu$. This completes the proof that $\f\in A_c^w(H)$.

$(iv)$: Since $H\subset A_c^v(H)\subset A_c^w(H)$ (by $(iii)$), we have
$$H_w\subset A_c^v(H)_w\subset A_c^w(H)_w=A_c(H_w),$$
where the last equality follows from Proposition~\ref{P:AcH vlastnosti}$(b)$.  It follows that
$A_c(A_c^v(H)_w)=A_c(H_w)$, thus for a measure $\mu\in M(K)$ we have
$$\begin{gathered}
    \mu\mbox{ is $H$-boundary} \iff \mu\mbox{ is $H_w$-boundary} 
\iff \mu\mbox{ is $A_c(H_w)$-boundary} \\ \iff \mu\mbox{ is $A_c^v(H)_w$-boundary} 
\iff \mu\mbox{ is $A_c^v(H)$-boundary}.\end{gathered}$$
Indeed, the first and the last equivalences follow from the definitions (using Theorem~\ref{T:H-maximal}). The second and the third equivalences follow from \cite[Proposition 3.67]{lmns}
(or \cite[Proposition 4.5]{bezkonstant}).
\end{proof}

The main result of this section is the following theorem. It may be viewed as a vector-valued variant of \cite[Proposition 5.4]{bezkonstant}.

\begin{thm}\label{T:ruznesimpl}
    Let $H\subset C(K,E)$ be a function space separating points of $K$ and containing constants.
    Then the following assertions are equivalent.
    \begin{enumerate}[$(1)$]
        \item $H$ is vector simplicial.
        \item $A_c^v(H)$ is vector simplicial.
        \item $A_c^v(H)$ is functionally vector simplicial.
        \item $H$ is weakly simplicial and $A_c^v(H)=A_c^w(H)$.
        \item The only $H$-boundary measure in $A_c^v(H)^\perp$ is the zero measure.
    \end{enumerate}
\end{thm}

\begin{proof}
    $(3)\implies(2)$: This is trivial.

    $(2)\iff(1)$: This follows by combining Lemma~\ref{L:AcvH}$(d)$ and Lemma~\ref{L:skonstantami}$(iv)$.

    $(1)\implies(4)$: Assume $H$ is vector simplicial. We first prove that $H$ is weakly simplicial.
    We proceed by contradition. Assume $H$ is not weakly simplicial. Then there is $t\in K$ and two distinct $H$-boundary measures $\mu_1,\mu_2\in M_t(H)$. Let $x^*\in E^*$ be nonzero. Then $\mu_1\otimes x^*,\mu_2\otimes x^*$ are two distinct boundary measures from $M_{x^*\circ\phi_H(t)}(H)$ (by Lemma~\ref{L:skonstantami}$(ii)$), so $H$ is not vector simplicial, a contradiction.

    Next we will show that $A_c^v(H)=A_c^w(H)$.  Inclusion `$\subset$' follows from Lemma~\ref{L:skonstantami}$(iii)$. To prove the converse fix $\f\in A_c^w(H)$, $t\in H$, $x^*\in E^*$ and $\mu\in M_{x^*\circ\phi_H(t)}(H)$. Let $\nu\in M(K\times B_{E^*})$ be such that $\norm{\nu}=\norm{\mu}$ and $T^*\nu=\mu$.

If $\g\in A_c^w(H)$ is arbitrary, then (using the notation from Section~\ref{sec:dilation})
$$\int \g\di\mu=\int T\g\di\nu=\int \Tb\g\di\nu=\int \g\di\Db\mu.$$
Indeed, the first equality follows from the choice of $\nu$. The second one follows from Lemma~\ref{L:vlast-Tb}$(c)$ and the third one from Lemma~\ref{L:vlast-Tb}$(d)$. 

If additionally $\g\in A_c^v(H)$, we deduce
$$\int \g\di\Db\mu=\int\g\di\mu=x^*(\g(t)).$$
Since  $\norm{\Db\mu}\le\norm{\mu}$ by Lemma~\ref{l:vlast-depe}$(3)$,
we conclude  that $\Db\mu\in  M_{x^*\circ\phi_H(t)}(H)$.
Further, $\Db\mu$ is an $H$-boundary measure by Lemma~\ref{l:vlast-depe}$(4)$ and hence by  vector simpliciality we deduce $\Db\mu=\delta_t\otimes x^*$.

Thus, applying the above equality to $\f$, we deduce that
$$\int\f\di\mu=\int \f\di\Db\mu=\int \f\di(\delta_t\otimes x^*)=
\int x^*\circ\f\di\delta_t=x^*(\f(t)).$$

Thus $\f\in A_c^v(H)$ and the proof is complete.

\smallskip

$(4)\implies(5)$:  Assume $(4)$ holds and fix an $H$-boundary measure $\mu\in A_c^v(H)^\perp$. For $x\in E$ and $f\in A_c(H_w)$ we have (due to Proposition~\ref{P:AcH vlastnosti}$(a)$) $f\cdot x\in A_c^w(H)=A_c^v(H)$, so
$$\int f\di\mu_x=\int f\cdot x\di\mu=0.$$
Hence $\mu_x\in A_c(H_w)^\perp$. Since $\mu_x$ is $H$-boundary and $A_c(H_w)$ is simplicial, we deduce $\mu_x=0$ (by  \cite[Proposition 5.4]{bezkonstant}). Since $x\in E$ is arbitrary, we conclude $\mu=0$. This completes the argument.

$(5)\implies(3)$: Assume $(5)$ holds and fix $\varphi\in A_c^v(H)^*$. Let $\mu_1,\mu_2\in M_\varphi(A_c^v(H))$ be $H$-boundary. Then $\mu_1-\mu_2\in A_c^v(H)^\perp$, so $\mu_1=\mu_2$. This completes the proof.
\end{proof}

We continue by a corollary on properties of $A_c^w(H)$ assuming just that $H_w$ contains constants.

\begin{cor}\label{cor:weaKsimpl}
    Let $H\subset C(K,E)$ be a function space separating points of $K$ such that $H_w$  contains constants.
    Then the following assertions are equivalent.
    \begin{enumerate}[$(1)$]
        \item $H$ is weakly simplicial.
        \item $A_c^w(H)$ is weakly simplicial.
        \item $A_c^w(H)$ is functionally weakly simplicial.
        \item $A_c^w(H)$ is vector simplicial.
        \item $A_c^w(H)$ is functionally vector simplicial.
        \item The only $H$-boundary measure in $A_c^w(H)^\perp$ is the zero measure.
    \end{enumerate}
\end{cor}

\begin{proof}
  Equivalence $(1)\iff(2)$ follows from Proposition~\ref{P:weaksimp}.

$(2)\iff(3)$: This follows by the following chain of equivalences:

\begin{multline*}
   A_c^w(H) \mbox{ is weakly simplicial} \iff A_c(H_w) \mbox{ is simplicial}
   \\ \iff A_c(H_w) \mbox{ is functionally simplicial} \\ \iff  A_c^w(H) \mbox{ is functionally weakly simplicial.}
\end{multline*}

The first and the third equivalences follow from Observation~\ref{obs:weak simpl} and Proposition~\ref{P:AcH vlastnosti}$(b)$. The second equivalence is a well-known classical fact,  it is proved for example in \cite[Proposition 5.4]{bezkonstant}.

  To prove the remaining equivalences set $H^\prime=A_c^w(H)$. Since $H_w$ contains constants, by Proposition~\ref{P:AcH vlastnosti}$(b)$ we know that $H^\prime$ contains constants. By Lemma~\ref{L:AcH}$(e)$ and Lemma~\ref{L:skonstantami}$(iii)$ we deduce $A_c^w(H^\prime)=A_c^v(H^\prime)=H^\prime$. Moreover, $H$-boundary and $H^\prime$-boundary measures coincide by Proposition~\ref{P:AcwHboundary}. Therefore, assertions $(2)$ and $(4)-(6)$ are equivalent by Theorem~\ref{T:ruznesimpl} applied to $H^\prime$.
 \end{proof}

The statement of \cite[Proposition 5.4]{bezkonstant} contains one more equivalent condition formulated using the notion of an $L^1$-predual. Recall that a (real or complex) Banach space is an $L^1$-predual if its dual is isometric to an $L^1$-space. In the vector setting the structure of $E$ plays role. More precisely, we have the following result.

\begin{thm}\label{T:L1pred}
    Let $H\subset C(K,E)$ be a function space separating points such that $H_w$ contains constants. Then the following assertions are equivalent.
    \begin{enumerate}[$(1)$]
        \item $A_c^w(H)$ is an $L^1$-predual.
        \item $H$ is weakly simplicial and $E$ is an $L^1$-predual.
    \end{enumerate}
\end{thm}

\begin{proof}
$(1)\implies (2)$:
Assume that $A_c^w(H)$ is an $L_1$-predual. Since $H_w$ contains constants, we know that $A_c^w(H)$ contains constants, so it contains an isometric copy of $E$ (formed by constant functions). Moreover, this subspace is $1$-complemented. Indeed, fix any $t_0\in K$. Then the assignment $\f\mapsto \f(t_0)$ (considered as a constant function)
is a norm-one projection. Thus  $E$ is an $L_1$-predual by \cite[\S 17, Theorem 3, p. 162]{lacey}.

Next we are going to show that $A_c(H_w)$ is an $L^1$-predual. To this end it is enough to show that $A_c(H_w)$ has the 4.2.I.P in case of the real structure or $A_c(H_w)$ has the 4.3.I.P in case of the complex structure (see \cite[\S 21, Theorem 6]{lacey} and \cite{lima1976complex}).

Let us first focus on the real case. Let $\{B(f_1,r_1),\dots, B(f_4,r_4)\}$ be a collection of four closed balls in $A_c(H_w)$ such that each two of them intersect.  Find $x\in S_E$ and consider functions $\h_i=f_i\cdot x\in A_c^w(H)$. Then each pair of the balls $\{B(\h_1,r_1),\dots, B(\h_4,r_4)\}$ intersects because if $f\in B(f_i,r_i)\cap B(f_j,r_j)$, then $f\cdot x\in B(\h_i,r_i)\cap B(\h_j,r_j)$. Using the 4.2.I.P in $A_c^w(H)$ we find $\h\in \bigcap_{i=1}^4 B(\h_i,r_i)$. Let $x^*\in S_{E^*}$ be such that $x^*(x)=1$. Then the function $f\colon t\mapsto (x^*\circ \h)(t)$ belongs to $A_c^w(H)_w=A_c(H_w)$. Further, 
\[
\begin{aligned}
\norm{f-f_i}&=\sup_{t\in K}\abs{x^*(\h(t))-x^*(x)f_i(t)}
=\sup_{t\in K}\abs{x^*(\h(t))-x^*((f_i\cdot x)(t))} \\
&\le \sup_{t\in K}\norm{\h(t)-(f_i\cdot x)(t)}=\norm{\h-\h_i}\le r_i,\quad i=1,\dots,4.
\end{aligned}
\]
Hence $A_c(H_w)$ has the 4.2.I.P.
Similarly we treat the complex case.

So, $A_c(H_w)$ is indeed an $L^1$-predual, hence $H_w$ is simplicial
(see, e.g., \cite[Proposition 5.4]{bezkonstant}), which completes the argument.

\smallskip

$(2)\implies (1)$: The first step is to prove that $C(K,E)$ is an $L^1$-predual. To this end we note that $C(K,E)$ is canonically identified with the injective tensor product of $C(K)$ and $E$, see \cite[Section 3.2]{ryan-tensor}. Since both $C(K)$ and $E$ are $L_1$-preduals (equivalently, $\L_{\infty,1+\ep}$-spaces for each $\ep>0$), $C(K,E)$ is also an $L_1$-predual (cf. \cite[Exercise 3.9]{ryan-tensor}). (In the complex case an alternative proof is given in \cite[Corollary 3.2]{rao-tensor}.) 

Since $H$ is weakly simplicial, we may consider mapping 
$\Db\colon M(K,E^*)\to M(K,E^*)$ from Lemma~\ref{l:vlast-depe}. By the quoted lemma we know that $\Db$ is a norm-one projection and its range is the space $M_{bnd}(K,E^*)$ of $H$-boundary measures from $M(K,E^*)$. 

 Consider the mapping $I\colon M_{bnd}(K,E^*)\to (A_c^w(H))^*$ defined by $I\mu=\mu|_{A_c^w(H)}$. It is clear that $I$ is linear and $\norm{I}\le1$. It follows from Corollary~\ref{cor:weaKsimpl}, implication $(1)\implies(6)$, that $I$ is one-to-one. Further, given $\varphi\in (A_c^w(H))^*$, by Hahn-Banach extension theorem there is $\mu\in M(K,E^*)$ with $\norm{\mu}=\norm{\varphi}$ such that $\mu|_{A_c^w(H)}=\varphi$. If $\f\in A_c^w(H)$, then
 $$\int \f\di \Db\mu=\int \Db\f\di\mu=\int \f\di\mu=\varphi(\f).$$
Hence $I(\Db\mu)=\varphi$. Moreover,
$$\norm{\varphi}\le\norm{\Db\mu}\le\norm{\mu}=\norm{\varphi},$$
so the equality holds. Thus $I$ is an onto isometry.
 
Now we are ready to complete the proof.
By the first paragraph we know that $M(K,E^*)$ is an $L_1$-space.
By \cite[Theorem 3, p. 162]{lacey} the space $M_{bnd}(K,E^*)$ is an $L_1$-space as well (being $1$-complemented in $M(K,E^*)$). Hence, $(A_c^w(H))^*$, being isometric to $M_{bnd}(K,E^*)$ is an $L^1$-space, too. I.e., $A_c^w(H)$ is an $L^1$-predual and the proof is complete.
\end{proof}

The following example (together with Theorem~\ref{T:ruznesimpl}) show that, for function spaces containing constants, vector simpliciality is strictly stronger than weak simpliciality.

\begin{example}\label{ex:wsnevs-const}
    There is a compact space $K$, a Banach space $E$ and a  function space $H\subset C(K,E)$ containing constants and separating points with the following properties.
    \begin{enumerate}[$(i)$]
        \item $H$ is weakly simplicial;
        \item $H$ is not vector simplicial;
        \item $A_c^v(H)\subsetneqq A_c^w(H)$.
    \end{enumerate}
\end{example}

\begin{proof}
     Let $K=\{0,1\}$, $E=(\ef^2,\norm{\cdot}_p)$ for some $p\in[1,\infty]$ and 
$$H=\{ \h=(h_1,h_2):K\to E\setsep h_2\mbox{ is constant }\}.$$
Then $H$ clearly contains constants and separates points of $K$. The proof continues in several steps.

\smallskip

{\tt Step 1:} $H_w=C(K)$ and hence $\Ch_HK=K$, $H$ is weakly simplicial and $A_c^w(H)=C(K,E)$.

\smallskip

Indeed, if $f\in C(K)$, we define $\h\in C(K,E)$ by setting $h_1=f$ and $h_2=0$. Then $\h\in H$ and $f=e_1^*\circ\h$, so $f\in H_w$. This proves that $H_w=C(K)$. The remaining statements now follow easily.

\smallskip

{\tt Step 2:} $M_{e_2^*\circ\phi_H(0)}(H)$ contains both $\ep_0\otimes e_2^*$ and $\ep_1\otimes e_2^*$, hence $H$ is not vector simplicial.

\smallskip

Indeed, $\norm{e_2^*\circ\phi_H(0)}=1$ by Lemma~\ref{L:normyevaluaci}. Both suggested measures have norm one and, moreover, for each $\h\in H$ we have
$$\int \h\di (\ep_0\otimes e_2^*)=h_2(0)=h_2(1)=\int \h\di(\ep_1\otimes e_2^*).$$
This proves that both measures belong to $M_{e_2^*\circ\phi_H(0)}(H)$. Since they are distinct and due to Step 1 they are $H$-boundary, we conclude that $H$ is not vector simplicial.

\smallskip

{\tt Step 3:} $A_c^v(H)=H$, so property $(iii)$ is fulfilled.

\smallskip

Let $\f\in A_c^v(H)$. By Step 2 we know that $\ep_1\otimes e_2^*\in M_{e_2^*\circ\phi_H(0)}(H)$. Thus
$$f_2(0)=\int \f\di (\ep_1\otimes e_2^*)=f_2(1),$$
so $\f\in H$. This completes the proof of the equality, property $(iii)$ then follows by combining with Step 1.
\end{proof}

The next example illustrates that, even for spaces containing constants, vector simpliciality depends on the isometric structure of $E$ and may be affected by renorming.

\begin{example}\label{ex:renorm-const}
    There is a compact space $K$, a Banach space $E$, a function space $H\subset C(K,E)$ separating points and containing constants such that the following conditions are satisfied.
    \begin{enumerate}[$(i)$]
        \item $H$ is weakly simplicial.
        \item There is an equivalent norm on $E$ making $H$ vector simplicial.
        \item There is an equivalent norm on $E$ such that $H$ is not vector simplicial with respect to this norm.
    \end{enumerate}
\end{example}

\begin{proof}
    Let $K=\{-1,0,1\}$, $E=\ef^2$ and
    $$H=\{\f\in C(K,E)\setsep f_1(0)+f_2(0)=f_1(1)+f_2(-1)\}.$$
    Note that all norms on $E$ are equivalent, the two choices will be specified later.
    It is clear that $H$ is a linear subspace on $C(K,E)$ containing constant functions. It also separates points because the function
    $$\f(1)=(1,0),\f(-1)=(0,1), \f(0)=(1,1)$$
    belongs to $H$. Further properties will be proved in several steps.

    \smallskip

    {\tt Step 1:} $H_w=C(K)$.

    \smallskip

    Indeed, fix $f\in C(K)$. Define $\f\in C(K,E)$ by
    $$\f(1)=(f(1),0),\f(0)=(f(0),f(1)),\f(-1)=(f(-1),f(0)).$$
    Then $\f\in H$ and $e_1^*\circ \f=f$, so $f\in H_w$.

    \smallskip

    {\tt Step 2:} It follows that $\Ch_HK=K$, $H$ is weakly simplicial and $A_c^w(H)=C(K,E)$. In particular, condition $(i)$ is fulfilled.

    \smallskip

    {\tt Step 3:} Assume that $E$ is equipped with the norm $\norm{\cdot}_\infty$. Then $H$ is not vector simplicial and $A_c^v(H)=H$.

    \smallskip

    Note that $E^*$ is canonically identified with $(\ef^2,\norm{\cdot}_1)$. Let $x^*$ be the functional represented by $(1,1)$. Then $\norm{x^*}=2$ and hence also $\norm{x^*\circ\phi_H(0)}=2$. Consider the following two measures 
    $$\mu_1=\ep_0\otimes(1,1), \mu_2=\ep_1\otimes(1,0)+\ep_{-1}\otimes(0,1).$$
    Then $\norm{\mu_1}=\norm{\mu_2}=2$ and for each $\f\in H$ we have
    $$\int \f\di\mu_1=f_1(0)+f_2(0)=f_1(1)+f_2(-1)=\int\f\di\mu_2.$$
So, $\mu_1,\mu_2\in M_{x^*\circ\phi_H(0)}$. Since these measures are different and both are $H$-boundary (by Step 2), we conclude that $H$ is not vector simplicial. Further, fix $\f\in A_c^v(H)$. Then
$$f_1(0)+f_2(0)=x^*(\f(0))=\int\f\di\mu_2=f_1(1)+f_2(-1),$$
so $\f\in H$. This completes the argument.

\smallskip 

{\tt Step 4:} Assume that $E$ is equipped with the norm $\norm{\cdot}_p$ for some $p\in (1,\infty)$. 
Then $M_{x^*\circ\phi_H(t)}=\{\ep_t\circ x^*\}$ for each $t\in K$ and $x^*\in E^*$.

\smallskip

Note that $E^*$ is identified with $(\ef^2,\norm{\cdot}_q)$ where $q$ is the conjugate exponent. 
The assertion is trivial if $x^*=0$, so assume that $x^*\in E^*\setminus\{0\}$. Then $x^*$ is represented by some pair $(\alpha,\beta)\in \ef^2$. Fix $(\gamma,\delta)\in \ef^2$ with $\norm{(\gamma,\delta)}_p=1$ such that $\alpha\gamma+\beta\delta=\norm{(\alpha,\beta)}_q$. (Since $p\in (1,\infty)$, this pair is uniquely determined.)

Assume $t\in K$ and $\mu\in M_{x^*\circ\phi_H(t)}(H)$. Then $\norm{\mu}=\norm{x^*}$ and 
$$\int \f\di\mu=\alpha f_1(t)+\beta f_2(t)\mbox{ for }\f\in H.$$ We will analyze the individual cases of $t$.

Assume first $t=1$. Consider the function
$$\h(1)=(\gamma,\delta), \h(0)=(\tfrac\gamma2,\tfrac\gamma2),\h(-1)=(0,0).$$
Then $\h\in H$, $\norm{\h}=1$ (observe that $\norm{\h(0)}_p=\abs{\gamma}\cdot 2^{\frac1p-1}<1$). 
Further,
$$\norm{\mu}=\norm{x^*}=x^*(\h(1))=\int\h\di\mu=\ip{\mu(1)}{(\gamma,\delta)} + \ip{\mu(0)}{(\tfrac\gamma2,\tfrac\gamma2)}.$$
It follows that $\mu$ is carried by $\{1\}$ and so $\mu=\ep_1\otimes x^*$.

The case $t=-1$ is completely analogous, instead of $\h$ use the function
$$\g(-1)=(\gamma,\delta), \g(0)=(\tfrac\delta2,\tfrac\delta2),\g(1)=(0,0).$$

Let us look at the last case, $t=0$. Consider the function
$$\f(1)=(\tfrac12(\gamma+\delta),0), \f(0)=(\gamma,\delta),\f(-1)=(0,\tfrac12(\gamma+\delta)).$$
Then $\f\in H$, $\norm{\f}=1$ (observe that $\abs{\frac12(\gamma+\delta)}<1$).
Further,
\begin{multline*}
    \norm{\mu}=\norm{x^*}=x^*(\f(0))=\int\f\di\mu\\=\ip{\mu(0)}{(\gamma,\delta)} + \ip{\mu(1)}{(\tfrac12(\gamma+\delta),0)}+\ip{\mu(-1)}{(0,\tfrac12(\gamma+\delta))}.\end{multline*}
It follows that $\mu$ is carried by $\{0\}$ and so $\mu=\ep_0\otimes x^*$. This completes the argument.

\smallskip

{\tt Step 5:} Assume that $E$ is equipped with the norm $\norm{\cdot}_p$ for some $p\in (1,\infty)$. It follows from Step 4 that $H$ is vector simplicial and $A_c^v(H)=C(K)$. This completes the proof.
\end{proof}

We finish this section by a more detailed analysis of dependence of $A_c^v(H)$ and vector simpliciality on the choice of norm on $E$. It is contained in the following proposition.

\begin{prop}\label{P:renorm-AcvH}
    Let $K$ be a compact space, $E$ a Banach space and $H\subset C(K,E)$ a function space containing constants and separating points of $K$. Then the following assertions hold.
    \begin{enumerate}[$(a)$]
        \item If $t\in K$ and $x^*\in E^*$, then
        $$M_{x^*\circ\phi_H(t)}(H)\supset\{\sigma\otimes x^*\setsep \sigma\in M_t(x^*\circ H)\}.$$
        If $E^*$ is strictly convex, the equality holds.
        \item Let $\norm{\cdot}_1$ and $\norm{\cdot}_2$ be two equivalent norms on $E$ such that $(E,\norm{\cdot}_1)^*$ is strictly convex. Then:
        \begin{enumerate}[$(i)$]
            \item $A_c^v(H)$ considered with respect to $(E,\norm{\cdot}_2)$ is contained in $A_c^v(H)$ considered with respect to $(E,\norm{\cdot}_1)$. 
            \item If $H$ is vector simplicial as a subspace of $C(K,(E,\norm{\cdot}_2))$, then $H$ is vector simplicial as a subspace of $C(K,(E,\norm{\cdot}_1))$ 
        \end{enumerate}
        
        \item In the setting of $(b)$ assume moreover that  $(E,\norm{\cdot}_2)^*$ is also strictly convex. Then:
        \begin{enumerate}[$(i)$]
            \item Spaces $A_c^v(H)$ considered with respect to $(E,\norm{\cdot}_2)$ and with respect to $(E,\norm{\cdot}_1)$ coincide. 
            \item  $H$ is vector simplicial as a subspace of $C(K,(E,\norm{\cdot}_2))$ if and only if $H$ is vector simplicial as a subspace of $C(K,(E,\norm{\cdot}_1))$. 
        \end{enumerate}
    \end{enumerate}
\end{prop}

\begin{proof}
$(a)$: Let $x^*\in E^*$ and $\sigma\in M(K)$. Then
$$\int \f\di (\sigma\otimes x^*) =\int x^*\circ\f\di\sigma,\quad \f\in C(K,E),$$
see the computation in the proof of Lemma~\ref{L:skonstantami}$(ii)$. Now the inclusion follows easily: Assume that $\sigma\in M_t(x^*\circ H)$. Since $x^*\circ H$ contains constants, $\sigma$ is a probability measure, hence $\norm{\sigma\otimes x^*}=\norm{x^*}=\norm{x^*\circ\phi_H(t)}$ (by Lemma~\ref{L:normyevaluaci}). The above equality then implies that for each $\f\in H$ we have
$$\int \f\di  (\sigma\otimes x^*) =\int x^*\circ\f\di\sigma=x^*(\f(t)).$$
Thus $\sigma\otimes x^*\in M_{x^*\circ\phi_H(t)}(H)$. 

Next assume that $E^*$ is strictly convex. Let $\mu\in M_{x^*\circ\phi_H(t)}(H)$. Then $\norm{\mu}=\norm{x^*}$ and $x^*(\f(t))=\int\f\di\mu$ for $\f\in X$. Since $H$ contains constants, for each $x\in E$ we get
$$x^*(x)=\int x\di\mu=\mu(K)(x),$$
so $\mu(K)=x^*$. Given $A\subset K$ Borel we have
$$\begin{aligned}
 \norm{x^*}&=\norm{\mu}=\abs{\mu}(K)=\abs{\mu}(A)+\abs{\mu}(K\setminus A) \ge \norm{\mu(A)}+\norm{\mu(K\setminus A)}  \\& 
 \ge\norm{\mu(A)+\mu(K\setminus A)}=\norm{\mu(K)}=\norm{x^*},
\end{aligned}$$
hence the equalities take place. In particular, $\abs{\mu}(A)=\norm{\mu(A)}$ and
$$\norm{\mu(A)}+\norm{\mu(K\setminus A)} =\norm{\mu(A)+\mu(K\setminus A)}.$$
Since $\mu(A)+\mu(K\setminus A)=\mu(K)=x^*$, the strict convexity of $E^*$ implies that
$\mu(A)=\sigma(A)\cdot x^*$ for some $\sigma(A)\in[0,1]$. Taking into account that
$$\abs{\mu}(A)=\norm{\mu(A)}=\sigma(A)\cdot\norm{x^*},$$
we deduce that $\sigma$ is a Radon probability measure. Then $\mu=\sigma\otimes x^*$. Moreover,
for each $\f\in H$ we have (by the above equality)
$$\int x^*\circ\f\di\sigma=\int \f\di(\sigma\otimes x^*)=\int\f\di\mu=x^*(\f(t)).$$
Thus $\sigma\in M_t(x^*\circ H)$ and the argument is complete.

Assertions $(b)$ and $(c)$ follow from $(a)$ using the observation that $x^*\circ H$ (and hence $M_t(x^*\circ H)$) is not affected by the renorming of $E$.
\end{proof}

\section{Ordering of measures}\label{sec:ordering}

In this section we characterize $H$-boundary measures in $M(K,E^*)$ as maximal measures in a certain partial order. The inspiration comes from \cite{batty-vector} and we use some results from that paper.
Our basic assumption in this section is that $H\subset C(K,E)$ is a function space separating points of $K$ such that $H_w$ contains constants.

Let $S(H_w)$ denote the cone of all continuous $H_w$-convex functions on $K$ (cf. \cite[Definition 3.8]{lmns}), i.e., 
$$S(H_w)=\left\{f\in C(K,\er)\setsep f(t)\le \int f\di\sigma\mbox{ for each }t\in K\mbox{ and }\sigma\in M_t(H_w)\right\}.$$
We further define
$$\begin{aligned}
\D_H=\{f\in C(K\times B_{E^*},\er)\setsep& f(\cdot,x^*)\mbox{ is $H_w$-convex for each }x^*\in B_{E^*},\\& f(t,\cdot)\mbox{ is superlinear for each }t\in K\}.\end{aligned}$$

\begin{obs}\label{obs:compatible}
    Using the terminology of \cite{batty-vector} we have the following:
    \begin{enumerate}[$(a)$]
        \item $S(H_w)$ is a max-stable cone containing constants and separating points;
        \item $\D_H$ is a cone of continuous $S(H_w)$-subharmonic proper functions compatible with $S(H_w)$.
    \end{enumerate}
\end{obs}

\begin{proof}
    Assertion $(a)$ is obvious. The first part of assertion $(b)$ follows from the definitions, only the compatibility should be checked. So, let us check the definition from \cite[p. 544]{batty-vector}. Assume that $h_1\in S(H_w)$, $h_2\in -S(H_w)$ are such that $h_1\le h_2$ and $x\in E$.
    We need to prove that the function $g$ defined by
   \[
g(t,x^*)=\begin{cases}
            h_1(t)\cdot\Re x^*(x),&\Re x^*(x)\ge 0,\\
            h_2(t)\cdot\Re x^*(x), &\Re x^*(x)\le 0,
\end{cases} \quad
(t,x^*)\in K\times E^*
\]
 belongs to $\D_H$. 

Indeed, it is clear that $g$ is well defined and continuous. Further, fix $x^*\in E^*$. If $\Re x^*(x)\ge0$, then $g(\cdot,x^*)$ is a nonnegative multiple of $h_1$, so it is $H_w$-convex. 
If $\Re x^*(x)<0$, then $g(\cdot,x^*)$ is a negative multiple of $h_2$, so it is $H_w$-convex as well.

Finally, fix $t\in K$. We will show that the function $g(\cdot,x^*)$ is superlinear on $E^*$.
This function is clearly positively homogeneous, so only superadditivity needs to be checked.
Fix $x^*,y^*\in E^*$. If $\Re x^*(x)$ and $\Re y^*(x)$ have the same sign, clearly
$g(t,x^*+y^*)=g(t,x^*)+g(t,y^*)$. So, assume they have opposite sign, say, $\Re x^*(x)<0<\Re y^*(x)$.
We distinguish two cases.

Case 1. $\Re x^*(x)+\Re y^*(x)\ge 0$. Then
$$\begin{aligned} g(t,x^*+y^*)&=h_1(t)(\Re x^*(x)+\Re y^*(x)) \ge h_2(t)\cdot\Re x^*(x)+h_1(t)\cdot\Re y^*(x)\\&=g(t,x^*)+g(t,y^*),\end{aligned}$$
where the inequality follows from the assumptions $h_1\le h_2$ and $\Re x^*(x)<0$.

Case 2. $\Re x^*(x)+\Re y^*(x)< 0$. Then
$$\begin{aligned} g(t,x^*+y^*)&=h_2(t)(\Re x^*(x)+\Re y^*(x)) \ge h_2(t)\cdot\Re x^*(x)+h_1(t)\cdot\Re y^*(x)\\&=g(t,x^*)+g(t,y^*),\end{aligned}$$
where the inequality follows from the assumptions $h_1\le h_2$ and $\Re y^*(x)>0$.

This completes the proof.
\end{proof}

Inspired by \cite{batty-vector} and \cite{transference-studia} we define the following preorders on $M_+(K\times B_{E^*})$. If $\nu_1,\nu_2\in M_+(K\times B_{E^*})$, we define
$$\nu_1\prec_H \nu_2\equiv^{df} \int f\di\nu_1\le \int f\di\nu_2\mbox{ for }f\in \D_H$$
and
$$\nu_1\prec_{H,c} \nu_2\equiv^{df} \nu_1\prec_H \nu_2\ \&\ \norm{\nu_2}\le\norm{\nu_1}.$$ 

Further, let $W:M(K,E^*)\to M_+(K\times B_{E^*})$ be the mapping provided by \cite[Proposition 3.3]{batty-vector} (denoted by $K$ in the quoted paper) and investigated in more detail in \cite{transference-studia} (where it was denoted by $W$ and we keep the notation). Now, following \cite{batty-vector} we define a partial order on $M(K,E^*)$. If $\mu_1,\mu_2\in M(K,E^*)$, we define
$$\mu_1\prec^H\mu_2 \equiv^{df}  W\mu_1\prec_{H,c} W\mu_2.$$
Finally, by $\prec_{H_w}$ we will denote the classical Choquet order on $M_+(K)$ induced by the function space $H_w$ in the sense of \cite[Definition 3.19]{lmns}, i.e.,
$$\sigma_1\prec_{H_w}\sigma_2 \iff \forall f\in S(H_w)\colon\int f\di\sigma_1\le\int f\di\sigma_2.$$

In the following lemma we collect basic properties of the above-defined relations.

\begin{lemma}\label{L:orders-zakl}
    Let $H\subset C(K,E)$ be a function space separating points of $K$ such that $H_w$ contains constants.
    Then the following assertions are valid.
    \begin{enumerate}[$(a)$]
        \item Relations $\prec_{H}$ and $\prec_{H,c}$ are preorders (i.e., reflexive and transitive).
         \item If $\nu_1\prec_H\nu_2$, then $\int \h \di T^*\nu_1=\int \h\di T^*\nu_2$ for each $\h\in A_c^w(H)$.
        \item Relations $\prec_{H}$ and $\prec_{H,c}$ restricted to measures carried by $K\times S_{E^*}$ coincide.
        \item If $\nu_1,\nu_2$ are carried by $K\times S_{E^*}$, $\nu_1\prec_H\nu_2$ and $\norm{\nu_1}=\norm{\nu_2}$, then $\pi_1(\nu_1)\prec_{H_w}\pi_1(\nu_2)$.
        \item 
         Relations $\prec_{H}$ and $\prec_{H,c}$ restricted to measures carried by $K\times S_{E^*}$ are partial orders (i.e., additionally weakly antisymmetric).
       
        \item $\nu\prec_{H,c} WT^*\nu$ for each $\nu\in M_+(K\times B_{E^*})$.
        \item Relation $\prec^H$ is a partial order on $M(K,E^*)$.
    \end{enumerate}
\end{lemma}

\begin{proof}
    Assertion $(a)$ is obvious.

   $(b)$: Assume $\nu_1\prec_{H}\nu_2$ and $\h\in A_c^w(H)$. Then $\Re T\h\in \D_H$. Indeed, it is a continuous function. For $x^*\in B_{E^*}$ we have $\Re T\h (\cdot,x^*)=\Re (x^*\circ\h)\in A_c(H_w)$, so it is $H_w$-convex. If $t\in K$, then $\Re T\h(t,\cdot)$ is clearly real-linear, hence superlinear. So, $\int \Re T\h\di\nu_1\le \int\Re T\h\di\nu_2$. Since $-\h\in A_c^w(H)$ as well, we infer $\int \Re T\h\di\nu_1= \int\Re T\h\di\nu_2$. Finally, since $\Im T\h=-\Re T(i\h)$, the equality holds for imaginary part as well. Thus $\int T\h\di\nu_1=\int T\h\di\nu_2$, so the required equality follows.

    $(c)$: Assume that $\nu_1,\nu_2\in M_+(K\times B_{E^*})$ are carried by $K\times S_{E^*}$ and satisfy $\nu_1\prec_H\nu_2$. For $x_1,\dots,x_n\in E$ define
    $$f_{x_1,\dots,x_n}(t,x^*)=\min_{1\le j\le n}\Re x^*(x_j).$$
    Since $H_w$ contains constants, it is clear that these functions belong to $\D_H$, therefore
    $$\int f_{x_1,\dots,x_n}\di\nu_1\le \int f_{x_1,\dots,x_n}\di\nu_2.$$
    Observe that for $j=1,2$ we have
    $$ \begin{aligned}
            \inf\left\{  \int f_{x_1,\dots,x_n}\di\nu_j\setsep x_1,\dots,x_n\in B_{E} \right\}
    &=\int\inf\{   f_{x_1,\dots,x_n}\setsep x_1,\dots,x_n\in B_{E} \}\di\nu_j
    \\&=\int -\norm{x^*}\di\nu_j(x^*)=\int -1\di\nu_j=-\norm{\nu_j}.\end{aligned}$$
    Here the first equality follows from the monotone convergence for nets, the second one follows from the definition of the norm on a dual space and the third one follows from the assumption that $\nu_j$ is carried by $K\times S_{E^*}$. So, by passing to the inf we deduce that $-\norm{\nu_1}\le-\norm{\nu_2}$, hence $\nu_1\prec_{H,c}\nu_2$. This completes the argument.

    $(d)$: Assume that $\nu_1,\nu_2$ are carried by $K\times S_{E^*}$, $\nu_1\prec_H\nu_2$ and $\norm{\nu_1}=\norm{\nu_2}$. Let $g\in S(H_w)$. Fix $c\in\er$ such that $g-c\le 0$. For each choice $x_1,\dots,x_n\in B_E$ the function
    $$(t,x^*)\mapsto (c-g(t))\cdot \min\{0,f_{x_1,\dots,x_n}(x^*)\}$$
    belongs to $\D_H$ (we use the notation from the proof of $(c)$). Hence
    \begin{multline*}
    \int (c-g(t))\cdot \min\{0,f_{x_1,\dots,x_n}(x^*)\} \di\nu_1(t,x^*)\\ \le \int (c-g(t))\cdot \min\{0,f_{x_1,\dots,x_n}(x^*)\} \di\nu_2(t,x^*).\end{multline*}
    Using the monotone convergence of nets and the assumption that $\nu_1,\nu_2$ are carried by $K\times S_{E^*}$ we obtain (as in the proof of $(c)$) that
    $$\int (c-g(t))\cdot (-1) \di\nu_1(t,x^*)\le \int (c-g(t))\cdot (-1) \di\nu_2(t,x^*).$$ 
    Taking into account that $\norm{\nu_1}=\norm{\nu_2}$ we deduce
    $$\begin{aligned}
          \int g\di\pi_1(\nu_1)&= \int (g(t)-c)\di\nu_1(t,x^*) + c\norm{\nu_1}
    \\&\le \int (g(t)-c)\di\nu_2(t,x^*) + c\norm{\nu_2} =\int g\di\pi_1(\nu_1).\end{aligned}$$
    This completes the argument.

   $(e)$: The proof will be done in several steps. 

\smallskip
  
    {\tt Step 1:} Let $f\in S(H_w)$ and let $u$ be a weak$^*$-continuous sublinear functional on $E^*$. Then the function $$(t,x^*)\mapsto f(t)\cdot u(x^*)$$
belongs to $\D_H-\D_H$.

\smallskip

If, additionally, $f\le 0$ and $u\ge 0$, the function clearly belongs even to $\D_H$.
In general we may fix $c\ge 0$ such that $f-c\le0$. Then
$$f(t)\cdot u(x^*)=(f(t)-c)\cdot 2u^+(x^*) - (-c)\cdot 2u^+(x^*) - (f(t)-c)\cdot \abs{u(x^*)}
+ (-c)\cdot\abs{u(x^*)}.$$
Since both $f-c$ and $-c$ are negative functions from $S(H_w)$ and $u^+$ and $\abs{u}$ are positive  weak$^*$-continuous sublinear functionals on $E^*$, the assertion follows.

\smallskip

{\tt Step 2:} Let $f\in C(K)$ and let $u$ be a weak$^*$-continuous sublinear functional on $E^*$.
Then the function 
$$(t,x^*)\mapsto f(t)\cdot u(x^*)$$
belongs to $\overline{\D_H-\D_H}$.

\smallskip

This follows from Step 1 using the Stone-Weierstrass theorem, as $S(H_w)$ is a max-stable convex cone containing constants and separating points.

\smallskip

{\tt Step 3:} Let $f\in C(K)$ and let $u$ be a weak$^*$-continuous positively homogeneous functional on $E^*$. Then the function 
$$(t,x^*)\mapsto f(t)\cdot u(x^*)$$
belongs to $\overline{\D_H-\D_H}$. 

\smallskip

Fix $f\in C(K)$ and set
$$Y=\{u\in C(B_{E^*})\setsep (t,x^*)\mapsto f(t)u(x^*)\mbox{ belongs to }\overline{\D_H-\D_H}.\}$$
It is clear that $Y$ is a closed subspace of $C(B_{E^*})$. By Step 2 it contains all sublinear functionals. Since the span of all sublinear functionals is a vector lattice, to prove $Y$ contains all positively homogeneous functionals on $E^*$, we may use the Stone-Kakutani theorem.

To this end fix $g:B_{E^*}\to\er$ a weak$^*$-continuous positively homogeneous functional and $x^*,y^*\in B_{E^*}$. We distinguish three cases:
\begin{itemize}
    \item $x^*,y^*$ are linearly independent over $\er$. Then there is $x\in E$ with $\Re x^*(x)=g(x^*)$ and $\Re y^*(x)=g(y^*)$. Then $u(z^*)=\Re z^*(x)$ is a sublinear functional coinciding with $g$ on $\{x^*,y^*\}$.
    \item $y^*=\alpha x^*$ or $x^*=\alpha y^*$ for some $\alpha\ge 0$. Without loss of generality the first case takes place. Let $x\in E$ be such that $\Re x^*(x)=g(x^*)$. Then $\Re y^*(x)=\alpha\Re x^*(x)=\alpha g(x^*)=g(\alpha x^*)=g(y^*)$. So, we conclude as in the first case.
    \item $y^*=\alpha x^*$ for some $\alpha<0$. Fix $x,y\in E$ with $\Re x^*(x)=g(x^*)$, $\Re y^*(y)=g(y^*)$. Define $u$ by
    $$u(z^*)=\sign g(x^*)\cdot (\sign g(x^*)\cdot \Re z^*(x))^+ + \sign g(y^*)\cdot (\sign g(y^*)\cdot \Re z^*(y))^+$$ for  $z^*\in E^*$. Then $u$ is a linear combination of sublinear functionals and coincides with $g$ on $\{x^*,y^*\}$.
\end{itemize}
This verifies the assumptions of the Stone-Kakutani theorem and completes the argument.

\smallskip

{\tt Step 4:} Conclusion -- the prooof of assertion $(e)$

\smallskip
  To prove that the two relations are partial orders, it is enough to prove the weak antisymmetry. So, assume that $\nu_1,\nu_2$ are carried by $K\times S_{E^*}$ and $\nu_1\prec_{H,c}\nu_2$ and $\nu_2\prec_{H,c}\nu_2$. Then $\norm{\nu_1}=\norm{\nu_2}$ and $\int f\di\nu_1=\int f\di\nu_2$ for $f\in\D_H$. 

Fix $g\in C(K)$ and denote by
$$\A_g=\{u\in C(B_{E^*})\setsep \int g(t) u(x^*)\di\nu_1(t,x^*)=\int g(t) u(x^*)\di\nu_2(t,x^*)\}.$$
By $(d)$ we know that $\pi_1(\nu_1)=\pi_1(\nu_2)$, hence $\A_g$ contains constant functions. By Step 3 we deduce $\A_g$ contains all positively homogeneous functions. Next we use the idea from \cite[proof of Lemma 3.2]{batty-vector}: 

For $x_1,\dots,x_n\in B_E$ define
$$u^{\{x_1,\dots,x_n\}}(x^*)= \Re x^*(x_1)\cdots \Re x^*(x_n), \quad x^*\in B_{E^*}.$$
 Further, if $F\subset B_E$ is a finite set containing $x_1,\dots,x_n$, define
$$u^{\{x_1,\dots,x_n\}}_F(x^*)=\begin{cases}
    0 & \mbox{ if }\Re x^*(x)=0\mbox{ for }x\in F, \\
    \frac{u^{\{x_1,\dots,x_n\}}(x^*)}{\max\{\abs{\Re x^*(x)}\setsep x\in F\}^{n-1}} & \mbox{ otherwise}.
\end{cases}$$
Then $u^{\{x_1,\dots,x_n\}}_F$ is continuous and positively homogeneous, hence $u^{\{x_1,\dots,x_n\}}_F\in\A_g$. Moreover, the net $(u^{\{x_1,\dots,x_n\}}_F)_F$ pointwise converges in a monotone way to the function $x^*\mapsto \frac{u^{\{x_1,\dots,x_n\}}(x^*)}{\norm{x^*}}$. Therefore the monotone convergence for nets together with the assumption that $\nu_1,\nu_2$ are carried by $K\times S_{E^*}$ implies
$$\int g(t)u^{\{x_1,\dots,x_n\}}(x^*)\di\nu_1(t,x^*) = \int g(t)u^{\{x_1,\dots,x_n\}}(x^*)\di\nu_2(t,x^*),$$
hence, $u^{\{x_1,\dots,x_n\}}\in\A_g$. 

Since the linear span of constants and  functions $u^{\{x_1,\dots,x_n\}}$ is a self-adjoint algebra containing constants and separating points and it is contained in $\A_g$, the Stone-Weierstrass theorem yields that $\A_g=C(B_{E^*})$. Using the Stone-Weierstrass theorem once more we deduce that $\int f\di\nu_1=\int f\di\nu_2$ for $f\in C(K\times B_{E^*})$. Thus $\nu_1=\nu_2$.

$(f)$: It follows from \cite[Lemma 4.1]{batty-vector} that $\nu\prec_H WT^*\nu$. Moreover, using \cite[Proposition 3.3]{batty-vector} we get $\norm{WT^*\nu}=\norm{T^*\nu}\le\norm{\nu}$. This completes the argument. (Alternatively we may use \cite[Lemma 5.2]{transference-studia}.)

$(g)$: This follows from \cite[Proposition 4.7]{batty-vector} using Observation~\ref{obs:compatible}. Alternatively, we may derive it from $(e)$: Assume $\mu_1\prec^H\mu_2$ and $\mu_2\prec^H\mu_1$.
By the very definition this implies $W\mu_1\prec_{H,c}W\mu_2$ and $W\mu_2\prec_{H,c}W\mu_1$. Since $W\mu_1$ and $W\mu_2$ are carried by $K\times S_{E^*}$ (by \cite[Proposition 3.4]{batty-vector} or \cite[Proposition 3.8(ii)]{transference-studia}), assertion $(e)$ yields $W\mu_1=W\mu_2$. Hence $\mu_1=\mu_2$ (since $\mu_j=T^*W\mu_j$ by \cite[Proposition 3.4(1)]{batty-vector} or \cite[Proposition 3.8(ii)]{transference-studia}).
\end{proof}

As in the classical Choquet theory, a distinguished role is played by maximal measures. In particular, there is a close connection between $\prec^H$-maximal measures from $M(K,E^*)$, $\prec_{H,c}$-maximal measures on $K\times B_{E^*}$ and $\prec_{H_w}$-maximal measures on $K$ (as pointed out already in \cite{batty-vector}). We first note that $\prec_{H,c}$ is just a preorder, not a partial order, so a measure $\nu$ is defined to be $\prec_{H,c}$-maximal if for any $\nu^\prime$ satisfying $\nu\prec_{H,c}\nu^\prime$ we have $\nu^\prime\prec_{H,c}\nu$. However, it is easy to observe that even the equality holds. It is the content of the following lemma.

\begin{lemma}\label{L:Hc-max}  Let $H$ be a subspace of $C(K,E)$ separating points of $K$ such that $H_w$ contains constants.
Let $\nu\in M_+(K\times B_{E^*})$ be $\prec_{H,c}$-maximal. Then we have the following:
\begin{enumerate}[$(a)$]
    \item If $\nu^\prime\in M_+(K\times B_{E^*})$ is such that $\nu\prec_{H,c}\nu^\prime$, then $\nu=\nu^\prime$.
    \item $\nu=WT^*\nu$.
    \item $\nu$ is carried by $K\times S_{E^*}$.
\end{enumerate}
\end{lemma}

\begin{proof}
    Assume $\nu$ is $\prec_{H,c}$-maximal. By Lemma~\ref{L:orders-zakl}$(f)$ we deduce $WT^*\nu\prec_{H,c}\nu$. In particular,
    $$\norm{\nu}\le \norm{WT^*\nu}=\norm{T^*\nu}\le \norm{\nu},$$
    so the equalities hold. In particular, $\norm{T^*\nu}=\norm{\nu}$ and hence $\nu$ is carried by $K\times S_{E^*}$ (by \cite[Lemma 3.1]{batty-vector} or \cite[Lemma 3.1]{transference-studia}).
    This proves assertion $(c)$.

    Now assume that $\nu^\prime\in M_+(K\times B_{E^*})$ is such that $\nu\prec_{H,c}\nu^\prime$. Then $\nu^\prime$ is clearly also $\prec_{H,c}$-maximal. So, assertion $(c)$ implies that both $\nu$ and $\nu^\prime$ are carried by $K\times S_{E^*}$. Since $\nu^\prime\prec_{H,c}\nu$, we deduce $\nu=\nu^\prime$ by Lemma~\ref{L:orders-zakl}$(e)$. This proves $(a)$.

    Finally, assertion $(b)$ follows from $(a)$ using Lemma~\ref{L:orders-zakl}$(f)$.
\end{proof}

We continue by a proposition characterizing $H$-boundary measures using various forms of maximality. A substantial part of this result follows from \cite{batty-vector}.

\begin{prop}\label{P:maximalitamu}  Let $H$ be a subspace of $C(K,E)$ separating points of $K$ such that $H_w$ contains constants.
    Let $\mu\in M(K,E^*)$. Then the following assertions are equivalent.
    \begin{enumerate}[$(1)$]
        \item $\mu$ is $H$-boundary;
        \item $\mu$ is $\prec^H$-maximal;
        \item $W\mu$ is $\prec_{H,c}$-maximal;
        \item $\abs{\mu}$ is $\prec_{H_w}$-maximal. 
    \end{enumerate}
\end{prop}

\begin{proof}
    Equivalences $(2)\iff(3)\iff(4)$ follow from Observation~\ref{obs:compatible} and \cite[Theorem 4.9]{batty-vector}.

    $(1)\iff(4)$: By Theorem~\ref{T:H-maximal} and the following definition $\mu$ is $H$-boundary if and only if $\phi_{H_w}(\abs{\mu})$ is a maximal measure on $B_{H_w^*}$. It follows from \cite[Proposition 4.28(d)]{lmns} that this is equivalent to $\abs{\mu}$ being $H_w$-maximal.
\end{proof}

We continue by a characterization of $\prec_{H,c}$-maximal measures which may be viewed as a strengthening of Lemma~\ref{L:Hc-max}.

\begin{prop}
    Let $\nu\in M_+(K\times B_{E^*})$. Then the following assertions are equivalent.
    \begin{enumerate}[$(1)$]
        \item $\nu$ is $\prec_{H,c}$-maximal.
        \item $\nu=WT^*\nu$ and $\pi_1(\nu)$ is $H_w$-maximal.
    \end{enumerate}
\end{prop}

\begin{proof}
    $(2)\implies (1)$: If $\nu=WT^*\nu$, then $\pi_1(\nu)=\abs{T^*\nu}$ (by \cite[Propositions 3.3 and 3.4]{batty-vector} or \cite[Proposition 3.8(iii)]{transference-studia}), so we may use implication $(4)\implies(3)$ from Proposition~\ref{P:maximalitamu}.

    $(1)\implies (2)$: Assume $\nu$ is $\prec_{H,c}$-maximal. By Lemma~\ref{L:Hc-max} we know that
    $\nu=WT^*\nu$. Hence $\pi_1(\nu)=\abs{T^*\nu}$ (as in the proof of the converse implication), so we conclude by implication $(3)\implies(4)$ from  Proposition~\ref{P:maximalitamu}.
\end{proof}

\begin{remarks}
    (1) The orders addressed in this section provide an alternative proof of the representation theorem, assuming that $H_w$ contains constants: Given $\varphi\in H^*$, find $\widetilde{\varphi}\in C(K,E)^*$ extending $\varphi$ with the preservation of the norm. Let $\mu_0\in M(K,E^*)$ represent $\widetilde{\varphi}$. Find $\mu\in M(K,E^*)$ $\prec^H$-maximal with $\mu_0\prec^H\mu$ (it exists due to \cite[Theorem 4.3]{batty-vector}). Then $\mu\in M_\varphi(H)$ is $H$-boundary.

    (2) Proposition~\ref{P:maximalitamu} also permits (assuming that $H_w$ contains constants)  to express vector simpliciality and functional vector simpliciality using uniqueness of representing $\prec^H$-maximal measures, similarly as in the classical setting.
\end{remarks}

In the following theorem we provide a characterization of weak simpliciality via uniqueness of $\prec^H$-maximal measures. 

\begin{thm}
    Let $H\subset C(K,E)$ be a linear subspace separating points such that $H_w$ contains constants. Then the following assertions are equivalent.
    \begin{enumerate}[$(1)$]
        \item $H$ is weakly simplicial.
        \item For any $\sigma\in M_+(K)$ there is a unique $H_w$-maximal $\sigma^\prime\in M_+(K)$ with $\sigma\prec_{H_w}\sigma^\prime$.
        \item  For any $\mu\in M(K,E^*)$ there is a unique $\prec^H$-maximal $\mu^\prime\in M(K,E^*)$ with $\mu\prec^{H}\mu^\prime$.
    \end{enumerate}
\end{thm}

\begin{proof}
    $(3)\implies(2)$: Let $\sigma\in M_+(K)$ and let $\sigma_1,\sigma_2\in M_+(K)$ be $H_w$-maximal such that $\sigma\prec_{H_w}\sigma_1$ and $\sigma\prec_{H_w}\sigma_2$. Fix $x^*\in S_{E^*}$. Then $\sigma_1\otimes x^*$ and $\sigma_2\otimes x^*$ are $\prec^H$-maximal by Proposition~\ref{P:maximalitamu}. Moreover, $\sigma\otimes x^*\prec^H\sigma_1\otimes x^*$ and $\sigma\otimes x^*\prec^H\sigma_2\otimes x^*$.

    Indeed, $W(\sigma\otimes x^*)=\sigma\times \ep_{x^*}$, $W(\sigma_1\otimes x^*)=\sigma_1\times \ep_{x^*}$ and clearly 
    $\sigma\times \ep_{x^*}\prec_{H,c}\sigma_1\times \ep_{x^*}$. Similarly for $\sigma_2$. 

    Hence, assuming $(3)$, we deduce $\sigma_1=\sigma_2$, which completes the argument.

    $(2)\implies(1)$: Let $t\in K$. Then $M_t(H_w)=\{\sigma\in M_+(K)\setsep \ep_t\prec_{H_w}\sigma\}$. Thus, assuming $(2)$, we get that $H_w$ is simplicial.

    $(1)\implies (3)$ Assume that $H$ is weakly simplicial. Let $\mu,\mu_1,\mu_2\in M(K,E^*)$ be such that $\mu\prec^H\mu_1$, $\mu\prec^H\mu_2$ and $\mu_1,\mu_2$ are $\prec^H$-maximal. By definition of $\prec^H$ and Lemma~\ref{L:orders-zakl}$(b)$ we get that $\mu_1|_{A_c^w(H)}=\mu_2|_{A_c^w(H)}=\mu|_{A_c^w(H)}$, hence $\mu_1-\mu_2\in A_c^w(H)^\perp$. Since $\mu_1-\mu_2$ is $H$-boundary (by Proposition~\ref{P:maximalitamu}), Corollary~\ref{cor:weaKsimpl} (implication $(1)\implies(6)$) yields $\mu_1=\mu_2$. This completes the proof.
\end{proof}

\begin{remark}
    The previous theorem is inspired by  \cite[Theorem 5.1 and Example 5.2]{batty-vector}. The quoted result says (in particular) that, assuming moreover that $H$ contains constants and $E^*$ is uniformly convex, condition $(2)$ is equivalent to
\begin{enumerate}[$(3^\prime)$]
    \item $\forall \mu\in M(K,E^*),\norm{\mu}=\norm{\mu|_H}\,\exists! \mu^\prime\in M(K,E^*) \prec^H\mbox{-maximal}, \mu\prec^H\mu^\prime$.
\end{enumerate}
Note that clearly $(3)\implies(3^\prime)$ and, if $H$ contains constants, the above proof of $(3)\implies (2)$ gives even $(3^\prime)\implies(2)$. So, the above theorem is a substantial strengthening of \cite[Example 5.2]{batty-vector}, as we may completely drop the assumption that $E^*$ is uniformly convex.

 We note that condition $(3^\prime)$ is strictly weaker than functional vector simpliciality (of $H$), even in case $A_c^v(H)=H$. This is witnessed by Example~\ref{ex:nezachovani}. The reason for this behavior is the observation that, given $t\in K$, $x^*\in E^*$ and  $\mu\in M(K,E^*)$, then
 $$\ep_t\otimes x^*\prec^H\mu \begin{smallmatrix}{}
     \mkern15mu\not{\ }\mkern-25mu{}\impliedby \\ \implies  \end{smallmatrix}\mu\in M_{x^*\circ\phi_H(t)}(H).
$$
 
\end{remark}

\section{More on the orderings on $M_+(K\times B_{E^*})$}\label{sec:nasoucinu}

In this section we provide a more detailed analysis of the relation $\prec_{H,c}$. We will apply and generalize results of \cite[Sections 5.3--5.5]{transference-studia}. First we recall some notation from \cite{transference-studia}. Given $\mu\in M(K,E^*)$, set
$$N(\mu)=\{\nu\in M_+(K\times B_{E^*})\setsep T^*\nu=\mu\ \&\ \norm{\nu}=\norm{\mu}\}.$$
All measures in $N(\mu)$ are carried by $K\times S_{E^*}$ (by \cite[Lemma 3.1]{transference-studia}), hence relations $\prec_{H,c}$ and $\prec_H$ coincide and are partial orders on $N(\mu)$ (by Lemma~\ref{L:orders-zakl}$(e)$). Always $W\mu\in N(\mu)$ and $W\mu$ is the $\prec_{H,c}$-largest element of $N(\mu)$ (by Lemma~\ref{L:orders-zakl}$(f)$). We start by a variant of \cite[Theorem 4.4]{transference-studia}.

\begin{prop}
    Let $H\subset C(K,E)$ be a function space separating points such that $H_w$ contains constant functions. Then the following assertions are equivalent.
    \begin{enumerate}[$(1)$]
        \item $E^*$ is strictly convex.
        \item $N(\mu)=\{W\mu\}$ for each $\mu\in M(K,E^*)$.
        \item $N(\mu)=\{W\mu\}$ for each $\prec^H$-maximal measure $\mu\in M(K,E^*)$.
    \end{enumerate}
  If, moreover, $H$ contains constants, the above assertions are also equivalent to the following one.
    \begin{enumerate}[$(4)$]
        \item $N(\mu)=\{W\mu\}$ for each $\prec^H$-maximal measure $\mu\in M(K,E^*)$ satisfying $\norm{\mu}=\norm{\mu|_H}$.
    \end{enumerate}
\end{prop}

\begin{proof}
    Implication $(1)\implies (2)$ follows from \cite[Theorem 4.4] {transference-studia} (implication $(1)\implies (3)$). Implication $(2)\implies (3)$ is trivial. 

    $(3)\implies(1)$: Assume $E^*$ is not strictly convex. Fix $t\in\Ch_HK$ and $x^*\in S_{E^*}\setminus \ext B_{E^*}$. Set $\mu=\ep_t\otimes x^*$. Then $\mu$ is $\prec^H$-maximal and it is easy to see that $N(\mu)$ contains more measures (see the proof of \cite[Theorem 4.4]{transference-studia}, implication $(2)\implies(1)$).

    Implication $(3)\implies(4)$ is trivial. If $H$ contains constants, then the above argument proving $(3)\implies (1)$ in fact yields $(4)\implies (1)$. Indeed, we have
    $$\norm{\ep_t\otimes x^*|_H}=\norm{x^*\circ\phi_H(t)}=\norm{x^*}=\norm{\ep_t\otimes x^*},$$
    where the second equality follows from Lemma~\ref{L:normyevaluaci}$(a)$.
\end{proof}

If $E^*$ is not strictly convex, $N(\mu)$ may have a richer structure. Using Zorn's lemma it is easy to see that there are $\prec_{H,c}$-minimal elements of $N(\mu)$. We will analyze this structure in more detail in case $\mu$ is $\prec^H$-maximal. To this end we need some tools which will be contained in the following two subsections.

\subsection{On the Choquet ordering on $B_{E^*}$.}

In this auxiliary subsection we prove a strengthening of \cite[Theorem 5.4]{transference-studia} which we find to be of an independent interest. It is contained in the following lemma.

\begin{lemma}\label{L:uspornakouli}
    Let $E$ be a real Banach space and $\sigma_1,\sigma_2\in M_1(B_{E^*})$ satisfy $r(\sigma_1)=r(\sigma_2)\in S_{E^*}$. Assume that
    $$\int p\di\sigma_1\le \int p\di\sigma_2\mbox{ whenever $p:E^*\to[0,\infty)$ is weak$^*$-continuous sublinear}.$$
    Then $\sigma_1\prec\sigma_2$ in the Choquet ordering.
\end{lemma}

\begin{proof} The proof will be done in several steps. 

\smallskip

{\tt Step 1:} Let $f:B_{E^*}\to [0,\infty)$ be a convex, weak$^*$-continuous and norm-Lipschitz function with $f(0)=0$. Then $\int f\di\sigma_1\le\int f\sigma_2$.

\smallskip

Denote by $L$ the Lipschitz constant of $f$. Let $\ep>0$ be arbitrary. By \cite[Lemma 5.5]{transference-studia} applied to $\frac12(\sigma_1+\sigma_2)$ we obtain a weak$^*$-compact convex set $L\subset S_{E^*}$ such that $\sigma_1(L)>1-\ep$ and $\sigma_2(L)>1-\ep$ (as in the proof of \cite[Theorem 5.4]{transference-studia}). By \cite[Lemma 5.6]{transference-studia} and the monotone convergence theorem for nets there is a $6L$-Lipschitz weak$^*$-continuous sublinear functional $p$ on $E$ satisfying $p\le f$ on $K$ and $\int_K p\di\sigma_1>\int_K f\di\sigma_1-\ep$. Then $p^+$ has the same properties and, moreover, it is non-negative. Hence
$$\begin{aligned}
  \int_{B_{E^*}} f\di\sigma_1&\le \int_K f\di\sigma_1+L\ep\le \int_K p^+ \di\sigma_1+ (L+1)\ep
  \le \int_{B_{E^*}}p^+ \di\sigma_1+ (7L+1)\ep
  \\& \le \int_{B_{E^*}}p^+ \di\sigma_2+ (7L+1)\ep 
  \le \int_{K}p^+ \di\sigma_2+ (13L+1)\ep \\&
  \le \int_{K}f \di\sigma_2+ (13L+1)\ep 
  \le \int_{B_{E^*}} f\di\sigma_2 + (14L+1)\ep. 
  \end{aligned}$$
Since $\ep>0$ is arbitrary, the argument is complete.

\smallskip

{\tt Step 2:} Let $f:B_{E^*}\to [0,\infty)$ be a bounded convex, weak$^*$-lower semicontinuous function with $f(0)=0$. Then $\int f\di\sigma_1\le\int f\sigma_2$.

\smallskip

It is a consequence of the Hahn-Banach theorem that
$$\begin{aligned}
   f&=\sup\{u\le f\setsep u\mbox{ affine continuous}\} 
\\& = \sup\{ (u_1\vee\dots\vee u_n)^+\setsep u_1,\dots, u_n \mbox{ affine continuous}, u_j\le f\mbox{ for }j=1,\dots,n\}.\end{aligned}$$
Each of the functions $(u_1\vee\dots\vee u_n)^+$ satisfies the assumptions of Step 1, therefore we may conclude by the monotone convergence theorem for nets.

\smallskip

{\tt Step 3:} Let $f:B_{E^*}\to [0,\infty)$ be a convex weak$^*$-continuous function. Then $\int f\di\sigma_1\le\int f\sigma_2$.

\smallskip

Define $g(0)=0$ and $g=f$ on $B_{E^*}\setminus \{0\}$. Then $g$ is a lower semicontinuous function.
Then $g_*$, the lower envelope of $g$, satisfies the assumptions of Step 2, hence $\int g_*\di\sigma_1\le\int g_*\di\sigma_2$. To complete the argument it is enough to show that $g_*=f$ on $S_{E^*}$ (note that $\sigma_1$ and $\sigma_2$, having the barycenter on the sphere, are carried by the sphere). So, fix $x^*\in S_{E^*}$. By \cite[Lemma 3.21]{lmns} (applied to $-g$) there is some $\mu\in M_1(B_{E^*})$ with $r(\mu)=x^*$ such that $g_*(x^*)=\int g\di\mu$.
Since $x^*\in S_{E^*}$, $\mu$ is carried by $S_{E^*}$.
Thus
$$g_*(x^*)=\int g\di\mu=\int f\di\mu\ge f(x^*)=g(x^*)\ge g_*(x^*),$$
where we used the convexity of $f$. Hence equalities hold, in particular $g_*(x^*)=f(x^*)$
and the argument is complete.

\smallskip

{\tt Step 4:} $\int f\di\sigma_1\le\int f\di\sigma_2$ for each $f:B_{E^*}\to\er$ convex weak$^*$-continuous, i.e., $\sigma_1\prec\sigma_2$.

\smallskip

Let $f$ be such a function. Then there is $c\in\er$ such that $f+c\ge0$. Since $\sigma_1,\sigma_2$ are probabilities, we conclude using Step 3.  
\end{proof}

\subsection{Disintegration of measures from $N(\mu)$}

Disintegration of measures was a key tool in \cite{transference-studia}. The basic result on disintegration was recalled in Lemma~\ref{L:dezintegrace-obec}. In this section we apply it for $B=B_{E^*}$. The following lemma collects some special properties for measures from $N(\mu)$.

\begin{lemma}\label{L:dezintegrace-Nmu}  Let $H\subset C(K,E)$ be a function space separating points such that $H_w$ contains constant functions. Let $\mu\in M(K,E^*)\setminus\{0\}$ and $\nu\in N(\mu)$. 
Then $\pi_1(\nu)=\abs{\mu}$.
Moreover, if $(\nu_t)_{t\in K}$ is a disintegration kernel of $\nu$, then $r(\nu_t)\in S_{E^*}$ for $\abs{\mu}$-almost all $t\in K$.
\end{lemma}

\begin{proof}
    The first statement follows from \cite[Proposition 3.5(a)]{transference-studia}). The `moreover part' follows from \cite[Proposition 3.5(b)]{transference-studia}.
\end{proof}

We continue by a lemma on fine properties of disintegrations of measures from $N(\mu)$.

\begin{lemma}\label{L:lifting}  Let $H\subset C(K,E)$ be a linear subspace separating points such that $H_w$ contains constant functions.
Let $\mu\in M(K,E^*)\setminus\{0\}$. Then there is an assignment of disintegration kernels $N(\mu)\ni\nu\mapsto(\nu_t)_{t\in K}$ such that for  any two measures $\nu_1,\nu_2\in N(\mu)$ the following conditions are fulfilled.
\begin{enumerate}[$(i)$]
    \item If  $g_1,g_2:B_{E^*}\to\er$  are bounded Borel functions, then 
   \begin{multline*}       
   \int g_1\di\nu_{1,t}\le \int g_2\di\nu_{2,t} \quad \abs{\mu}\mbox{-almost everywhere} \\
   \implies \int g_1\di\nu_{1,t}\le \int g_2\di\nu_{2,t} \quad \mbox{for each }t\in K.\end{multline*}    
   \item $r(\nu_{1,t})=r(\nu_{2,t})$ for each $t\in K$.
\end{enumerate}
\end{lemma}

\begin{proof}
    Since $\pi_1(\nu)=\abs{\mu}$ for each $\nu\in N(\mu)$ (by Lemma~\ref{L:dezintegrace-Nmu}), the existence of an assignment satisfying $(i)$ follows from \cite[Proposition 2.10]{transference-studia} applied to $\abs{\mu}$ in place of $\sigma$.

    Next we will show that $(ii)$ is a consequence of $(i)$. Since $W\mu\in N(\mu)$, without loss of generality we may assume that $\nu_2=W\mu$, hence $\nu_1\prec_{\D}\nu_2$, where $\prec_{\D}$ is the preorder introduced and investigated in \cite{transference-studia} (cf. \cite[Lemma 5.2(c)]{transference-studia}). By \cite[Theorem 5.11]{transference-studia} we deduce that, given  a weak$^*$-continuous sublinear functional $p$ on $E^*$, we have  $\int p\di\nu_{2,t}\le\int p\di\nu_{1,t}$ $\abs{\mu}$-almost everywhere.  By $(i)$ we infer that the inequality holds everywhere. It follows that for $t\in K$ and any weak$^*$-continuous linear functional $p$ on $E^*$ we have $\int p\di\nu_{2,t}=\int p\di\nu_{1,t}$. Now it easily follows that $r(\nu_{1,t})=r(\nu_{2,t})$ for each $t\in K$.
\end{proof}

\subsection{Finer structure of $N(\mu)$ for $\mu$ $\prec^H$-maximal}

Now we are ready to formulate and prove the promised results on the structure of order $\prec_{H,c}$ on certain $N(\mu)$. As in \cite{transference-studia} we set 
$$\K=\{f\in C(K\times B_{E^*})\setsep f(t,\cdot)\mbox{ is concave for each }t\in K\}$$
and by $\prec_{\K}$ we denote the partial order on $M_+(K\times B_{E^*})$ induced by $\K$. 

\begin{thm}\label{T:Nmu-uspor}  Let $H\subset C(K,E)$ be a function space separating points such that $H_w$ contains constant functions.
    Let $\mu\in M(K,E^*)\setminus\{0\}$ be $\prec^H$-maximal.  
    Then there is an assignment of disintegration kernels $\nu\mapsto (\nu_t)_{t\in K}$ for $\nu\in N(\mu)$ such that, given $\nu_1,\nu_2\in N(\mu)$, the following assertions are equivalent.
     \begin{enumerate}[$(1)$]
        \item $\nu_1\prec_{H,c}\nu_2$.
        \item $\nu_{2,t}\prec\nu_{1,t}$  for each $t\in K$.
        \item $\nu_1\prec_{\K}\nu_2$.
    \end{enumerate}
\end{thm}

To prove the theorem we will use the following lemma.

\begin{lemma}\label{L:Nmu-uspor}  Let $H\subset C(K,E)$ be a function space separating points such that $H_w$ contains constant functions.
    Assume that $\mu\in M(K,E^*)\setminus\{0\}$ is $\prec^H$-maximal. Let $\nu_1,\nu_2\in N(\mu)$ be such that $\nu_1\prec_{H,c}\nu_2$. Let $(\nu_{1,t})_{t\in K}$ and $(\nu_{2,t})_{t\in K}$ be disintegration kernels of $\nu_1$ and $\nu_2$. If $p:B_{E^*}\to[0,\infty)$ is weak$^*$-continuous and sublinear, then
    $$\int p\di\nu_{2,t}\le\int p\di\nu_{1,t}\mbox{ for $\abs{\mu}$-almost all }t\in K.$$
\end{lemma}

\begin{proof}
Let $p$ be as above. The proof will be done in two steps.

\smallskip

{\tt Step 1:} If $f\in C(K)$ is such that $f\le 0$, then 

$$\int f(t)p(x^*)\di\nu_1(t,x^*)\le \int f(t)p(x^*)\di\nu_2(t,x^*).$$

\smallskip

Let $f\in C(K)$ be such that $f\le0$. Then 
    $$\begin{aligned}
         \int_{K\times B_{E^*}} f(t)p(x^*)&\di\nu_1(t,x^*)=\int_K\left( f(t) \int_{B_E^*} p\di\nu_{1,t}\right)\di\abs{\mu}(t) \\& =  \int_K\left( \widecheck{f}(t) \int_{B_E^*} p\di\nu_{1,t}\right)\di\abs{\mu}(t)
         \\&= \sup\left\{ \int_K\left( g(t) \int_{B_E^*} p\di\nu_{1,t}\right)\di\abs{\mu}(t)\setsep g\le f\ H_w\mbox{-convex}\right\}
         \\&=\sup\left\{ \int_{K\times B_{E^*}} g(t) p(x^*)\di\nu_1(t,x^*)\setsep g\le f\ H_w\mbox{-convex}\right\}
         \\&\le \sup\left\{ \int_{K\times B_{E^*}} g(t) p(x^*)\di\nu_2(t,x^*)\setsep g\le f\ H_w\mbox{-convex}\right\}
         \\&=\int_{K\times B_{E^*}} f(t)p(x^*)\di\nu_2(t,x^*).
    \end{aligned}$$
The first equality follows from the disintegration formula (Lemma~\ref{L:dezintegrace-obec}$(i)$). To explain the second one recall that $\widecheck{f}$ is the lower envelope of $f$ and $f=\widecheck{f}$ $\abs{\mu}$-almost everywhere by the Mokobodzki criterion (recall that $\abs{\mu}$ is $H_w$-maximal by Proposition~\ref{P:maximalitamu}). The third equality follows from the definition of the lower envelope and monotone convergence of nets. The fourth equality is again an application of the disintegration formula. The inequality follows from the assumption $\nu_1\prec_{H,c}\nu_2$ as the function $(t,x^*)\mapsto g(t)p(x^*)$ belongs to $\D_H$. The last equality follows from the first four equalities applied to $\nu_2$.

\smallskip

{\tt Step 2:} If $\int p\di\nu_{2,t}>\int p\di\nu_{1,t}$ for $t$ from a set of strictly positive measure, then there is $f\in C(K)$ with $f\le 0$ such that $\int f(t)p(x^*)\di\nu_1(t,x^*)> \int f(t)p(x^*)\di\nu_2(t,x^*)$.

\smallskip

This follows from the proof of implication $(1)\implies(2)$ of \cite[Theorem 5.11]{transference-studia} (the computation there may be directly applied to our setting).

\smallskip

We now conclude by combining Steps 1 and 2.
\end{proof}

\begin{proof}[Proof of Theorem~\ref{T:Nmu-uspor}]
We choose an assignment of disintegration kernels provided by Lemma~\ref{L:lifting}. Let us prove the individual implications.

$(1)\implies(2)$: Assume $\nu_1\prec_{H,c}\nu_2$. Let $p:B_{E^*}\to[0,\infty)$ be a weak$^*$-continuous sublinear function. By Lemma~\ref{L:Nmu-uspor} we know that $\int p\di\nu_{2,t}\le\int p\di\nu_{1,t}$ for $\abs{\mu}$-almost all $t\in K$. By property $(i)$ from Lemma~\ref{L:lifting} the inequality holds for each $t\in K$.
Further, property $(ii)$ of the same lemma implies that $r(\nu_{1,t})=r(\nu_{2,t})$ for each $t\in K$. By Lemma~\ref{L:dezintegrace-Nmu} almost all the barycenters are on the sphere, hence Lemma~\ref{L:uspornakouli} implies that $\nu_{2,t}\prec\nu_{1,t}$ for $\abs{\mu}$-almost all $t\in K$.

It implies that for any convex continuous function $g:B_{E^*}\to \er$ we have $\int g\di\nu_{2,t}\le\int g\di\nu_{1,t}$ for $\abs{\mu}$-almost all $t\in K$. By property $(i)$ from Lemma~\ref{L:lifting} the inequality holds for each $t\in K$. Now it follows that $\nu_{2,t}\prec\nu_{1,t}$ for all $t\in K$, hence the argument is complete.

$(2)\implies (3)$: This follows from the disintegration formula in Lemma~\ref{L:dezintegrace-obec}$(i)$.

$(3)\implies(1)$: Assume $\nu_1\prec_{\K}\nu_2$. Since $\D_H\subset\K$, we deduce $\nu_1\prec_H\nu_2$. Since both measures have the same norm, we conclude $\nu_1\prec_{H,c}\nu_2$.
\end{proof}

Theorem~\ref{T:Nmu-uspor} yields that, given a $\prec^H$-maximal $\mu\in M(K,E^*)$, the order $\prec_{H,c}$ on $N(\mu)$ coincides with the order $\prec_{\D}$ from \cite{transference-studia}. Therefore several results of \cite{transference-studia} may be directly applied to our setting. 
In particular, the following proposition is an immediate consequence of \cite[Theorem 5.20]{transference-studia}.

\begin{prop}\label{P:minimalchar}
     Let $H\subset C(K,E)$ be a linear subspace separating points such that $H_w$ contains constant functions. Let $\mu\in M(K,E^*)\setminus\{0\}$ be $\prec^H$-maximal. Then a measure $\nu\in N(\mu)$ is $\prec_{H,c}$-minimal in $N(\mu)$ if and only if it admits a disintegration kernel consisting of maximal measures.
\end{prop}

We continue by a theorem characterizing uniqueness of $\prec_{H,c}$-minimal measures. We recall that a convex set is a \emph{simplexoid} if each of its proper faces is a simplex. This notion was introduced in \cite{phelps-complex} and used in \cite{bezkonstant,transference-studia}. We also note that $B_{E^*}$ is a simplexoid if and only if any $x^*\in S_{E^*}$ admits a unique maximal representing probability measure (cf. \cite[Fact 2.4]{bezkonstant}).

\begin{thm}\label{T:simplexoid}
     Let $H\subset C(K,E)$ be a linear subspace separating points such that $H_w$ contains constant functions. The following assertions are equivalent.
     \begin{enumerate}[$(1)$]
         \item $B_{E^*}$ is a simplexoid.
         \item For each $\mu\in M(K,E^*)$ $\prec^H$-maximal there is a unique $\prec_{H,c}$-minimal measure $\nu\in N(\mu)$.
     \end{enumerate}
     If, moreover, $H$ contains constants, then these assertions are also equivalent to the following one.
      \begin{enumerate}[$(3)$]
          \item For each $\mu\in M(K,E^*)$ $\prec^H$-maximal satisfying $\norm{\mu|_H}=\norm{\mu}$ there is a unique $\prec_{H,c}$-minimal measure $\nu\in N(\mu)$.
     \end{enumerate}
\end{thm}

\begin{proof}
    Implication $(1)\implies(2)$ follows directly from the corresponding implication of \cite[Theorem 5.23]{transference-studia}.

    Implication $(2)\implies(1)$ follows from the proof of the corresponding implication of \cite[Theorem 5.23]{transference-studia}. The only modification is to choose $t\in \Ch_HK$.

    Implication $(2)\implies (3)$ is trivial. If $H$ contains constants, the proof of $(2)\implies (1)$ yields $(3)\implies (1)$ due to Lemma~\ref{L:normyevaluaci}.
\end{proof}

We continue by the following proposition on the relationship of $\prec_{H,c}$-minimal measures to Choquet boundaries.

\begin{prop}\label{P:hranice-max}
    Let $H\subset C(K,E)$ be a linear subspace separating points such that $H_w$ contains constant functions. Let $\nu\in M_+(K\times B_{E^*})$.
    \begin{enumerate}[$(a)$]
        \item Assume that $T^*\nu$ is $\prec^H$-maximal and  $\nu$ is a $\prec_{H,c}$-minimal element of $N(T^*\nu)$. Then $\nu$ is carried by any set
        of the form $A\times B$, where $A\subset K$ is a Baire set containing $\Ch_HK$ and $B\subset B_{E^*}$ is a Baire set containing $\ext B_{E^*}$.

        \item  Assume that, moreover, $K$ is metrizable, $E$ is separable and $\norm{T^*\nu}=\norm{\nu}$. Then the following two assertions are equivalent.
        \begin{enumerate}[$(i)$]
            \item $T^*\nu$ is $\prec^H$-maximal and $\nu$ is $\prec_{H,c}$-minimal in $N(T^*\nu)$.
            \item $\nu$ is carried by $\Ch_HK\times\ext B_{E^*}$.
        \end{enumerate}
    \end{enumerate}
\end{prop}

\begin{proof}
    $(a)$: Let $A$ and $B$ be as in the statement.    
    If $T^*\nu$ is $\prec^H$-maximal, then $\abs{T^*\nu}$ is 
    $\prec_{H_w}$-maximal by Proposition~\ref{P:maximalitamu}, hence $\abs{T^*\nu}$ is carried by $A$ (due to \cite[Theorem 3.79(a)]{lmns}). Assume that, additionally, $\nu$ is $\prec_{H,c}$-minimal element of $N(T^*\nu)$. It follows, in particular, that $\nu\in N(T^*\nu)$ and hence $\abs{T^*\nu}=\pi_1(\nu)$ (by Lemma~\ref{L:dezintegrace-Nmu}), so $\nu$ is carried by $A\times B_{E^*}$. Further, by combining Theorem~\ref{T:Nmu-uspor} with \cite[Corollary 5.19]{bezkonstant} 
    we get that $\nu$ is carried by $K\times B$. Now it easily follows that $\nu$ is carried by $A\times B$.

    $(b)$: Implication $(i)\implies (ii)$ follows from $(a)$, as in this case $\Ch_HK$ and $\ext B_{E^*}$ are Baire sets.  To prove $(ii)\implies (i)$ assume that $\nu$ is carried by $\Ch_HK\times B_{E^*}$. Then $\pi_1(\nu)$ is carried by $\Ch_HK$ and so it is $\prec_{H_w}$-maximal (by \cite[Corollary 3.62]{lmns}). Since $\pi_1(\nu)=\abs{T^*\nu}$ (by \cite[Proposition 3.5(a)]{transference-studia}), Proposition~\ref{P:maximalitamu} implies that $T^*\nu$ is $\prec^H$-maximal. Further, let $(\nu_t)_{t\in K}$ be a disintegration kernel of $\nu$ provided by Lemma~\ref{L:lifting}. Then
    $$0=\nu(K\times (B_{E^*}\setminus \ext B_{E^*}))=\int_K \nu_t(B_{E^*}\setminus \ext B_{E^*})\di\pi_1(\nu)(t),$$
    so $\nu_t(B_{E^*}\setminus \ext B_{E^*})=0$ for $\pi_1(\nu)$-almost all $t\in K$. By the properties of the kernel this equality holds for each $t\in K$. Thus each $\nu_t$ is a maximal measure on $B_{E^*}$ (by \cite[Corollary 3.62]{lmns}). Hence $\nu$ is $\prec_{H,c}$-minimal by Proposition~\ref{P:minimalchar}.
\end{proof}

\subsection{Another form of representing measures and their uniqueness}\label{ss:jinarep}

We now present a new type of a boundary representation theorem, with the use of positive measures on $K\times B_{E^*}$.

\begin{thm}\label{T:reprez-soucin}
    Let $H\subset C(K,E)$ be a function space separating points such that $H_w$ contains constants. Let $\varphi\in H^*$ be given. Then there is a measure $\nu\in M_+(K\times B_{E^*})$ satisfying the following conditions.
    \begin{enumerate}[$(i)$]
        \item $\norm{\nu}=\norm{\varphi}$;
        \item $\varphi(\f)=\int x^*(\f(t))\di\nu(t,x^*)$ for each $\f\in H$;
        \item $\pi_1(\nu)$ is $\prec_{H_w}$-maximal;
        \item $\nu$ is $\prec_{H,c}$-minimal within the set
    $$N(T^*\nu)=\{\nu^\prime\in M_+(K\times B_{E^*})\setsep \norm{\nu^\prime}=\norm{\nu}\ \&\ T^*\nu^\prime=T^*\nu\}.$$
    \end{enumerate}
\end{thm}

\begin{proof}
  By Theorem~\ref{T:reprez-funct} there is an $H$-boundary measure $\mu\in M(K,E^*)$ with $\norm{\mu}=\norm{\varphi}$ such that $\varphi(\f)=\int\f\di\mu$ for $\f\in H$. By Proposition~\ref{P:maximalitamu} $\mu$ is $\prec^H$-maximal and hence $\abs{\mu}$ is $\prec_{H_w}$-maximal. Let $\nu\in N(\mu)$ be a $\prec_{H,c}$-minimal element. Then $\nu$ satisfies $(i)$ and $(ii)$. Since $\pi_1(\nu)=\abs{\mu}$ it satisfies also $(iii)$.
 Property $(iv)$ clearly follows from the choice of $\nu$.
  \end{proof}

We continue by characterizing uniqueness of the representation provided by the previous theorem.

\begin{prop}
   Let $H\subset C(K,E)$ be a function space separating points and containing constants.
\begin{enumerate}[$(a)$]
    \item The following assertions are equivalent:
    \begin{enumerate}[$(i)$]
        \item For any $\varphi\in H^*$ there is a unique $\nu\in M_+(K\times B_{E^*})$ with properties $(i)-(iv)$ from Theorem~\ref{T:reprez-soucin}.
        \item $H$ is functionally vector simplicial and $B_{E^*}$ is a simplexoid.
    \end{enumerate}
    \item The following assertions are equivalent:
    \begin{enumerate}[$(i)$]
        \item For any $t\in K$ and $x^*\in E^*$ there is a unique $\nu\in M_+(K\times B_{E^*})$ corresponding to $\varphi=x^*\circ\phi_H(t)$ with properties $(i)-(iv)$ from Theorem~\ref{T:reprez-soucin}.
        \item $H$ is vector simplicial and $B_{E^*}$ is a simplexoid.
    \end{enumerate}
\end{enumerate}
\end{prop}

\begin{proof}
   Implications $(ii)\implies (i)$ in both cases follow from definitions, Proposition~\ref{P:maximalitamu} and Theorem~\ref{T:simplexoid} (implication $(1)\implies(3)$).

   Let us look at $(i)\implies (ii)$ in both cases. Firstly, uniqueness of $\nu$ implies uniqueness of $T^*\nu$, thus we get functional vector simpliciality in $(a)$ and vector simpliciality in $(b)$ (using definitions and Proposition~\ref{P:maximalitamu}).
   By the proof of Theorem~\ref{T:simplexoid}, implication $(3)\implies(1)$ we deduce that $B_{E^*}$ is a simplexoid.
\end{proof}

We finish this section by formulating the above two results for the special case when $K$ is metrizable and $E$ is separable. They follow directly from the previous two results using Proposition~\ref{P:hranice-max}$(b)$.

\begin{prop} \label{P:reprez-souc-separ}
    Assume that $K$ is metrizable and $E$ is separable. 
      Let $H\subset C(K,E)$ be a function space separating points such that $H_w$ contains constants. Let $\varphi\in H^*$ be given. Then there is a measure $\nu\in M_+(K\times B_{E^*})$ satisfying the following conditions.
    \begin{enumerate}[$(i)$]
        \item $\norm{\nu}=\norm{\varphi}$;
        \item $\varphi(\f)=\int x^*(\f(t))\di\nu(t,x^*)$ for each $\f\in H$;
        \item $\nu$ is carried by $\Ch_HK\times \ext B_{E^*}$.
    \end{enumerate}
\end{prop}

\begin{prop}
  Assume that $K$ is metrizable and $E$ is separable. 
   Let $H\subset C(K,E)$ be a function space separating points and containing constants.
\begin{enumerate}[$(a)$]
    \item The following assertions are equivalent:
    \begin{enumerate}[$(i)$]
        \item For any $\varphi\in H^*$ there is a unique $\nu\in M_+(K\times B_{E^*})$ with properties $(i)-(iii)$ from Proposition~\ref{P:reprez-souc-separ}.
        \item $H$ is functionally vector simplicial and $B_{E^*}$ is a simplexoid.
    \end{enumerate}
    \item The following assertions are equivalent:
    \begin{enumerate}[$(i)$]
        \item For any $t\in K$ and $x^*\in E^*$ there is a unique $\nu\in M_+(K\times B_{E^*})$ corresponding to $\varphi=x^*\circ\phi_H(t)$ with properties  $(i)-(iii)$ from Proposition~\ref{P:reprez-souc-separ}.
        \item $H$ is vector simplicial and $B_{E^*}$ is a simplexoid.
    \end{enumerate}
\end{enumerate}
\end{prop}

\section{Overview of the results}

It turns out that the theory of integral representation for vector-valued function spaces $H\subset C(K,E)$ has simultaneously similarities and differences to the scalar theory. One of the major differences is the fact that there are two different natural directions of representation theorems (see Section~\ref{sec:repmeas}). The first possibility is representing functionals from $H^*$ using $E^*$-valued vector measures on $K$. This approach is used for example in \cite{saab-tal} and we address it in Section~\ref{sec:reprez-functionals}. In this case any functional admits a representing measure of the same norm by a combination of Hanh-Banach and Singer's theorems.  

The second possibility consists in representation of operators from $L(H,E)$ by scalar measures on $K$ with the use of Bochner integral. We address it in Section~\ref{sec:reprez-oper}. In this case not all operators may be represented, so we introduce \emph{representable} operators. Not all operators are representable (Example~\ref{ex:nonrepresentable}) and, moreover, the smallest norm of a representing measure may be larger than the operator norm (Example~\ref{ex:mensinorma}). Further, representable operators are in a natural isometric correspondence with $H_w^*$ (see Proposition~\ref{P:Upsilon}), where $H_w$ is the scalar function space generated by $x^*\circ\h$ for $\h\in H$ and $x^*\in E^*$.

On the other hand, there are unique natural notions of Choquet boundary and of boundary measures and they coincide with the notions used (explicitly or implicitly in the literature). There are several equivalent descriptions (see Section~\ref{sec:hranice a miry}) which work in full generality (i.e., without assuming any properties of $H$ like separation of points or containment of constants). In both variants of representation we have general representation theorems by boundary measures (see Section~\ref{sec:reprez}). These theorems are essentially known and may be proved using a result of \cite{saab-canad}. The two representation theorems lead to two different way of looking at uniqueness, i.e. to different variants of simpliciality -- vector simpliciality and weak simpliciality (see Section~\ref{sec:reprez}). They appear to have a quite different nature.

Weak simpliciality, similiarly as simpliciality of scalar function spaces, has two degrees.
The first one (called weak simpliciality) is defined by uniqueness of boundary representation of evaluation operators, the second one (called functional weak simpliciality) is defined by uniqueness of boundary representation of all representable operators. It turns out that these properties are equivalent to the corresponding properties of the scalar function space $H_w$ (Observation~\ref{obs:weak simpl}). In particular, these properties are not affected by a renorming of $E$. Further, the space $A_c^w(H)$ defined in Section~\ref{sec:H-aff} is a complete analogy of the classical space of continuous $H$-affine functions known from the scalar theory. In particular, Choquet boundaries of $H$ and $A_c^w(H)$ coincide, $H$-boundary and $A_c^w(H)$-boundary measures coincide (Proposition~\ref{P:AcwHboundary}) and $H$ is weakly simplicial if and only if $A_c^w(H)$ is weakly simplicial (Proposition~\ref{P:weaksimp}).

Vector simpliciality has a quite different nature. It has also two degrees -- functional vector simpliciality defined by uniqueness of boundary representation of any functional from $H^*$ and vector simpliciality defined by uniqueness of boundary representation of special functionals, namely
compositions of functionals on $E$ with evaluation operators. Also in this case we may define an analogue of continuous $H$-affine functions, this time denoted by $A_c^v(H)$. But similarity with the scalar case is rather limited. The Choquet boundary of $A_c^v(H)$ may be strictly larger than that of $H$, the vector simpliciality of $H$ does not imply that of $A_c^v(H)$ (Example~\ref{ex:nezachovani}).
Moreover, both vector simpliciality and the space $A_c^v(H)$ may be affected by renorming $E$ (see Examples~\ref{ex:nezachovani},~\ref{ex:renorm1} and~\ref{ex:renorm2}). Further, spaces $A_c^v(H)$ and $A_c^w(H)$ may be mutually incomparable (Example~\ref{ex:nezachovani}) and weak simpliciality and vector simpliciality are mutually incomparable as well (see Examples~\ref{ex:nezachovani} and~\ref{ex:vsnews}).

The situation is much easier if $H$ contains constants (see Section~\ref{sec:skonstantami}). Then $A_c^v(H)\subset A_c^w(H)$ (Lemma~\ref{L:skonstantami}) and vector simpliciality admits a nice characterization witnessing that it is strictly stronger than weak simpliciality (Theorem~\ref{T:ruznesimpl} and Example~\ref{ex:wsnevs-const}). However, even in this case vector simpliciality and $A_c^v(H)$ are affected by renorming of $E$ (Example~\ref{ex:renorm-const}). A distinguished role is played by equivalent norms with strictly convex dual (Proposition~\ref{P:renorm-AcvH}).

If $H_w$ contains constants (but $H$ not necessarily), weak simpliciality has nice characterizations, in particular, it is equivalent to vector simpliciality of $A_c^w(H)$ (Corollary~\ref{cor:weaKsimpl}).
If additionally $E$ is an $L^1$-predual, then $H$ is weakly simplicial if and only if $A_c^w(H)$ is an $L^1$-predual (Theorem~\ref{T:L1pred}) which extends a known result from the scalar setting. Further, in this case there are mutually related natural order structures on $M_+(K)$, $M_+(K\times B_{E^*})$ and $M(K,E^*)$ (Section~\ref{sec:ordering}) which provide a characterization of $H$-boundary vector measures by certain maximality conditions (Proposition~\ref{P:maximalitamu}).
Finally, using the method of disintegration and some results from \cite{batty-vector} and \cite{transference-studia} we formulate and prove a finer representation theorem using positive measures on $K\times B_{E^*}$ and characterize their uniqueness (Section~\ref{ss:jinarep}).

\bibliographystyle{acm}
\bibliography{vector-obecH}

\begin{thebibliography}{10}

\bibitem{alfsen}
{\sc Alfsen, E.}
\newblock {\em Compact convex sets and boundary integrals}.
\newblock Springer-Verlag, New York, 1971.
\newblock Ergebnisse der Mathematik und ihrer Grenzgebiete, Band 57.

\bibitem{batty-vector}
{\sc Batty, C. J.~K.}
\newblock Vector-valued {C}hoquet theory and transference of boundary measures.
\newblock {\em Proc. London Math. Soc. (3) 60}, 3 (1990), 530--548.

\bibitem{effros}
{\sc Effros, E.}
\newblock On a class of complex {B}anach spaces.
\newblock {\em Illinois J. Math. 18\/} (1974), 48--59.

\bibitem{fremlin1}
{\sc Fremlin, D.~H.}
\newblock {\em Measure theory. {V}ol. 1}.
\newblock Torres Fremlin, Colchester, 2004.
\newblock The irreducible minimum, Corrected third printing of the 2000
  original.

\bibitem{fremlin4}
{\sc Fremlin, D.~H.}
\newblock {\em Measure theory. {V}ol. 4}.
\newblock Torres Fremlin, Colchester, 2006.
\newblock Topological measure spaces. Part I, II, Corrected second printing of
  the 2003 original.

\bibitem{fuhr-phelps}
{\sc Fuhr, R., and Phelps, R.~R.}
\newblock Uniqueness of complex representing measures on the {C}hoquet
  boundary.
\newblock {\em J. Functional Analysis 14\/} (1973), 1--27.

\bibitem{hensgen}
{\sc Hensgen, W.}
\newblock A simple proof of {S}inger's representation theorem.
\newblock {\em Proc. Amer. Math. Soc. 124}, 10 (1996), 3211--3212.

\bibitem{hirsberg72}
{\sc Hirsberg, B.}
\newblock Repr\'{e}sentations int\'{e}grales des formes lin\'{e}aires
  complexes.
\newblock {\em C. R. Acad. Sci. Paris S\'{e}r. A-B 274\/} (1972), A1222--A1224.

\bibitem{hustad71}
{\sc Hustad, O.}
\newblock A norm preserving complex {C}hoquet theorem.
\newblock {\em Math. Scand. 29\/} (1971), 272--278.

\bibitem{Jel}
{\sc Jellett, F.}
\newblock On affine extensions of continuous functions defined on the extreme
  boundary of a {C}hoquet simplex.
\newblock {\em Quart. J. Math. Oxford Ser. (1) 36}, 141 (1985), 71--73.

\bibitem{transference-studia}
{\sc Kalenda, O. F.~K., and Spurn{\'y}, J.}
\newblock Transference of measures via disintegration.
\newblock to appear in Studia Math., arXiv:2405.04202.

\bibitem{bezkonstant}
{\sc Kalenda, O. F.~K., and Spurn\'y, J.}
\newblock On simpliciality of function spaces not containing constants.
\newblock {\em J. Funct. Anal. 288}, 4 (2025), Paper No. 110756.

\bibitem{lacey}
{\sc Lacey, H.}
\newblock {\em The isometric theory of classical {B}anach spaces}.
\newblock Springer-Verlag, New York, 1974.
\newblock Die Grundlehren der mathematischen Wissenschaften, Band 208.

\bibitem{lima1976complex}
{\sc Lima, A.}
\newblock Complex {B}anach spaces whose duals are {$L_1$}-spaces.
\newblock {\em Israel Journal of Mathematics 24}, 1 (1976), 59--72.

\bibitem{lmns}
{\sc Luke{\v{s}}, J., Mal{\'y}, J., Netuka, I., and Spurn{\'y}, J.}
\newblock {\em Integral representation theory}, vol.~35 of {\em de Gruyter
  Studies in Mathematics}.
\newblock Walter de Gruyter \& Co., Berlin, 2010.
\newblock Applications to convexity, Banach spaces and potential theory.

\bibitem{phelps-complex}
{\sc Phelps, R.~R.}
\newblock The {C}hoquet representation in the complex case.
\newblock {\em Bull. Amer. Math. Soc. 83}, 3 (1977), 299--312.

\bibitem{rao-tensor}
{\sc Rao, T. S. S. R.~K., Roy, A.~K., and Sundaresan, K.}
\newblock Intersection properties of balls in tensor products of some {B}anach
  spaces.
\newblock {\em Math. Scand. 65}, 1 (1989), 103--118.

\bibitem{rondos-spurny}
{\sc Rondo\v{s}, J., and Spurn\'{y}, J.}
\newblock The {D}irichlet problem on compact convex sets.
\newblock {\em J. Funct. Anal. 281}, 12 (2021), Paper No. 109251, 20.

\bibitem{ryan-tensor}
{\sc Ryan, R.~A.}
\newblock {\em Introduction to tensor products of {B}anach spaces}.
\newblock Springer Monographs in Mathematics. Springer-Verlag London, Ltd.,
  London, 2002.

\bibitem{saab-cr}
{\sc Saab, P.}
\newblock La repr\'esentation int\'egrale de {C}hoquet dans le cas vectoriel.
\newblock {\em C. R. Acad. Sci. Paris S\'er. A-B 286}, 14 (1978), A621--A624.

\bibitem{saab-aeq}
{\sc Saab, P.}
\newblock The {C}hoquet integral representation in the affine vector-valued
  case.
\newblock {\em Aequationes Math. 20}, 2-3 (1980), 252--262.

\bibitem{saab-canad}
{\sc Saab, P.}
\newblock Integral representation by boundary vector measures.
\newblock {\em Canad. Math. Bull. 25}, 2 (1982), 164--168.

\bibitem{saab-tal}
{\sc Saab, P., and Talagrand, M.}
\newblock A {C}hoquet theorem for general subspaces of vector-valued functions.
\newblock {\em Math. Proc. Cambridge Philos. Soc. 98}, 2 (1985), 323--326.

\bibitem{spurny-archiv}
{\sc Spurn\'y, J.~r.}
\newblock The abstract {D}irichlet problem for continuous vector-valued
  functions.
\newblock {\em Arch. Math. (Basel) 108}, 5 (2017), 473--483.

\end{thebibliography}

\end{document}